\documentclass[12pt,reqno]{amsart}
\usepackage[numbers,sort&compress]{natbib}
\usepackage{amsfonts,bbm}
\usepackage{amssymb,color}
\usepackage{amssymb}
\usepackage{fancyhdr}
\usepackage[titletoc]{appendix}
\usepackage{enumitem}
\usepackage{amsgen}
\usepackage{amscd}
\usepackage{amsmath}
\usepackage{mathrsfs}
\usepackage{cases}

\usepackage[colorlinks=true,backref]{hyperref}
\hypersetup{urlcolor=blue, citecolor=blue, linkcolor=black}

\usepackage[margin=3cm, a4paper]{geometry}

\newtheorem{thm}{Theorem}[section]
\newtheorem{cor}[thm]{Corollary}
\newtheorem{lem}[thm]{Lemma}
\newtheorem{prop}[thm]{Proposition}
\newtheorem{defn}[thm]{ \bf{Definition}}
\theoremstyle{remark}
\newtheorem{remark}[thm]{Remark}

\newcommand{\EQ}[1]{\begin{align*}\begin{split} #1 \end{split}\end{align*}}
\newcommand{\EQn}[1]{\begin{align}\begin{split} #1 \end{split}\end{align}}
\newcommand{\EQnn}[1]{\begin{align} #1 \end{align}}
\newcommand{\EQnnsub}[1]{\begin{subequations}\begin{align} #1 \end{align}\end{subequations}}

\newcommand{\enu}[1]{\begin{enumerate} #1 \end{enumerate}}

\newcommand{\Del}[1]{}

\def\norm#1{\left\|#1\right\|}
\def\normo#1{\|#1\|}
\def\normb#1{\big\|#1\big\|}

\def\abs#1{\left|#1\right|}

\def\absb#1{\big|#1\big|}

\def\brk#1{\left(#1\right)}
\def\brko#1{(#1)}
\def\brkb#1{\big(#1\big)}

\def\fbrk#1{\left\lbrace#1\right\rbrace}
\def\fbrko#1{\lbrace#1\rbrace}
\def\fbrkb#1{\big\lbrace#1\big\rbrace}

\def\jb#1{\langle#1\rangle}

\def\wt#1{\widetilde{#1}}
\def\wh#1{\widehat{#1}}
\def\wb#1{\overline{#1}}

\def\pd{\partial}

\newcommand{\ra}{{\rightarrow}}
\newcommand{\hra}{{\hookrightarrow}}

\def\loe{\leqslant}
\def\goe{\geqslant}
\def\lsm{\lesssim}
\def\gsm{\gtrsim}

\newcommand{\N}{{\mathbb N}}
\newcommand{\T}{{\mathbb T}}
\newcommand{\R}{{\mathbb R}}
\newcommand{\C}{{\mathbb C}}
\newcommand{\Z}{{\mathbb Z}}
\newcommand{\PP}{{\mathbb P}}

\newcommand{\F}{{\mathcal{F}}}
\newcommand{\A}{{\mathcal{A}}}

\newcommand{\M}{{\mathcal{M}}}

\newcommand{\Sch}{{\mathcal{S}}}

\newcommand{\ZZ}{{\mathcal{Z}}}
\newcommand{\dd}{{\mathrm{d}}}

\newcommand{\supp}{{\mbox{supp}\ }}

\def\dx{\mathrm{\ d} x}
\def\dy{\mathrm{\ d} y}
\def\ds{\mathrm{\ d} s}

\newcommand{\re}{{\mathrm{Re}}}
\newcommand{\im}{{\mathrm{Im}}}

\def\ep{\varepsilon}
\def\al{\alpha}

\def\Om{\Omega}
\def\om{\omega}
\def\ph{\varphi}
\def\th{\theta}
\def\si{\sigma}
\def\de{\delta}
\def\De{\Delta}

\def\la{\lambda}
\def\ga{\gamma}

\newcommand{\I}{\infty}
\def\rev#1{\frac{1}{#1}}
\def\half#1{\frac{#1}{2}}

\def\bx{\square}

\numberwithin{equation}{section}

\begin{document}
\title[3D cubic NLS]{Almost sure well-posedness and scattering of the 3D cubic nonlinear Schr\"odinger equation}

\author{Jia Shen}
\address{(J. Shen) Center for Applied Mathematics\\
Tianjin University\\
Tianjin 300072, China}
\email{shenjia@tju.edu.cn}

\author{Avy Soffer}
\address{(A. Soffer) Rutgers University\\
	Department of Mathematics\\
	110 Frelinghuysen Rd.\\
	Piscataway, NJ, 08854, USA\\}
\email{soffer@math.rutgers.edu}
\thanks{}

\author{Yifei Wu}
\address{(Y. Wu) Center for Applied Mathematics\\
	Tianjin University\\
	Tianjin 300072, China}
\email{yerfmath@gmail.com}
\thanks{}


\date{}

\begin{abstract}\noindent
We study the random data problem for 3D, defocusing, cubic nonlinear Schr\"odinger equation in $H_x^s(\R^3)$ with $s<\half 1$. First, we prove that the almost sure local well-posedness holds when $\frac{1}{6}\loe s<\half 1$ in the sense that the Duhamel term belongs to $H_x^{1/2}(\R^3)$.

Furthermore, we prove that the global well-posedness and scattering hold for randomized, radial, large data $f\in H_x^{s}(\R^3)$ when $\frac{17}{40}< s<\half 1$. The key ingredient is to control the energy increment including the terms where the first order derivative acts on the linear flow, and our argument can lower down the order of derivative more than $\frac12$. To our best knowledge, this is the first almost sure large data global result for this model.
\end{abstract}

\maketitle

\tableofcontents

\section{Introduction}

In this paper, we consider the nonlinear Schr\"odinger equations (NLS):
\EQn{
	\label{eq:nls}
	\left\{ \aligned
	&i\pd_t u + \De u = \mu |u|^p u, \\
	& u(0,x) = u_0(x),
	\endaligned
	\right.
}
where $p>0$, $\mu=\pm1$, and
$u(t,x):\R\times\R^d\rightarrow \C$ is a complex-valued function. The positive sign ``$+$" in nonlinear term of \eqref{eq:nls} denotes defocusing source,   and the negative sign ``$-$" denotes the focusing one.

 The equation \eqref{eq:nls} has conserved mass
\EQn{\label{NLS:mass}
	M(u(t)) :=\int_{\R^d} \abs{u(t,x)}^2 \dx=M(u_0),
}
and energy
\EQn{\label{NLS:Energy}
	E(u(t)) :=\int_{\R^d} \half 1\abs{\nabla u(t,x)}^2 \dx + \mu \int_{\R^d} \rev{p+2} \abs{u(t,x)}^{p+2} \dx=E(u_0).
}

The class of solutions to equation (\ref{eq:nls}) is invariant under the scaling
\begin{equation}\label{eqs:scaling-alpha}
u(t,x)\to u_\lambda(t,x) = \lambda^{\frac{2}{p}} u(\lambda^2 t, \lambda x) \ \ {\rm for}\ \ \lambda>0,
\end{equation}
which maps the initial data as
\EQn{
u(0)\to u_{\lambda}(0):=\lambda^{\frac{2}{p}} u_0(\lambda x) \ \ {\rm for}\ \ \lambda>0.
}
Denote
$$
s_c=\frac d2-\frac{2}{p},
$$
then the scaling  leaves  $\dot{H}^{s_{c}}$ norm invariant, that is,
\begin{eqnarray*}
	\|u(0)\|_{\dot H^{s_{c}}}=\|u_{\lambda}(0)\|_{\dot H^{s_{c}}}.
\end{eqnarray*}
This gives the scaling critical exponent $s_c$. Let
$$
2^*=\I, \mbox{ when } d=1 \mbox{ or } d=2;  \quad
2^*=\frac4{d-2}, \mbox{ when }  d\goe 3.
$$
Therefore, according to the conservation law, the equation is called mass or $L_x^2$ \textit{critical} when $p=\frac 4d$, and energy or $\dot H_x^1$ \textit{critical} when $p=\frac4{d-2}$. Moreover, when $\frac 4d<p<2^*$, we say that the equation is \textit{inter-critical}.

Let us now take a brief overview on the well-posedness and scattering theory of NLS \eqref{eq:nls}. Kato \cite{Kat87poincare} first proposed a method based on the contraction mapping and the Strichartz estimate, and obtained the local well-posedness when $p<\frac{4}{d-2}$ in $H_x^1$. See also \cite{Tsu87Funk} by Tsutsumi for the $L_x^2$-solution when $p<\frac4d$. Note that the above two results concerned the sub-critical cases when $s>s_c$. The local well-posedness in the critical sense was solved by Cazenave and Weissler, see \cite{CW89remarks}. Moreover, we refer the readers to Cazenave's textbook \cite{Caz03book} for more complete local results of NLS.

The global well-posedness and scattering are basic topics for the long time behaviour of NLS. Lin and Strauss \cite{LS78JFA} obtained the large data scattering for the 3D, defocusing, cubic NLS with decaying data. Their argument relied on the Morawetz estimate, which was first discovered by Morawetz \cite{Mor68London} for the Klein-Gordon equations. The global well-posedness and scattering in energy space were solved by Ginibre and Velo \cite{GV85JMPA} in the defocusing inter-critical cases for $d\goe 3$. In this paper, we mainly focus on the results in $L_x^2$-based Sobolev spaces, thus we do not mention the vast scattering theory for NLS with decaying data.

The main breakthrough of the energy critical NLS was owed to Bourgain \cite{Bou99JAMS}. He introduced the powerful induction-on-energy method and the localised Morawetz estimate to study the defocusing equations with radial data for $d=3,4$. Bourgain's method was then further exploited extensively: Nakanishi \cite{Nak99JFA} introduced a modified version of Morawetz estimate for low dimensions, and solved the energy scattering in the inter-critical cases for $d=1,2$; Bourgain's 3D result was extended to non-radial by Colliander, Keel, Staffilani, Takaoka, and Tao \cite{CKSTT08Annals}, based on a localised version of their interaction Morawetz estimate \cite{CKSTT04CPAM}. The results for defocusing energy critical NLS in higher dimensions were obtained by Ryckman and Visan \cite{RV07AJM,Vis07Duke}. Very few large data global results were available for supercritical models, however, see Li's recent groundbreaking work \cite{LiDuke} on a supercritical wave model.

For the focusing equations, Kenig and Merle \cite{KM06Invent} introduced the concentration compactness method to give a complete dynamical characterization below the energy of ground state, for the energy critical NLS in $d=3,4,5$ with radial data. Their study opened a way to study the scattering of focusing equations below the ground state. Then, Duyckaerts, Holmer, and Roudenko \cite{DHR08MRL,HR08CMP} gave the result for 3D, focusing, cubic NLS, which is a typical model in the inter-critical cases. For the non-radial focusing energy-critical NLS, Killip and Visan \cite{KV10AJM} solved the $d\goe 5$ case, and Dodson \cite{Dod19ASENS} solved the 4D case. The scattering of 3D, focusing, energy-critical NLS in the non-radial case remains open. 

The concentration compactness method also enlighten the development of mass critical NLS. Killip, Tao, Visan, and Zhang \cite{TVZ07Duke,KVZ08APDE,KTV09JEMS} studied the mass critical NLS in the radial case. Dodson then remove the radial assumption, and completely solved the global well-posedness and scattering of mass critical NLS in the defocusing case \cite{Dod12JAMS,Dod16Duke,Dod16AJM}, and in the focusing case below the mass of ground state \cite{Dod15Adv}.

Next, we focus on the well-posedness results of 3D, defocusing, cubic NLS, for which the critical regularity exponent $s_c=1/2$. We have learned that the equation is local well-posed in $\dot H_x^{1/2}$, while the global well-posedness and scattering hold in a smaller space $H_x^1$. A natural question is to ask the weakest space $X$ to guarantee the global well-posedness in $\dot H_x^{1/2} \cap X$. Bourgain \cite{Bou98JAM} used the high-low decomposition method (introduced in \cite{Bou98IMRN}) to give $X=H_x^s$ with $s>\frac{11}{13}$. The lower bound was then improved by ``I-method" gradually in \cite{CKSTT02MRL,CKSTT04CPAM,Su12MRL}, and so far, the best result is $s>\frac{5}{7}$. Under the radial assumption, Dodson \cite{Dod19CJM} showed that the result holds for almost critical space $s>\frac12$. 

Note that $X$ spaces in the above mentioned results are all $\dot H_x^{1/2}$ super-critical. Recently, Dodson \cite{Dodson20critical} gave a result in the critical space $X=\dot W_x^{11/7,7/6}$, based on the observation that linear solution becomes more regular with initial data in $L_x^p$ with $p<2$. Using this observation and the method in \cite{BDSW18sub-critical}, the authors \cite{SW20cubicNLS}  obtained that $X=\dot W^{s,1}$ for $s>\frac{12}{13}$, which is a sub-critical space with the order of derivative less than $1$. 

Currently, there is no result for the global well-posedness of 3D defocusing cubic NLS merely in $\dot H_x^{1/2}$ or $H_x^{1/2}$. Kenig and Merle \cite{KM10TranAMS} initiated another approach towards this problem. They proposed the concept of ``conditional scattering", namely the global well-posedness and scattering hold for the solution that is uniformly bounded in the critical space on the maximal existence interval. Generally for the inter-critical NLS, no global well-posed result is known in the critical space. See \cite{BDSW19inter-critical,Dodson20critical} for some related results. 

Now, we turn to the probability theory of NLS. Although it is  ill-posedness for NLS below the critical regularity due to the result of Christ, Colliander, and Tao \cite{CCT03illposed}, Bourgain \cite{Bou94CMP,Bou96CMP} first introduced a probabilistic method to study the well-posedness problem for periodic NLS for ``almost'' all the initial data in super-critical spaces, based on the Gibbs measure constructed by Lebowitz-Rose-Speer \cite{LRS88JSP}. The probabilistic well-posedness result for super-critical wave equations on compact manifolds was also studied by Burq and Tzvetkov \cite{BT08inventI,BT08inventII}. There have been extensive studies about such subject since then, and we refer the readers to \cite{BOP19note} for more complete overviews.

Next, we only review the study of random data theory for NLS on $\R^d$. There are several ways of randomization for the initial data. We recall the one relying on the unit-scale decomposition in frequency, which is named as the Wiener randomization (see \cite{ZF12JMFM}). Then, under the Wiener randomization, B\'enyi, Oh, and Pocovnicu \cite{BOP15TranB} studied the cubic NLS when $d\goe 3$. They proved the almost sure local well-posedness, small data scattering, and a ``conditional'' global well-posedness under some a priori hypothesis. Afterwards, the almost sure local results were improved by B\'enyi-Oh-Pocovnicu \cite{BOP19TranB} and Pocovnicu-Wang \cite{PW18lp}. The random data well-posedness for quintic NLS was studied by Brereton \cite{Bre19tunis}. Later, Oh, Okamoto, and Pocovnicu \cite{OOP19DCDS} studied the almost sure global well-posedness  for energy critical NLS on $d=5,6$. Here, we only focus on the NLS, and see L\"uhrmann-Mendelson \cite{LM14CPDE} for the first probabilistic result of non-linear wave equations on Euclidean space.

%Furthermore, the regularity requirement for almost sure local well-posedness can be improved, if one only requires the Duhamel term belonging to $L_x^p$ type spaces with $p>2$. This subject was studied by Pocovnicu and Wang \cite{PW18lp}. Taking 3D, cubic NLS for example, they proved that the random data local well-posedness holds merely in $L_x^2$.

The large data almost sure scattering was first obtained by Dodson, L\"uhrmann, and Mendelson \cite{DLM20AJM} in the context of 4D, defocusing, energy-critical, nonlinear wave equation with randomized radial data, using a double bootstrap argument combining the energy and Morawetz estimates. The result was extended by Bringmann to non-radial 4D case \cite{Bri18NLW}, and to radial 3D case \cite{Bri20APDE}. The related results on non-radial energy-critical nonlinear Klein-Gordon equations were studied by Chen and Wang \cite{CW20NLW}. The first almost sure scattering result for NLS was given by Killip, Murphy, and Visan \cite{KMV19CPDE}. They proved the result for 4D, defocusing, energy-critical case with almost all the radial initial data in $H_x^s$ for $\frac56<s<1$. This result was then improved to $\frac12<s<1$ by Dodson, L\"uhrmann, and Mendelson \cite{DLM19Adv}. 

We remark that the Wiener randomization is closely related to the modulation space introduced by Feichtinger \cite{Fei83modulation}. Such space has been applied to non-linear evolution equations before the development of Wiener randomization, dating back to the results of Wang, Zhao, Guo, and Hudzik \cite{WZG06JFA,WH07JDE}.

There are also other kinds of randomization for NLS on $\R^d$. Burq, Thomann, and Tzvetkov \cite{BTT13Fourier} constructed a Gibbs measure for NLS with harmonic potential, and proved almost sure $L^2$-scattering for 1D, defocusing NLS with $p\goe 5$, after changing the Schr\"odinger equations into the ones with harmonic oscillator potential by lens transform. Recently, Burq and Thomann \cite{BT20scattering} improved the result to all the short range exponents $p>3$. See also \cite{Lat20scattering} for higher dimensional extensions. 

In addition, Murphy \cite{Mur19ProAMS} introduced a new kind of randomization based on the physical space unit-scale decomposition, and studied the almost sure existence and uniqueness of wave operator for $L^2$ sub-critical NLS above the Strauss exponent. Then, Nakanishi and Yamamoto \cite{NY19MRL} extended the result below Strauss exponent, and applied the method on some quadratic Schr\"odinger models. We also mention that Bringmann's almost sure scattering results \cite{Bri20APDE,Bri18NLW} include other kinds of randomization for nonlinear wave equations on $\R^d$, involving the micro-local and the annuli decompositions of initial data.

To the best of our knowledge, the only study so far of global well-posedness and scattering for inter-critical NLS seems Burq, Thomann, and Tzvetkov's 1D $L^2$-scattering result \cite{BTT13Fourier}, based on the invariant measure for the Schr\"odinger equations with harmonic potential. Very recently, we learnt that Duerinckx \cite{Due21NLS} also studied the global well-posedness of cubic NLS adding a tiny dissipation with spatial inhomogeneous random initial data. In this paper, we intend to study a typical model of inter-critical NLS, namely the 3D, defocusing, cubic NLS under the Wiener randomization, at super-critical regularity. 
%For the randomized data, we improve the previous local well-posedness, and obtain the global well-posedness and scattering for radial data.

Before stating the main result, we give the definition of the randomization:
\begin{defn}[Wiener randomization]\label{defn:randomization}
Let $\wt \psi\in C_0^\I(\R^3)$ be a real-valued function such that $\wt \psi\goe 0$, $\wt \psi(-\xi)=\wt \psi(\xi)$ for all $\xi\in\R^3$ and 
\EQ{
	\wt \psi(\xi) =\left\{ \aligned
	&1\text{, when $\xi\in [-\frac12,\frac12]^3$,}\\
	&0\text{, when $\xi\notin [-1,1]^3$.}
	\endaligned
	\right.
}
Let
\EQ{
	\psi(\xi):=\frac{\wt \psi(\xi)}{\sum_{k\in\Z^3}\wt \psi(\xi-k)}.
}
Then, $\psi\in C_0^\I(\R^3)$ is a real-valued function, satisfying for all $\xi\in\R^3$, $0\loe \psi\loe 1$, $\supp \psi\subset[-1,1]^3$, $\psi(-\xi)=\psi(\xi)$, and $\sum_{k\in\Z^3}\psi(\xi-k)=1$.

For any $k\in\Z^3$, define $\psi_k(\xi)=\psi(\xi-k)$. Denote the Fourier transform by $\F$. Then, we define 
\EQ{
	\bx_kf=\F^{-1}\brkb{\psi_k\F f}.
}
Let $(\Om, \A, \PP)$ be a probability space. Let $\fbrk{g_k}_{k\in\Z^3}$ be a sequence of zero-mean, complex-valued Gaussian random variables on $\Om$, where the real and imaginary parts of $g_k$ are independent. Then, for any function $f$, we define its randomization $f^\om$ by
\EQn{\label{eq:randomization}
	f^\om=\sum_{k\in\Z^3} g_k(\om)\bx_k f.
}
\end{defn}

In the following, we use the statement ``almost every $\om\in\Om$, $PC(\om)$ holds'' to mean that 
$$
\PP\Big(\big\{\om\in \Om: PC(\omega)\> \mbox{ holds} \big\}\Big)=1.
$$

Now, we study the 3D, defocusing, cubic NLS with randomized initial data:
\EQn{
	\label{eq:nls-3D}
	\left\{ \aligned
	&i\pd_t u + \De u =  |u|^2 u, \\
	& u(0,x) = f^\om(x).
	\endaligned
	\right.
}
For this model under the probabilistic setting, the local well-posedness, small data scattering, and conditional global well-posedness results have been established before.

We first recall the local results for \eqref{eq:nls-3D}. B\'enyi, Oh, and Pocovnicu \cite{BOP19TranB} proved the local result with $f\in H_x^{s}$ when $\frac25\si<s<\frac12$ in the sense that Duhamel term belongs to $C(I;H_x^{\si})$ for any fixed $\frac12\loe \si\loe 1$. They also proved the improved local result when $\frac16<s<\frac12$ (except for the lower endpoint) by weakening the definition of local solution:
\EQ{
	u-z_1-z_3-\cdot\cdot\cdot-z_{2k-1}\in C(I;H_x^{1/2}),
}
where the function $z_k\in C(I;H_x^{s_k})$ is defined by iteration with some $s_k<\frac12$. Pocovnicu and Wang \cite{PW18lp} also proved the local result in $L_x^2$ with Duhamel term in $C(I;L_x^4)$. 

There are also global results of \eqref{eq:nls-3D}, either with small data restriction or with suitable a priori assumptions. B\'enyi, Oh, and Pocovnicu \cite{BOP15TranB} proved the almost sure small data global well-posedness and scattering for $\frac14<s<\frac12$. Furthermore, they \cite{BOP15TranB} also proved the random data global well-posedness when $\frac14<s<\frac12$ under two a priori assumptions: 
\begin{itemize}
	\item The Duhamel term is uniformly bounded in the critical space $H_x^{1/2}$ in the probabilistic setting.
	\item The 3D, defocusing, cubic NLS is globally well-posed with deterministic initial data in $H_x^{1/2}$.
\end{itemize}
Each of the above two a priori assumptions seems very difficult to verify. 

In this paper, we improve the previous local results  for \eqref{eq:nls-3D}. Moreover with the radial data, we prove the global well-posedness as well as scattering, without imposing any a priori assumption or size restriction, where the scattering result holds in the energy space.

\subsection{Almost sure local well-posedness}
The first main result in this paper concerns the almost sure local well-posedness. Previously for the random data local result, B\'enyi, Oh, and Pocovnicu \cite{BOP19TranB} introduced the higher order expansion method; Pocovnicu and Wang's argument \cite{PW18lp} is based on the dispersive inequality; Dodson, L\"uhrmann, and Mendelson \cite{DLM19Adv} used the high dimensional version of smoothing effect and maximal function estimates. In this paper, we give some simple new approaches combining the atom space method by Koch-Tataru \cite{KT05CPAM} and the variants of bilinear Strichartz estimate.
\begin{thm}[Local well-posedness]\label{thm:local}
Let $f\in H_x^s(\R^3)$. Then, for almost every $\om\in \Om$, it holds that:
\enu{
\item 
If $\frac16\loe s<\frac12$, then there exists $T>0$ and a solution $u$ of \eqref{eq:nls-3D} on $[0,T]$ such that
\EQ{
	u-e^{it\De}f^\om\in C([0,T];H_x^{\frac12}(\R^3)).
}
\item 
If $\frac13< s<\frac12$, then there exists $T>0$ and a solution $u$ of \eqref{eq:nls-3D} on $[0,T]$ such that
\EQ{
	u-e^{it\De}f^\om\in C([0,T];H_x^{1}(\R^3)).
}
}
\end{thm}

Our result  improves the local results in \cite{BOP19TranB}, where B\'enyi, Oh, and Pocovnicu \cite{BOP19TranB} proved the same results in Theorem \ref{thm:local} $(1)$ for $\frac15<s<\frac12$ and Theorem \ref{thm:local} $(2)$ for $\frac25<s<\frac12$.

The following are some remarks concerning the theorem. 
\begin{remark}
\enu{
\item
For the first result in Theorem \ref{thm:local}, the lower exponent $\frac16$ is optimal for the non-linear estimate of 
\EQ{
\int_0^t e^{i(t-s)\De}(|e^{is\De}f^\om|^2e^{is\De}f^\om) \ds
}
with $f\in H_x^{1/6}(\R^3)$. In fact, this optimality can be verified by taking the Knapp counter-example $f$ such that $\wh f(\xi)=\jb{k}^{-s}\psi(\xi-k)$.
%\item 
%We believe that the first result in Theorem \ref{thm:local} is optimal in the following sense. In fact, we need to control the term
%\EQ{
%\brkb{\sqrt{-\De}}^{\frac12}\brkb{|e^{it\De}f^\om|^2 e^{it\De}f^\om},
%}
%and there is at least $\frac16$-order derivative acting on each $f^\om$.
\item 
It seems very difficult to extend the local solution obtained in Theorem \ref{thm:local} $(1)$ to global directly. Therefore, we establish the local solution with higher regularity in Theorem \ref{thm:local} $(2)$. 
%As remarked before, the lower bound $\frac13$ seems also sharp in this case.
\item Our second result can be compared to Dodson, L\"uhrmann, and Mendelson's local result \cite{DLM19Adv} for 4D, cubic NLS, which is energy critical, since both results put the Duhamel term in $C([0,T];H_x^1(\R^4))$. They proved local well-posedness for $\frac13<s<1$, also except the endpoint exponent $\frac13$.
\item 
Apparently, the local results in Theorem \ref{thm:local} also hold for focusing equations.
}
\end{remark}
The proof of Theorem \ref{thm:local} $(2)$ is more difficult than the first local result. We postponed here to illustrate the main idea. It reduces to consider the term
\EQ{
	\nabla\brkb{|e^{it\De}f^\om|^2 e^{it\De}f^\om}
}
with $f$ merely in $H_x^{\frac13}$. The task is how to allocate the first order derivative to each $e^{it\De}f^\om$. However, the use of bilinear Strichartz estimate or local smoothing can only lower down semi-derivative. Then, we overcome the difficulty by following two tools:
\begin{itemize}
\item
We employ  the $U^p$-$V^p$ method introduced by Koch-Tataru \cite{KT05CPAM} to exploit the duality structure.
\item 
We also apply the bilinear Strichartz estimate in the form of
\EQn{\label{eq:bilinear}
	\norm{[e^{it\De}\phi] [e^{\pm it\De}\psi]}_{L_t^q L_x^r(\R\times\R^3)}
}
with $1\loe q,r\loe 2$, see Candy's result \cite{Can19MathAnn}. Particularly, the use of  \eqref{eq:bilinear} with $q<2$ and $r=2$ can reduce the loss of derivative, at the cost of lower time integration exponents.
\end{itemize} 
Let $v=e^{it\De}f^\om$, and it suffices to control
\EQ{
\int_0^T\int_{\R^3} g_{\rm hi} \nabla v_{\rm hi} v_{\rm low}^2\dx\dd t,
}
where $e^{-it\De}g\in V^2(\R;L_x^2)$, and ``hi'', ``low'' represent the size of frequency. Heuristically, by H\"older's inequality,
\EQ{
	\int_0^T\int_{\R^3} g_{\rm hi} \nabla v_{\rm hi} v_{\rm low}^2\dx\dd t\lsm \norm{g_{\rm hi} v_{\rm low}}_{L_t^{\frac32}L_x^2} \norm{\nabla v_{\rm hi}v_{\rm low}}_{L_t^1 L_x^2}^{\frac13} \norm{\nabla v_{\rm hi}}_{L_{t,x}^\I}^{\frac23} \norm{v_{\rm low}}_{L_t^\I L_x^2}^{\frac23}.
}
This can cut down the order of high-frequency derivative to $\frac13$ for $v_{\rm hi}$, with a total loss of low-frequency derivative of order $\frac23$, which can be assigned to each $v_{\rm low}$. 

Note that the above observation is sharp with respect to the regularity. Furthermore, we have a logarithmic loss of derivative when passing $g$ into $V_\De^2$ by interpolation. That is the main reason why we need $s>\frac13$ to acquire additional regularity for summation. %Of course in this case, we also need to make a $\ep$-perturbation, see \eqref{esti:local-h1-nablav-2low-2} below for the precise estimate.

Moreover, if we only requires 
\EQ{
	u-e^{it\De}f^\om\in C([0,T];H_x^{\si}(\R^3))
}
for $\frac12\loe\si<1$, the approach in above observation can provide enough additional regularity for summation. Thus, we expect that the argument works for the optimal lower endpoint, namely $\frac13\si\loe s<\frac12$, which clearly includes the result in Theorem \ref{thm:local} $(1)$. However, in this paper, we only consider two endpoint cases when $\si=\frac12$ or $\si=1$, and present two different methods, respectively. For Theorem \ref{thm:local} $(1)$, we provide another proof without exploiting the duality structure. In fact, there is only $\frac12$-order derivative acting on the nonlinear term, and we can transfer it simply using the bilinear Strichartz estimate.

Next, we also compare our local result to the torus case, where there exists the invariant Gibbs measure for NLS, see \cite{LRS88JSP,Bou94CMP,Bou96CMP,DNY19gibbs}. The common key ingredient is to update the regularity of the first iteration
\EQ{
u^{(1)}:=\int_0^t e^{i(t-s)\De}(|e^{is\De}f^\om|^2e^{is\De}f^\om) \ds
}
for rough data $f^\om\in H_x^s$. Theorem \ref{thm:local} indicates that in the Euclidean case, the first iteration $u^{(1)}$ has improved regularity $\si=3s$ when $s=\frac16$, and $\si=3s-$ when $s=\frac13$. However, in the periodic case when the initial data is given by
\EQ{
f^{\om}(x) = \sum_{n\in\Z^d}\frac{g_n(\om)}{\jb{n}^{s+d/2}} e^{in\cdot x},
}
the first iteration has improved regularity $\si=s+\frac12-$, which is higher than the Euclidean case when $s<\frac14$, and lower when $s\goe\frac14$. Inspired by this, the result in Theorem \ref{thm:local} $(2)$ seems better than the observation in periodic case, and there seems still some room to further improve Theorem \ref{thm:local} $(1)$. 
%the key is to obtain more smoothing effect in the Euclidean case for
%\EQ{
%\brkb{\sqrt{-\De}}^{\frac12}\brkb{|e^{it\De}P_Nf^\om|^2 e^{it\De}P_Nf^\om},
%}
%where the three inputs $P_Nf^\om\in H_x^s$ ($s<\frac16$) has the same frequency $|\xi|\sim N$.

\subsection{Almost sure scattering}
Now we turn to our second main result for the global well-posedness and scattering: 
\begin{thm}[Global well-posedness and scattering]\label{thm:global}
Let $\frac{17}{40}< s\loe\frac12$ and $f\in H_x^s(\R^3)$ be radial. Then, for almost every $\om\in \Om$, there exists a solution $u$ of \eqref{eq:nls-3D} on $\R$ such that
\EQ{
	u-e^{it\De}f^\om\in C(\R;H_x^{1}(\R^3)). 
}
Moreover, the solution $u$ scatters, in the sense that there exist $u_\pm\in H_x^1$ such that
\EQ{
\lim_{t\ra\pm\I}\norm{u-e^{it\De}f^\om-e^{it\De}u_\pm}_{H_x^1} =0.
}
\end{thm}

%Compared with the previous results \cite{BOP15TranB, Due21NLS}, we do not impose 
The most significant point of this result is that we are able to control the energy increment containing the term $\nabla e^{it\De}f^\om$, under the assumption that $f$ merely belongs to $H_x^s$ with some $s<\frac12$. 

% Very recently, we learnt that Camps \cite{Cam21NLS} also obtained a similar result to Theorem \ref{thm:global} independently, using a different argument.

Comparing to the energy-critical results in \cite{KMV19CPDE,DLM19Adv}, for 3D, defocusing, cubic NLS, it is easier to derive space-time estimates, since the Morawetz type estimates are energy sub-critical. On the other hand, however, this problem seems more difficult, in the sense that we need to reduce the order of derivative more than $\frac12$ for $\nabla e^{it\De}f^\om$, while the current results for energy-critical NLS lower down at most $\frac12$-order derivative, in the view of local smoothing effect.

Our method is different from the recent results on the almost sure scattering of nonlinear Schr\"odinger and wave equations \cite{DLM19Adv,DLM20AJM,KMV19CPDE,Bri18NLW,Bri20APDE}. To establish the almost energy conservation of $u-e^{it\De}f^\om$, we make a high-low frequency decomposition of the initial data, and keep track of the explicit increase of energy bound. Then, we implement a bootstrap argument for the energy, building upon a perturbed interaction Morawetz estimate, various nonlinear estimates and the bilinear Strichartz estimate.

After the submission of this paper, we also learned that Camps \cite{Cam21NLS} obtained an almost sure scattering result independently by a different approach.
 
Lastly, we remark that the lower bound of regularity $\frac{17}{40}$ is not sharp. Here, we do not achieve this optimality, and only give a well-presented result. However, it is of great interest to improve the regularity's lower bound down to $\frac13$, or even $\frac16$.

%\noindent {\bf Sketch the proof of Theorem \ref{thm:global}.}

\subsubsection{Sketch the proof of Theorem \ref{thm:global}}

The main ingredient of the proof is summarized as follows. 

%$\bullet$ {\bf Probability space decomposition.} 
%In probabilistic setting, we only have the boundedness in the almost every sense. For example, $f\in H_x^s(\R^3)$, then we only have that $f^\om \in H_x^s(\R^3)$ a.e. $\omega\in \Omega$, and $\|f^\om\|_{H^s}\npreceq \|f\|_{H^s}$.
%In order to quantify the size of the initial value in $H^s$ and some space-time norms under the linear flow, we decompose the Probability space $\Omega$ by setting 
%$$
%\Omega_M=\big\{\omega\in\Om: \|f^\om \|_{H^s}+\|e^{it\Delta}f^\om\|_{Y(\R)}\loe M,\  N_0^{s-1}\|P_{\goe N_0}f^\om \|_{H^s}\loe M\big\}.
%$$
%Here $\|\cdot\|_{Y(\R)}$ is some required space-time norm, and $N_0\in 2^{\N}$ depends only on $M$. Then it follows from the Borel-Cantelli Lemma that 
%$$
%\PP\big(\cup_{M\goe 1}  \Omega_M\big)=1.
%$$

$\bullet$ {\bf High-low frequency decomposition in the probabilistic setting.}

In probabilistic setting, we only have the boundedness in the almost every sense. 
Roughly speaking, in order to quantify the size of energy, we decompose the probability space $\Omega$ by setting 
\EQ{
\wt \Omega_M=\big\{\omega\in\Om: \|f^\om \|_{H^s} + N_0^{s-1}\|P_{\loe N_0}f^\om \|_{H^1}+\|e^{it\Delta}f^\om\|_{\wt Y^s(\R)}\loe M\norm{f}_{H^s}\big\},
}
where the $\wt Y^s$-norm is some required space-time norm defined by \eqref{defn:ys-norm} and \eqref{defn:ys-tilde-norm} below, and $N_0\in 2^{\N}$ depends only on $M$ and $\norm{f}_{H^s}$. See \eqref{defn:omega-M} for the precise definition of $\wt\Om_M$. Then it follows from the Borel-Cantelli Lemma that 
$$
\PP\big(\cup_{M\goe 1}  \wt \Omega_M\big)=1.
$$

According to the decomposition above, we may consider $\omega\in \wt\Omega_M$ for each $M$ separately. Now, we give the high-low frequency decomposition $v=e^{it\De}P_{\goe N_0}f^\om$ and $w=u-v$. Then, for any $\omega \in\wt \Omega_M$, there exists a constant $C(M)>0$ such that 
\EQ{
E(w(0))\loe C(M,\norm{f}_{H^s})N_0^{2(1-s)}.
}

The application of Bourgain's high-low decomposition method \cite{Bou98IMRN} to random Cauchy problem was first made by Colliander and Oh \cite{CO12duke} for 1D NLS on $\T$. However, in this paper, we do not intend to carry out Bourgain's iteration procedure. We only make the decomposition in order for two benefits:
\begin{itemize}
\item[(1)] 
$\wh v$ is supported on $\fbrk{|\xi|\gsm N_0}$.
\item[(2)] 
We can explicitly keep track of the energy increment of $N_0$.
\end{itemize} 

$\bullet$ {\bf Strichartz estimates with $\frac12-$derivative gain.}

Note that $v$ is not radial anymore under the Wiener randomization. However, due to the pioneering works  \cite{DLM20AJM,DLM19Adv}, we can prove that  for the radial $f$, 
\EQn{\label{eq:intro-super-ctical}
\normb{\jb{\nabla}^{s+\frac12}v}_{L_t^2 L_x^\I}<\I,
}
which is followed by combining a ``radialish'' Sobolev inequality for the square function and the local smoothing estimate. Note that the estimate \eqref{eq:intro-super-ctical} gains $\frac12$-order derivative.

$\bullet$ {\bf Global  space-time bound for the nonlinear solution.}

From the perturbed interaction Morawetz estimates, we  can derive the bound of 
$$
\norm{w}_{L_{t,x}^4}\lsm N_0^{\frac14(1-s)},
$$
which is  $H^\frac14$-critical, under the a priori hypothesis of $H^1$-bound. The high-low frequency decomposition also plays a crucial role for controlling the remainder. Using this estimate, we can obtain the following $L^2$-critical estimate:
\EQ{
	\brkb{\sum_{N\in2^\N} \norm{P_N w}_{U_\De^2(L_x^2)}^2}^{1/2}\lsm N_0^{\frac56(1-s)+}.
}However, these are far from sufficient for the estimates of energy bound.

Then, an observation is that combining the above $L_{t,x}^4$-estimate and the integral equation, for any $L^2$-admissible $(q,r)$, we can further control $\dot H^\frac12$-critical space-time norm:
$$
\brkb{\sum_{N\in2^\N} N \norm{w_N}_{L_t^{q}L_x^{r}}^2}^{1/2}\lsm N_0^{\frac32(1-s)}.
$$
Keeping in mind that the equation is $\dot H^\frac12$-critical, the space-time estimate under the same scaling plays an important role throughout the whole argument. %\red{More importantly, this estimate is more flexible because we can take any $q\goe2$.}

Furthermore, applying the above $\dot H^\frac12$-critical estimates, we can update the scaling up to $\dot H_x^1$:
$$
\brkb{\sum_{N\in2^\N} N^{2}\norm{P_N w}_{U_\De^2(L_x^2)}^2}^{1/2}\lsm N_0^{3(1-s)}.%\text{, for any }l\in (\frac12,1].
$$
For this purpose, we also need to use  the maximal function techniques to deal with some critical cases. In the above argument, the use of $U_\De^2$-space has three advantages: we can transfer the derivative by duality formula, the $U_\De^2$-space can control the bilinear Strichartz estimate, and it allows estimates on any long-time interval.
%
%\begin{itemize}
%\item Perturbed interaction Morawetz: we can derive 
%\EQ{
%\norm{w}_{L_{t,x}^4}^4\lsm N_0^{1-s} + N_0^{2(1-s)}\norm{v}_{L_t^2 L_x^\I}^2.
%}
%Thanks to the high frequency property of $v$, we can bound the perturbed term by \eqref{eq:intro-super-ctical}, and obtain $\norm{w}_{L_{t,x}^4}\lsm N_0^{\frac14(1-s)}$. This is the main reason why we make the high-low frequency decomposition.
%\item $\dot H^{\frac12}$-critical space-time estimates: this kind of norms can be bounded by $N_0^{\frac32(1-s)}$, including $L_{t,x}^5$, $L_t^4 L_x^6$ and $L_t^{2+} L_x^{\I-}$.
%\item $\dot H^{l}$-critical space-time estimates with $l>\frac12$: this kind of norms can be bounded by $N_0^{(s+\frac72)(1-s)}$, which appear when we use the bilinear Strichartz estimate. However, this will increase the energy bound dramatically.
%\end{itemize}

$\bullet$ {\bf Energy bound.}

The main goal is to prove
\EQ{
	\sup_{t\in \R} E(w(t))\lsm_M N_0^{2(1-s)}.
}
It suffices to prove the bootstrap inequality
\EQ{
	\sup_{t\in I} \int_0^t \frac{\dd}{\dd s}E(w(s)) \dd s\lsm_M N_0^{-\al} N_0^{2(1-s)},
}
for some $\al>0$ under the assumption $\sup_{t\in I} E(w(t))\lsm_M N_0^{2(1-s)}$. Under this bootstrap hypothesis, we can also give the precise increase of $N_0$ for the various global space-time estimates on $I$ obtained in the previous step, which are very useful for the control of energy increment.
%Let $M>1$ and take some large $N_0=N_0(M)$ that does not depend on $\om$. We decompose $v=e^{it\De}P_{\goe N_0}f^\om$ and $w=u-v$. Then, there exists $\wt \Om_M\subset \Om$ with $\PP(\wt \Om_M^c)\lsm e^{-CM^2}$ such that $\wh v$ is supported on $\fbrk{|\xi|\gsm N_0}$, and  for any $\om\in \wt \Om_M$,
%\EQ{
%E(w_0)\loe C(M) N_0^{2(1-s)}.
%}
%We are going to solve  \eqref{eq:nls-3D} on each $\wt \Om_M$, where we can explicitly keep track of the increment of $N_0$, and then finish the whole proof by taking the set $\cup_M\wt \Om_M$. We refer the reader to proof of Theorem \ref{thm:global} in Section \ref{sec:reduction-global} for the concrete argument. 
Now, the main term in the energy estimate is 
\EQn{\label{eq:intro-energy-increment}
\absb{\int_I \int_{\R^3} \nabla w_{\rm hi} \nabla v_{\rm hi} w_{\rm low}^2 \dx \dd t},
}
where ``hi'' and ``low'' represent the size of frequency. 

We remark that the Morawetz estimate of the form $\iint \frac{|w|^4}{|x|}\dx\dd t$ plays an important role in the former energy-critical results \cite{DLM19Adv,DLM20AJM,KMV19CPDE}, but it is not sufficient for the 3D cubic case. First, the Morawetz estimate cannot yield the global space-time bounds, as in the previous step. That is the reason why the use of interaction Morawetz estimate is necessary. Second, using their method, \eqref{eq:intro-energy-increment} can be controlled by 
%can be rewritten as
%\EQ{
%\absb{\int_I \int_{\R^3} \nabla w_{\rm hi}\> \sqrt{|x|}\nabla v_{\rm hi}\> \frac{w_{\rm low}^2}{\sqrt{|x|}} \dx \dt},
%}
%and thus controlled by 
$$
\norm{\nabla w_{\rm hi}}_{L_t^\I L_x^2} \normb{\sqrt{|x|}\nabla v_{\rm hi}}_{L_t^2 L_x^\I}\Big( \int_I \int_{\R^3} \frac{|w_{\rm low}|^4}{|x|}\,\dx\dd t\Big)^\frac12. 
$$
Then, the energy bound can be followed by the  ``radialish'' Sobolev inequality, local smoothing effect, and Morawetz estimate when $s>\frac12$.  Unfortunately, this argument does not work here, since we lack the $\nabla v$-estimates in the view of \eqref{eq:intro-super-ctical}, when $s<\frac12$ in our situation.

To overcome the difficulty, we observe that there is still some gap in the estimate
\EQ{
\eqref{eq:intro-energy-increment}\lsm \norm{\nabla w_{\rm hi}}_{L_t^\I L_x^2} \norm{\nabla v_{\rm hi}}_{L_t^2 L_x^\I} \norm{w_{\rm low}}_{L_{t,x}^4}^2,
} 
towards the desired bound $N_0^{2(1-s)}$. This gives us the room to use the global space-time estimates obtained above and the bilinear Strichartz estimate, which can further lower down the derivative for $\nabla v_{\rm hi}$.

\subsection{Organization of the paper}
In Section \ref{sec:pre}, we give some notation and useful results. In Section \ref{sec:str}, we prove the almost sure space-time estimates for the linear solution. Then, we prove the local results in Theorem \ref{thm:local} in Section \ref{sec:lwp}, and prove the global well-posedness and scattering results in Theorem \ref{thm:global} in Section \ref{sec:gwp}.

\vskip 1.5cm

\section{Preliminary}\label{sec:pre}

\vskip .5cm
\subsection{Notation}
For any $a\in\R$, $a\pm:=a\pm\epsilon$ for some small $\epsilon>0$. For any $z\in\C$, we define $\re z$ and $\im z$ as the real and imaginary part of $z$, respectively.

Let $C>0$ denote some constant, and write $C(a)>0$ for some constant depending on coefficient $a$. If $f\loe C g$, we write $f\lsm g$. If $f\loe C g$ and $g\loe C f$, we write $f\sim g$. Suppose further that $C=C(a)$ depends on $a$, then we write $f\lsm_a g$ and $f\sim_a g$, respectively. If $f\loe 2^{-5}g$, we denote $f\ll g$ or $g\gg f$.

Moreover, we write ``a.e. $\om\in\Om$'' to mean ``almost every $\om\in\Om$''.

We use $\wh f$ or $\F f$ to denote the Fourier transform of $f$:
\EQ{
\wh f(\xi)=\F f(\xi):= \int_{\R^d} e^{-ix\cdot\xi}f(x)\rm dx.
}
We also define
\EQ{
\F^{-1} g(x):= \rev{(2\pi)^d}\int_{\R^d} e^{ix\cdot\xi}g(\xi)\rm d\xi.
}
Using the Fourier transform, we can define the fractional derivative $\abs{\nabla} := \F^{-1}|\xi|\F $ and $\abs{\nabla}^s:=\F^{-1}|\xi|^s\F $.  

We next recall the unit-scale frequency decomposition in Definition \ref{defn:randomization}. Let $\wt \psi\in C_0^\I(\R^3)$ be a real-valued function such that $\wt \psi\goe 0$, $\wt \psi(-\xi)=\wt \psi(\xi)$ for all $\xi\in\R^3$ and 
\EQ{
	\wt \psi(\xi) =\left\{ \aligned
	&1\text{, when $\xi\in [-\frac12,\frac12]^3$,}\\
	&0\text{, when $\xi\notin [-1,1]^3$.}
	\endaligned
	\right.
}
Let $\psi\in C_0^\I(\R^d)$ be a real-valued function, satisfying that for all $\xi\in\R^d$, $0\loe \psi\loe 1$, $\supp \psi\subset[-1,1]^d$, $\psi(-\xi)=\psi(\xi)$, and $\sum_{k\in\Z^d}\psi(\xi-k)=1$. For any $k\in\Z^d$, define $\psi_k(\xi)=\psi(\xi-k)$. Then, we define 
\EQ{
	f_k=\bx_kf:=\F^{-1}\brkb{\psi_k\F f}.
}
We also introduce a fattening version of the unit-scale decomposition:
\EQ{
\wt \bx_kf:=\F^{-1}\brkb{\wt\psi(2^{-1}(\xi-k))\F f}.
}
Therefore, by $1\loe\sum_{k\in\Z^d}\wt\psi(2^{-1}(\xi-k))\loe C$, for all $\xi\in\R^d$, we have orthogonality
\EQ{
\normb{\wt \bx_k f}_{l_{k\in\Z^d}^2L_x^2}\sim \norm{f}_{L_x^2}.
}
Note also that on the support of $\psi_k(\xi)$, we have $\wt \psi(2^{-1}(\xi-k))=1$, which implies $\bx_k=\bx_k\wt \bx_k$.

We also need the usual inhomogeneous Littlewood-Paley decomposition for the dyadic number. Take a cut-off function $\phi\in C_{0}^{\infty}(\R)$ such that $\phi(r)=1$ if $0\loe r\loe1$ and $\phi(r)=0$ if $r>2$. 

Then, we introduce the spatial cut-off function. Denote $\chi_0(r)=\phi(r)$, and $\chi_j(r)=\phi(2^{-j}r)-\phi(2^{-j+1}r)$ for $j\in\N^+$. We also define a fattening version $\wt  \chi_j:=\phi(2^{-j-1}|\xi|)--\phi(2^{-j+2}r)$ with the property $\chi_j=\chi_j\wt \chi_j$.

Next, we give the definition of Littlewood-Paley dyadic projection operator. For dyadic $N\in 2^\N$, let $\phi_{N}(r) := \phi(N^{-1}r)$. Then, we define
\EQ{
\ph_1(r):=\phi(r)\text{, and }\ph_N(r):=\phi_N(r)-\phi_{N/2}(r)\text{, for any }N\goe 2.
}
We define the inhomogeneous Littlewood-Paley dyadic operator 
\EQ{
f_1=P_1f:=\mathcal{F}^{-1}\brko{ \ph_{1}(|\xi|) \hat{f}(\xi)},
}
and for any $N\goe 2$,
\EQ{
f_{N}= P_N f := \mathcal{F}^{-1}\brko{ \ph_N(|\xi|) \hat{f}(\xi)}.
}
Then, by definition, we have $f=\sum_{N\in2^\N}f_N$. Moreover, we also need the following: for any $N\in2^\N$,
\EQ{
f_{\loe N}=P_{\loe N} f := &\mathcal{F}^{-1}\brko{ \phi_{N}(|\xi|) \hat{f}(\xi)}, \\
f_{\ll N}=P_{\ll N} f :=& \mathcal{F}^{-1}\brko{\phi_{N}(2^{5}|\xi|) \hat{f}(\xi)}, \\
f_{\lsm N}=P_{\lsm N} f :=& \mathcal{F}^{-1}\brko{ \phi_{N}(2^{-5}|\xi|) \hat{f}(\xi)},
}
and
\EQ{
f_{\sim N}=P_{\sim N}f:= f_{\lsm N}-f_{\ll N}.
}
We also denote that $f_{\goe N}=P_{\goe N} f := f- P_{\loe N} f$, $f_{\gg N}=P_{\gg N} f:=f-P_{\lsm N}f$, and $f_{\gsm N}=P_{\gsm N}f:= f-P_{\ll N}f$.

Let $\Sch(\R^d)$ be the Schwartz space,  $\Sch'(\R^d)$ be the tempered distribution space, and $C_0^\I(\R^d)$ be the space of all the smooth compact-supported functions.

For Banach spaces $X$ and $Y$, we denote $X + Y$ as the sum space of $X$ and $Y$.

Given $1\loe p \loe \I$, $L^p(\R^d)$ denotes the usual Lebesgue space. For any $0\loe s< d/p$, we define the Sobolev space
\EQ{
	\dot W^{s,p}(\R^d) := \fbrkb{f\in\Sch'(\R^d): \norm{f}_{\dot{W}^{s,p}(\R^d)}:= \norm{\abs{\nabla}^{s}f}_{L^p(\R^d)}<+\I}.
}
We denote that $\dot{H}^s(\R^d):=\dot{W}^{s,2}(\R^d)$. The inhomogeneous spaces are defined by 
\EQ{
W^{s,p}(\R^d)=\dot W^{s,p} \cap L^p(\R^d)\text{, and }H^{s}(\R^d)=\dot H^{s} \cap L^2(\R^d).
}
We often use the abbreviations $H^s=H^s(\R^d)$ and $L^p=L^p(\R^d)$. We also define $\jb{\cdot,\cdot}$ as real $L^2$ inner product:
\EQ{
	\jb{f,g} = \re\int f(x)\wb{g}(x)\dx.
}

For any $1\loe p <\I$, define $l_N^p=l_{N\in 2^\N}^p$ by its norm
\EQ{
\norm{c_N}_{l_{N\in 2^\N}^p}^p:=\sum_{N\in 2^\N}|c_N|^p.
}
The space $l_k^p=l_{k\in\Z^d}^p$ is defined in a similar way. In this paper, we use the following abbreviations
\EQ{
\sum_{N:N\loe N_1}:=\sum_{N\in2^\N:N\loe N_1}\text{, and }\sum_{N\loe N_1}:=\sum_{N,N_1\in2^\N:N\loe N_1}.
}

We then define the mixed norms: for $1\loe q< \I$, $1\loe r\loe \I$, and the function $u(t,x)$, we define
\EQ{
\norm{u}_{L_t^q L_x^r(\R\times \R^d)}^q:= \int_{\R}\norm{u(t,\cdot)}_{L_x^r}^q\dd t,
}
and for the function $u_N(x)$, we define
\EQ{
\norm{u_N}_{l_N^q L_x^r(2^\N\times \R^d)}^q:= \sum_{N}\norm{u_N(\cdot)}_{L_x^r}^q.
}
The $q=\I$ case can be defined similarly.

For any $0\loe\gamma\loe1$, we call that the exponent pair $(q,r)\in\R^2$ is $\dot H^\ga$-$admissible$, if $\frac{2}{q}+\frac{d}{r}=\half d-\ga$, $2\loe q\loe\I$, $2\loe r\loe\I$, and $(q,r,d)\ne(2,\I,2)$. If $\ga=0$, we say that $(q,r)$ is $L^2$-$admissible$.

%Finally, in the proof of scattering, we use the following Strichartz norm: let $I\subset \R$ be an interval, then for any $0\loe l \loe 1$, we denote
%\EQn{\label{defn:sl-norm}
%\norm{u}_{S^l(I)} := \normb{N^l\norm{P_Nu}_{L_t^\I L_x^2(I\times\R^3)} + N^l\norm{(i\pd_t +\De)P_N u}_{L_t^{\frac43}L_x^{\frac32}(I\times\R^3)}}_{l_{N\in 2^\N}^2}.
%}
\subsection{Atom space and bounded variation space}
We recall the definitions of $U^p$ and $V^p$, and some properties used in this paper. The $U^p$-$V^p$ method was first introduced by Koch-Tataru \cite{KT05CPAM}, and we also refer the readers to \cite{HHK10Poinc,KTV14book,KT18Duke,CH18AnnPDE} for their complete theories. 

\begin{defn}
	Let $\mathcal{Z}$ be the set of finite partitions $-\I<t_0<t_1<...<t_K=\I$. 
	\enu{
		\item
		For $\fbrk{t_k}_{k=0}^K\in\ZZ$ and $\fbrk{\phi_k}_{k=0}^{K-1}\subset L_x^2$ with $\sum_{k=0}^{K-1} \norm{\phi_k}_{L_x^2}^p =1$, we call the function $a:\R\ra L_x^2$ given by $a=\sum_{k=1}^K \mathbbm{1}_{[t_{k-1},t_k)}\phi_{k-1}$ a $U^p$-$atom$. Furthermore, we define the atomic space
		\EQn{
			U^p(\R;L^2):=\fbrkb{u=\sum_{j=1}^\I \la_ja_j: a_j\ U^p\text{-atom, }\la_j\in\C\text{ with }\sum_{j=1}^\I\abs{\la_j}<\I},
		}
		with norm
		\EQn{\label{eq:upnorm}
			\norm{u}_{U^p(\R;L^2)}:=\inf\fbrkb{\sum_{j=1}^\I\abs{\la_j}:u=\sum_{j=1}^\I \la_ja_j,a_j\  U^p\text{-atom, }\la_j\in\C}.
		}
		\item
		We define $V^p(\R;L^2)$ as the normed space of all functions $v:\R\ra L^2$ such that 
		\EQn{\label{eq:vpnorm}
			\norm{v}_{V^p(\R;L_x^2)}:=\sup_{\fbrk{t_k}_{k=0}^K\in\ZZ}\brkb{\sum_{k=1}^K\norm{v(t_k)-v(t_{k-1})}_{L_x^2}^p}^{1/p}
		}
		is finite, where we use the convention $v(t_K)=v(\I)=0$.  $V_{rc}^p$ denotes the closed subspace of all right-continuous $V^p$ functions with $\lim_{t\ra -\I}v(t)=0$.
		\item
		We define $U_\De^2(\R; L_x^2)$ as the adapted normed space:
		\EQ{
			U_\De^2(\R; L_x^2):=\fbrkb{u:\norm{u}_{U_\De^2(\R; L_x^2)}:=\norm{e^{-it\De}u}_{U^2(\R;L_x^2)}<\I}.
		}
		Similarly, $V_\De^2(\R; L_x^2)$ denotes the adapted normed space
		\EQ{
			V_\De^2(\R; L_x^2):=\fbrkb{u:\norm{u}_{V_\De^2(\R; L_x^2)}:=\norm{e^{-it\De}u}_{V^2(\R;L_x^2)}<\I, e^{-it\De}u\in V_{rc}^2}.
		}
	}
\end{defn}

In this paper, we will use restriction spaces to some interval $I\subset \R$: $U^p(I;L_x^2)$, $V^p(I;L_x^2)$, $U_\De^p(I;L_x^2)$, and $V_\De^p(I;L_x^2)$. See Remark 2.23 in \cite{HHK10Poinc} for more details.

Note that for $1\loe p<q<\I$, the embeddings
\EQ{
U^p(\R;L_x^2) \hra L_t^\I(\R;L_x^2)\text{, }V^2(\R;L_x^2) \hra L_t^\I(\R;L_x^2),
}
and $U^p\hra V_{rc}^p \hra U^q$ are continuous.

We need the following classical linear estimate and duality formula:
\begin{lem}[\cite{HHK10Poinc}]\label{lem:upvpduality}
	Let $I$ be an interval such that $0=\inf I$. Then, for any $f\in L_x^2$,
	\EQ{
	\norm{e^{it\De}f}_{U_\De^2(I;L_x^2)} \lsm \norm{f}_{L_x^2},
} 
and for $F(t,x)\in L_t^1L_x^2(I\times\R^d)$,
	\EQ{
		\normb{\int_0^t e^{i(t-s)\De} F(s) \ds}_{U_\De^2(I;L_x^2)}
		= \sup_{\norm{g}_{V_\De^2(I;L_x^2)}=1} \abs{\int_{I}\int_{\R^3} F(t)\wb{g(t)}\dx\dd t}.
	}
\end{lem}
We also need the following interpolation result to transfer from $U_\De^2$ into $V_\De^2$.
\begin{lem}[\cite{HHK10Poinc}]\label{lem:upvp-interpolation}
	Let $q>1$, $E$ be a Banach space and $T:U_\De^q \ra E$ be a bounded, linear operator with $\norm{Tu}_E\loe C_q \norm{u}_{U_\De^q}$. In addition, assume that for some $1\loe p <q$, there exists $C_p\in(0,C_q]$ such that the estimate $\norm{Tu}_E\loe C_p \norm{u}_{U_\De^p}$ holds true for all $u\in U_\De^p$. Then, $T$ satisfies the estimate for $u\in V_{\De}^p$,
	\EQn{\label{esti:interpolation}
		\norm{Tu}_E\loe \frac{4}{(1-p/q)\ln 2} C_p \brkb{1+2(1-p/q)\ln 2 +\ln\frac{C_q}{C_p}}\norm{u}_{V_\De^p}.
	}
\end{lem}

%\begin{rem}
%	Inequality \eqref{esti:interpolation} provides an interpolation with epsilon derivative loss. There is an alternative way transferring from $U^p$ to $V^p$, see \cite{CH18APDE}.
%\end{rem}

\subsection{Useful lemmas}
In this subsection, we gather some useful results.
\begin{lem}[Schur's test]\label{lem:schurtest}
	For any $a>0$, let sequences $\fbrk{a_N}$,  $\fbrk{b_N}\in l_{N\in2^\N}^2$, then we have
	\EQn{\label{eq:schurtest}
		\sum_{N_1\loe N} \brkb{\frac{N_1}{N}}^a a_N b_{N_1} \lsm \norm{a_N}_{l_N^2} \norm{b_N}_{l_N^2}.
	}
\end{lem}
\begin{lem}[Hardy's inequality]\label{lem:hardy}
	For $1<p<d$, we have that
	\EQ{
		\normb{|x|^{-1}u}_{L_x^p(\R^d)}\lsm \norm{\nabla u}_{L_x^p(\R^d)}.
	}
\end{lem}
\begin{lem}[Local smoothing, \cite{CS88JAMS,GV92CMP,KPV91Indiana}]\label{lem:local-smoothing}
	We have that 
	\EQ{
		\sup_{R>0}R^{-\frac12}\norm{e^{it\De}f}_{L_t^2(\R;L_{|x|\loe R}^2)}\lsm \normb{|\nabla|^{-\frac12}f}_{L_x^2}.
	}
\end{lem}
\begin{lem}[Strichartz estimate, \cite{KT98AJM,KTV14book}]\label{lem:strichartz}
	Let $I\subset \R$. Suppose that $(q,r)$ and $(\wt{q},\wt{r})$ are $L_x^2$-admissible.
	Then,
	\EQn{\label{eq:strichartz-1}
		\norm{ e^{it\De}\ph}_{L_t^qL_x^r(\R\times\R^d)} \lsm \norm{\ph}_{L_x^2},
	}
	and
	\EQn{\label{eq:strichartz-2}
		\normb{\int_0^t e^{i(t-s)\De} F(s)\ds}_{L_t^qL_x^r(\R\times \R^d)} \lsm \norm{F}_{L_t^{\wt{q}'} L_x^{\wt{r}'}(\R\times\R^d)}.
	}
	Moreover, for $2\loe q<\I$,
	\EQn{
	\norm{u}_{L_t^q L_x^r(I\times \R^3)} \lsm \norm{u}_{U_\De^q(I;L_x^2)},
}
	and if we assume further $q\ne2$, then 
	\EQn{\label{eq:strichartz-v2}
	\norm{u}_{L_t^q L_x^r(I\times \R^3)}	 \lsm \norm{u}_{V_\De^2(I;L_x^2)}.
	}
\end{lem}
In this paper, we need the the following multi-scale bi-linear Strichartz estimate for Schr\"odinger equation, which is a particular case of Theorem 1.2 in \cite{Can19MathAnn}: 
\begin{lem}\label{lem:bilinearstrichartz-origin}
	Let $1\loe q,r \loe 2$, $\rev q + \frac{2}{r}<2$, and suppose that $M,N\in2^\N$ satisfy $M\ll N$. Then for any $\phi,\psi\in L_x^2(\R^3)$,
	\EQn{
		\norm{[e^{it\De}P_N\phi ][e^{\pm it\De}P_M \psi]}_{L_t^qL_x^r(\R\times\R^3)} \lsm &  \frac{M^{4-\frac{4}{r}-\frac{2}{q}}}{N^{1-\rev r}}\norm{P_N\phi}_{L_x^2}\norm{P_M\psi}_{L_x^2}.
	}
\end{lem}
The bilinear Strichartz estimate was first introduced by Bourgain \cite{Bou98IMRN}, and further extended in \cite{CKSTT08Annals,Vis07Duke}, when $q=r=2$. The $q,r<2$ case was referred to bilinear restriction estimates for paraboloid, first obtained by Tao \cite{Tao03GAFA}, based on the method developed by Wolff \cite{Wol01Annals}.

We will frequently use the version of bi-linear estimate for general functions, see Candy's result \cite{Can19MathAnn}.
\begin{lem}\label{lem:bilinearstrichartz}
	Let $I\subset\R$, $a\in I$, $1\loe q,r \loe 2$, $\rev q + \frac{2}{r}<2$, and suppose that $M,N\in2^\N$ satisfy $M\ll N$. 
	Then,
	\EQn{\label{eq:bilinearstrichartz}
	\norm{P_NuP_Mv}_{L_t^qL_x^r(I\times\R^3)} \lsm &  \frac{M^{4-\frac{4}{r}-\frac{2}{q}}}{N^{1-\rev r}}\norm{P_Nu}_{ U_\De^2(I;L_x^2(\R^3))} \norm{P_Mv}_{ U_\De^2(I;L_x^2(\R^3))},
	}
%where $S^0$-norm is defined in \eqref{defn:sl-norm}.
\end{lem}
\subsection{Maximal function estimates and Littlewood-Paley theory}
Let $\M$ be the Hardy-Littlewood maximal operator:
\EQ{
\M f(x):=\sup_{r>0}\frac{1}{|B(0,r)|}\int_{B(0,r)} |f(x-y)|\dy,
}
where $B(0,r)=\fbrk{x\in\R^d:|x|\loe r}$. $\M$ is bounded on $L_x^p$ for $1<p<\I$. Furthermore, we have the vector-valued version of the boundedness:
\begin{lem}[$L^pl^2$-boundedness for maximal function, see \cite{Stein93book}]\label{lem:HL-boundedness}
Let $1<p<\I$ and $\fbrk{f_j}_{j\in \N^+}$ be a sequence of functions such that $\norm{f_j}_{l_{j\in\N^+}^2}\in L_x^p$. Then, we have
\EQ{
\norm{\M(f_j)}_{L_x^p l_{j\in\N^+}^2}\lsm \norm{f_j}_{L_x^p l_{j\in\N^+}^2}.
}
\end{lem}
We also gather some useful classical results about the Littlewood-Paley projection operator.
\begin{lem}[Maximal Littlewood-Paley estimates]\label{lem:linfty-littewoodpaley}
Let $1<p<\I$ and $f\in L_x^p(\R^d)$. Then, we have
\EQ{
\normb{\sup_{N\in 2^\N}\absb{P_{N} f}}_{L_x^p} + \normb{\sup_{N\in 2^\N}\absb{P_{\loe N} f}}_{L_x^p} \lsm \norm{f}_{L_x^p}.
}
\end{lem}
\begin{proof}
Note that  $\F^{-1}(\phi_{N})$ is a $L^1$-renormalised, radial Schwartz function, we have that for any $x\in \R^d$, 
\EQ{
\absb{P_{N} f(x)} = \absb{\F^{-1}(\phi_{N})* f(x)}\lsm \M(f)(x),
}
where $\M$ is the Hardy-Littlewood maximal operator. Then, by the $L^p$ boundedness of $\M$,
\EQ{
\normb{\sup_{N\in 2^\N}\absb{P_{N} f}}_{L_x^p} \lsm \normb{\M(f)}_{L_x^p} \lsm \norm{f}_{L_x^p}.
}
The proof for $P_{\loe N}f$ follows similarly.
\end{proof}

\begin{lem}[Littlewood-Paley estimates]\label{lem:littlewood-paley}
	Let $1<p<\I$ and $f\in L_x^p(\R^d)$. Then, we have
	\EQ{
		\norm{f_N}_{L_x^p l_{N\in2^\N}^2}\sim_p \norm{f}_{L_x^p}.
	}
\end{lem}

\subsection{Probabilistic theory}
We recall the large deviation estimate, which holds for the random variable sequence $\fbrk{\re g_k,\im g_k}$ in the Definition \ref{defn:randomization}.
\begin{lem}[Large deviation estimate, \cite{BT08inventI}]\label{lem:large-deviation}
Let $(\Om, \A, \PP)$ be a probability space. Let $\fbrk{g_n}_{n\in\N^+}$ be a sequence of real-valued, independent, zero-mean random variables with associated distributions $\fbrk{\mu_n}_{n\in\N^+}$ on $\Om$. Suppose $\fbrk{\mu_n}_{n\in\N^+}$ satisfies that there exists $c>0$ such that for all $\ga\in\R$ and $n\in\N^+$
\EQ{
\absb{\int_{\R}e^{\ga x}\mathrm d \mu_n(x)}\loe e^{c\ga^2},
}
then there exists $\al>0$ such that for any $\la>0$ and any complex-valued sequence $\fbrk{c_n}_{n\in\N^+}\in l_n^2$, we have
\EQ{
\PP\brkb{\fbrkb{\om:\absb{\sum_{n=1}^\I c_n g_n(\om)}>\la}}\loe 2\exp\fbrkb{-\al\la\norm{c_n}_{l_n^2}^{-2}}.
}
Furthermore, there exists $C>0$ such that for any $2\loe p<\I$ and complex-valued sequence $\fbrk{c_n}_{n\in\N^+}\in l_n^2$, we have
\EQn{\label{eq:large-deviation}
\normb{\sum_{n=1}^\I c_n g_n(\om)}_{L_\om^p(\Om)} \loe C\sqrt{p} \norm{c_n}_{l_n^2}.
}
\end{lem}
The following lemma can be proved by the method in \cite{Tzv10gibbs}, see also \cite{DLM20AJM,DLM19Adv}.
\begin{lem}\label{lem:probability-estimate}
Let $F$ be a real-valued measurable function on a probability space $(\Om, \A, \PP)$. Suppose that there exists $C_0>0$, $K>0$ and $p_0\goe 1$ such that for any $p\goe p_0$, we have
\EQ{
\norm{F}_{L_\om^p(\Om)} \loe \sqrt{p} C_0 K.
}
Then, there exist $c>0$ and $C_1>0$, depending on $C_0$ and $p_0$ but independent of $K$, such that for any $\la>0$,
\EQ{
\PP\brkb{\fbrkb{\om\in\Om:|F(\om)|>\la}} \loe C_1 e^{-c\la^2K^{-2}}.
}
%Particularly, we have
%\EQ{
%\PP\brkb{\fbrkb{\om\in\Om:|F(\om)|<\I}}=1.
%}
\end{lem}

\vskip 1.5cm

\section{Almost sure Strichartz estimates}\label{sec:str}

\vskip .5cm

\subsection{Non-radial data}
\begin{lem}\label{lem:almostsure-strichartz}
Let $s\in \R$ and $f\in H_x^s$. Suppose that the randomization $f^\om$ is defined in Definition \ref{defn:randomization}. Then, we have the following estimates:
\enu{
\item 
For any $2\loe q, r<\I$ with $\frac2q+\frac 3r\loe \frac 32$, and for any $p\goe 2$,
\EQn{\label{eq:bound-v-pro-ltqxr}
\norm{\jb{\nabla}^se^{it\De}P_Nf^\om}_{L_\om^pl_N^2L_t^q L_x^r(\Om\times 2^\N\times\R\times\R^3)}\lsm_{q,r}\sqrt{p}\norm{f}_{H_x^s}.
}
\item 
For any $p\goe2$,
\EQn{\label{eq:bound-v-pro-ltinftyx2}
	\norm{\jb{\nabla}^{s}e^{it\De}P_Nf^\om}_{L_\om^pl_N^2L_t^\I L_x^2(\Om\times 2^\N\times\R\times\R^3)}\lsm\sqrt{p}\norm{f}_{H_x^s}.
}
\item 
For any $2\loe q<\I$ and $p\goe 2$,
\EQn{\label{eq:bound-v-pro-ltqxinfty}
\norm{\jb{\nabla}^{s-}e^{it\De}P_Nf^\om}_{L_\om^pl_N^2L_t^q L_x^\I(\Om\times 2^\N\times\R\times\R^3)}\lsm_{q}\sqrt{p}\norm{f}_{H_x^s}.
}
\item 
For any $2< r\loe\I$ and $p\goe 2$,
\EQn{\label{eq:bound-v-pro-ltinftyxr}
\norm{\jb{\nabla}^{s-}e^{it\De}P_Nf^\om}_{L_\om^pl_N^2L_t^\I L_x^r(\Om\times 2^\N\times\R\times\R^3)}\lsm_{r}\sqrt{p}\norm{f}_{H_x^s}.
}
} 
\end{lem}
\begin{remark}
We remark that for example, by Minkowski's inequality and Lemma \ref{lem:littlewood-paley}, \eqref{eq:bound-v-pro-ltqxr} also gives that
\EQ{
	\norm{\jb{\nabla}^se^{it\De}f^\om}_{L_\om^pL_t^q L_x^r(\Om\times\R\times\R^3)}\lsm\sqrt{p}\norm{f}_{H_x^s}.
}
\end{remark}
\begin{proof}
In the proof of this lemma, we restrict the variables on $\om\in\Om$, $N\in2^\N$, $t\in\R$, $x\in\R^3$, and $k\in\Z^3$. The implicit constant $C(q,r)$ may depend on some fixed exponents $q,r$, and we abbreviate it to $C$ for short.

We first prove \eqref{eq:bound-v-pro-ltqxr}. Let $p_0=\max\fbrk{q,r}$. If $2\loe p\loe p_0$, then $p\sim p_0\sim 1$. By H\"older's inequality in $\om$, Minkowski's inequality, and Lemma \ref{lem:large-deviation}, 
\EQ{
\norm{\jb{\nabla}^se^{it\De}P_Nf^\om}_{L_\om^pl_N^2L_t^q L_x^r}\lsm & \norm{\jb{\nabla}^se^{it\De}P_Nf^\om}_{l_N^2L_t^q L_x^rL_\om^{p_0}} \\
\lsm & \sqrt{p_0} \norm{\jb{\nabla}^se^{it\De}\bx_kP_Nf}_{l_N^2L_t^q L_x^rl_k^2} \\
\lsm & \sqrt{p} \norm{\jb{\nabla}^se^{it\De}P_N\bx_kf}_{l_N^2l_k^2L_t^q L_x^r}.
}
If $p\goe p_0$, simply using Minkowski's inequality and Lemma \ref{lem:large-deviation}, 
\EQ{
\norm{\jb{\nabla}^se^{it\De}P_Nf^\om}_{L_\om^pl_N^2L_t^q L_x^r}\lsm & \norm{\jb{\nabla}^se^{it\De}P_Nf^\om}_{l_N^2L_t^q L_x^rL_\om^{p}} \\
\lsm & \sqrt{p} \norm{\jb{\nabla}^se^{it\De}\bx_kP_Nf}_{l_N^2L_t^q L_x^rl_k^2} \\
\lsm & \sqrt{p} \norm{\jb{\nabla}^se^{it\De}P_N\bx_kf}_{l_N^2l_k^2L_t^q L_x^r}.
}
Therefore, we have that for any $p\goe 2$,
\EQn{\label{esti:bound-v-pro-ltqxr-1}
\norm{\jb{\nabla}^se^{it\De}P_Nf^\om}_{L_\om^pl_N^2L_t^q L_x^r}
\lsm & \sqrt{p} \norm{\jb{\nabla}^se^{it\De}P_N\bx_kf}_{l_N^2l_k^2L_t^q L_x^r}.
}
Now, let $2\loe r_0\loe r$ such that $(q,r_0)$ is $L_x^2$-admissible. For any $k\in\Z^3$, by the support property of $\psi_k$ and  Bernstein's inequality, we have
\EQn{\label{esti:bound-v-pro-ltqxr-2}
\norm{\jb{\nabla}^se^{it\De}P_N\bx_kf}_{L_t^q L_x^r}\lsm\norm{\jb{\nabla}^se^{it\De}P_N\bx_kf}_{L_t^q L_x^{r_0}}.
}
Then, by \eqref{esti:bound-v-pro-ltqxr-1}, \eqref{esti:bound-v-pro-ltqxr-2}, Lemma \ref{lem:strichartz}, and orthogonality, we have
\EQ{
\norm{\jb{\nabla}^se^{it\De}P_Nf^\om}_{L_\om^pl_N^2L_t^q L_x^r}\lsm & \sqrt{p} \norm{\jb{\nabla}^se^{it\De}P_N\bx_kf}_{l_N^2l_k^2L_t^q L_x^{r_0}} \\
\lsm & \sqrt{p} \norm{\jb{\nabla}^sP_N\bx_kf}_{l_N^2l_k^2 L_x^2}\lsm \sqrt{p} \norm{f}_{H_x^s}.
}
This gives \eqref{eq:bound-v-pro-ltqxr}.

Next, we prove \eqref{eq:bound-v-pro-ltinftyx2}. By Plancherel's identity, we have
\EQn{\label{esti:bound-v-pro-ltinftyx2-1}
	\norm{\jb{\nabla}^{s}e^{it\De}P_Nf^\om}_{L_\om^pl_N^2L_t^\I L_x^2}\lsm \norm{\jb{\nabla}^{s}P_Nf^\om}_{L_\om^pl_N^2 L_x^2}\lsm \norm{\jb{\nabla}^{s}f^\om}_{L_\om^p L_x^2}.
}
Then, by Minkowski's inequality and Lemma \ref{lem:large-deviation},
\EQn{\label{esti:bound-v-pro-ltinftyx2-2}
	\norm{\jb{\nabla}^{s}f^\om}_{L_\om^p L_x^2}\lsm & \norm{\jb{\nabla}^{s}f^\om}_{L_x^2 L_\om^p} \\
	\lsm & \sqrt{p} \norm{\jb{\nabla}^{s}\bx_k f}_{L_x^2 l_k^2}\lsm \sqrt{p}\norm{f}_{H_x^s}.
}
Then, \eqref{esti:bound-v-pro-ltinftyx2-1} and \eqref{esti:bound-v-pro-ltinftyx2-2} imply \eqref{eq:bound-v-pro-ltinftyx2}.

We then prove \eqref{eq:bound-v-pro-ltqxinfty}. Let $0<\ep\ll 1$. Using the Sobolev's embedding $ W_x^{2\ep,\frac3\ep} \hra L_x^\I$ in $x$, we have
\EQn{\label{esti:bound-v-pro-ltqxinfty-1}
	\norm{\jb{\nabla}^{s-2\ep}e^{it\De}P_Nf^\om}_{L_\om^pl_N^2L_t^q L_x^\I}\lsm \norm{\jb{\nabla}^{s}e^{it\De}P_Nf^\om}_{L_\om^pl_N^2L_t^q L_x^{\frac3\ep}}.
}
%Let $p_0=\max\fbrk{q,\frac3\ep}$ and $2\loe r_0\loe\frac3\ep$ such that $(q,r_0)$ is $L_x^2$-admissible. Then, similar as above, by H\"older's, Minkowski's, Bernstein's inequalities, Lemmas \ref{lem:strichartz}, and \ref{lem:large-deviation}, we have
%\EQn{\label{esti:bound-v-pro-ltqxinfty-2}
%\norm{\jb{\nabla}^{s}e^{it\De}P_Nf^\om}_{L_\om^pl_N^2L_t^q L_x^{\frac3\ep}} 
%\lsm & \norm{\jb{\nabla}^{s}e^{it\De}P_Nf^\om}_{l_N^2L_t^q L_x^{\frac3\ep}L_\om^{p_0}} \\
%\lsm & \sqrt{p} \norm{\jb{\nabla}^{s}e^{it\De}P_N\bx_kf}_{l_N^2L_t^q L_x^{\frac3\ep}l_k^2} \\
%\lsm & \sqrt{p} \norm{\jb{\nabla}^{s}e^{it\De}P_N\bx_kf}_{l_N^2l_k^2L_t^q L_x^{\frac3\ep}} \\
%\lsm & \sqrt{p} \norm{\jb{\nabla}^{s}e^{it\De}P_N\bx_kf}_{l_N^2l_k^2L_t^q L_x^{r_0}} \\
%\lsm & \sqrt{p} \norm{\jb{\nabla}^{s}P_N\bx_kf}_{l_N^2l_k^2 L_x^2}\lsm \sqrt{p} \norm{f}_{H_x^s}.
%}
By \eqref{esti:bound-v-pro-ltqxinfty-1} and \eqref{eq:bound-v-pro-ltqxr}, we have that \eqref{eq:bound-v-pro-ltqxinfty} holds.

Finally, we prove \eqref{eq:bound-v-pro-ltinftyxr}. We only consider the $r=\I$ case. In fact, when $r<\I$, we can prove it using interpolation between \eqref{eq:bound-v-pro-ltinftyx2} and the $r=\I$ case. Let $0<\ep\ll1$. Using Minkowski's inequality, the Sobolev's embedding $W_t^{2\ep,\frac1\ep}\hra L_t^\I$ in $t$ and $ W_x^{2\ep,\frac3\ep} \hra L_x^\I$ in $x$, we have
\EQn{\label{esti:bound-v-pro-ltinftyxr-1}
	\norm{\jb{\nabla}^{s-6\ep}e^{it\De}P_Nf^\om}_{L_\om^pl_N^2L_t^\I L_x^\I}
	\lsm & 	\norm{\jb{\nabla}^{s-4\ep}e^{it\De}P_Nf^\om}_{L_\om^pl_N^2L_t^\I L_x^{\frac{3}{\ep}}} \\
	\lsm & \norm{\jb{\nabla}^{s-4\ep}\jb{\pd_t}^{2\ep}e^{it\De}P_Nf^\om}_{L_\om^pl_N^2 L_x^{\frac{3}{\ep}} L_t^{\frac1\ep}}\\
	\lsm & \norm{\jb{\nabla}^{s}e^{it\De}P_Nf^\om}_{L_\om^pl_N^2 L_x^{\frac{3}{\ep}} L_t^{\frac1\ep}}\\
	\lsm & \norm{\jb{\nabla}^{s}e^{it\De}P_Nf^\om}_{L_\om^pl_N^2  L_t^{\frac1\ep} L_x^{\frac{3}{\ep}}}.
}
%Let $p_0=\frac3\ep$ and $2< r_0 \loe \frac3\ep$ such that $(\frac1\ep,r_0)$ is $L_x^2$-admissible. Then, similar as above, we have
%\EQn{\label{esti:bound-v-pro-ltinftyxr-2}
%\norm{\jb{\nabla}^{s}e^{it\De}P_Nf^\om}_{L_\om^pl_N^2L_t^{\frac1\ep} L_x^{\frac3\ep}} \lsm & \norm{\jb{\nabla}^{s}e^{it\De}P_Nf^\om}_{l_N^2L_t^{\frac1\ep} L_x^{\frac3\ep} L_\om^{p_0}} \\
%\lsm & \sqrt{p} \norm{\jb{\nabla}^{s}e^{it\De}P_N\bx_kf}_{l_N^2L_t^{\frac1\ep} L_x^{\frac3\ep} l_k^2} \\
%\lsm & \sqrt{p} \norm{\jb{\nabla}^{s}e^{it\De}P_N\bx_kf}_{l_N^2 l_k^2 L_t^{\frac1\ep} L_x^{\frac3\ep}} \\
%\lsm & \sqrt{p} \norm{\jb{\nabla}^{s}e^{it\De}P_N\bx_kf}_{l_N^2 l_k^2 L_t^{\frac1\ep} L_x^{r_0}} \\
%\lsm & \sqrt{p} \norm{\jb{\nabla}^{s}P_N\bx_kf}_{l_N^2 l_k^2 L_x^2} \lsm \sqrt{p} \norm{f}_{H_x^s}.
%}
Then, it follows from \eqref{esti:bound-v-pro-ltinftyxr-1} and \eqref{eq:bound-v-pro-ltqxr} that \eqref{eq:bound-v-pro-ltinftyxr} holds.
\end{proof}

Next, we gather all the space-time norms that will be used below. Let $\frac16\loe s<\frac12$, and given any suitably small constant $\ep>0$. Define the $Y^s(I)$ space by its norm
\EQn{\label{defn:ys-norm}
	\norm{v}_{Y^s(I)}:=&\normb{\jb{\nabla}^{s} v_{ N}}_{l_{N}^2 L_t^2 L_x^{6}\cap l_N^2L_t^{\frac1\ep} L_x^{\frac{6}{3-4\ep}}(2^\N\times I\times\R^3)} 
	+ \normb{\jb{\nabla}^{s}v_N}_{l_N^2L_t^{2} L_x^{1/\ep} \cap l_N^2L_{t,x}^{1/\ep}(2^\N\times I\times\R^3)}\\
	& + \normb{\jb{\nabla}^{s-\ep}v_N}_{l_N^2L_t^{2} L_x^{\I} \cap l_N^2L_{t}^{1/\ep} L_x^{\I}(2^\N\times I\times\R^3)},
}
%\EQn{\label{defn:ys-norm}
%	\norm{v}_{Y^s(I)}:=&\normb{\jb{\nabla}^{s} v_{ N}}_{l_{N}^2 L_t^4 L_x^{\frac92}(2^\N\times I\times\R^3)} + \normb{\jb{\nabla}^{s}v_N}_{l_N^2L_{t,x}^6(2^\N\times I\times\R^3)} \\
%	&+ \normb{\jb{\nabla}^{s}v_N}_{l_N^2L_t^4 L_x^{18}(2^\N\times I\times\R^3)} + \normb{\jb{\nabla}^{s}v_N}_{l_N^2L_t^2 L_x^6(2^\N\times I\times\R^3)} \\
%	& + \norm{v_N}_{l_{N}^2L_t^{\frac{20}{7}} L_x^{10}(2^\N\times I\times\R^3)} + \norm{v_N}_{l_{N}^2 L_{t,x}^4(2^\N\times I\times\R^3)} \\
%	& + \norm{v}_{L_t^4 L_x^{12}( I\times\R^3)} + \norm{v}_{L_t^3 L_x^\I( I\times\R^3)} + \norm{v}_{L_t^4 L_x^6( I\times\R^3)} +  \norm{v}_{L_t^{\frac{3}{1-\ep}}L_x^{\frac{6}{1-3\ep}}(I\times \R^3)} \\
%	& + \norm{v}_{L_t^3 L_x^6( I\times\R^3)} + \norm{v}_{L_t^{\frac{4(4-3\ep)}{5+3\ep}}L_x^{\frac{2(4-3\ep)}{1-3\ep}}(I\times \R^3)}
%	+ \norm{\jb{\nabla}^{3\ep}v}_{L_t^{\frac{6}{1-6\ep}} L_x^{\frac{9}{2+6\ep}}} + \norm{v}_{L_t^{6} L_x^{\frac{9}{2}}},
%}
and the $Z$-norm by
\EQn{\label{defn:zs-norm}
\norm{v}_{Z^s(I)}:=\norm{\jb{\nabla}^{s-\ep}v_N}_{l_N^2L_t^\I L_x^\I(2^\N\times I\times\R^3)} + \norm{\jb{\nabla}^{s}v_N}_{l_N^2L_t^\I L_x^2(2^\N\times I\times\R^3)}.
}
Note that $(2,6)$ and $(1/\ep,6/(3-4\ep))$ are $L_x^2$-admissible. 

Since $\ep$ is a fixed small constant, by Lemmas \ref{lem:almostsure-strichartz} and \ref{lem:probability-estimate}, we have
\begin{cor}\label{cor:Y-norm-Z-norm}
Let $\frac16\loe s<\frac12$ and $f\in H_x^s$. Then, there exist constants $C,c>0$ such that for any $\la$,
\EQ{
	\PP\brkb{\fbrkb{\om\in \Om:\norm{e^{it\De}f^\om}_{Y^s(\R)} + \norm{e^{it\De}f^\om}_{Z^s(\R)}>\la}}\loe C\exp\fbrkb{-c\la^2\norm{f}_{H_x^{s}(\R^3)}^{-2}}.
}
Moreover,  
\EQ{
\norm{e^{it\De}f^\om}_{Y^s(\R)} + \norm{e^{it\De}f^\om}_{Z^s(\R)} <+\I,\quad \mbox{a.e.} \>\>\om\in\Om.
}
\end{cor}
Finally, we remark that in the following, we will frequently using the inequalities without mentioning: for any $2\loe q,r\loe \frac1\ep$ with $\frac2q +\frac 3r\loe \frac32$,
\EQ{
	\norm{\jb{\nabla}^sv}_{L_t^q L_x^r} \lsm \norm{v}_{Y^s},
}
and for any $2\loe q \loe 1/\ep$,
\EQ{
	\norm{\jb{\nabla}^{s-\ep}v}_{L_t^q L_x^\I} \lsm \norm{v}_{Y^s}.
}

\subsection{Radial data}
Here, we derive a super-critical estimate for the randomized radial data that can acquire $\frac12$-derivative. Such class of estimates was first proved by Dodson, L\"uhrmann, and Mendelson  \cite{DLM19Adv} in 4D case, based on their ``radialish" Sobolev's inequality \cite{DLM20AJM}. We adapt their method to the 3D case:
\begin{prop}\label{prop:nabla2infty}
Let $4<r\loe \I$, $s_0\in\R$. Suppose that $f\in H_x^{s_0}(\R^3)$ is radial. Then for any  $s<s_0 + \half 1$, there exist constants $C,c>0$ such that for any $\la>0$,
\EQn{\label{eq:nabla2infty-pro}
\PP\brkb{\fbrkb{\om\in \Om:\normb{\jb{\nabla}^{s}e^{it\De}f^\om}_{L_t^2 L_x^{r}(\R\times \R^3)}>\la}}\loe C\exp\fbrkb{-c\la^2\norm{f}_{H_x^{s_0}(\R^3)}^{-2}}.
}
Moreover, 
\EQn{\label{eq:nabla2infty}
\normb{\jb{\nabla}^{s}e^{it\De}f^\om}_{L_t^2 L_x^{r}(\R\times \R^3)}<\I,\quad \mbox{a.e.} \>\>\om\in\Om.
}
\end{prop}

To prove Proposition \ref{prop:nabla2infty}, we need the following  3D version of radial Sobolev estimate for the square function.
\begin{lem}\label{lem:radialsobolev}
Suppose that the function $f$ is radial and that $2\loe r\loe\I$. Then, for any $\epsilon>0$, there exists $C_\epsilon>0$ such that,
\EQn{\label{eq:radialsobolev}
\normb{|x|^{1-\frac 2r}\norm{f_k}_{l_k^2}}_{L_x^{r}(\R^3)}\loe C_\epsilon \norm{f}_{H_x^{\epsilon}(\R^3)}.
}
\end{lem}
\begin{proof}
It suffices to prove the $r=\I$ case, since the general case can be obtained by interpolation with
\EQ{
\normb{\norm{f_k}_{l_k^2}}_{L_x^{2}(\R^3)}\sim \norm{f}_{L_x^{2}(\R^3)}.
}

Since $f$ is radial, we can write $\wh f(x)=\wh f(|x|)$. Assume without loss of generality that $x=(0,0,|x|)$. Then, by integration-by-parts and the spherical coordinate
\EQ{
	\xi(\rho,\th,\al)=(\rho\sin{\th}\cos{\al}, \rho\sin{\th}\sin{\al}, \rho\cos{\th}),
} 
we have
\begin{subequations}
\EQnn{
f_k=& \int_0^{2\pi} \!\! \int_0^\pi \int_0^\infty \psi_k(\xi(\rho,\th,\al)) \wh f(\rho) e^{i|x|\rho\cos\th }\rho^2\sin \th\rm{d}\rho\rm d\th\rm d\al\nonumber\\
%=& \int_0^{2\pi} \!\! \int_0^\infty  \int_0^\pi  \psi_k(\xi(\rho,\th,\al))\wh f(\rho) \rho\rev{i|x|} \mbox{d}\brk{e^{i|x|\rho\cos\th }}\rm d\th\rm{d}\rho \rm d\al\nonumber\\
= & -\rev{i|x|}   \int_0^\infty\!\! \int_0^{2\pi} \!\! \int_0^\pi  \psi_k\big(\xi(\rho,\th,\al)\big) \>\pd_\th\brk{e^{i|x|\rho\cos\th }}\wh f(\rho) \rho \, d\th\mbox{d}\al\,\rm{d}\rho \nonumber\\
= & \rev{i|x|}   \int_0^\infty\!\!  \int_0^{2\pi}  \!\!\int_0^\pi  \pd_\th\big(\psi_k(\xi(\rho,\th,\al))\big) e^{i|x|\rho\cos\th }\wh f(\rho) \rho\,\rm{d}\th\rm d\al\rm{d}\rho\label{esti:radialstrichartz-1}\\
& - \rev{i|x|}   \int_0^\infty\!\!  \int_0^{2\pi}   \!\!\psi_k(\xi(\rho,\pi,\al)) e^{-i|x|\rho}\wh f(\rho) \rho\,\rm d\al\rm{d}\rho \label{esti:radialstrichartz-2}\\
& +\rev{i|x|}   \int_0^\infty\!\!  \int_0^{2\pi}\!\!   \psi_k(\xi(\rho,0,\al)) e^{i|x|\rho}\wh f(\rho) \rho\,\rm d\al\rm{d}\rho. \label{esti:radialstrichartz-3}
}
\end{subequations}

Denote that 
\EQ{
	k=(k_1,k_2,k_3)=(|k|\sin{\th_k}\cos{\al_k}, |k|\sin{\th_k}\sin{\al_k}, |k|\cos{\th_k}).
} 
By the support of $\psi_k$, we have that $\rho$ is supported on the set $\fbrk{\rho:|\rho-|k||\lsm 1}$,  $\th$ is supported on  $\fbrk{\th:|\th-\th_k|\lsm \rev{|k|}}$, 
and $\al$ is supported on $\fbrk{\al:|\al-\al_k|\lsm \rev{\langle |k|\sin \th_k\rangle}}$. Furthermore, we also have
\EQ{
\abs{\pd_\th\brk{\psi_k(\xi(\rho,\th,\al))}}\lsm \rho \abs{\nabla_\xi\psi_k(\xi(\rho,\th,\al))}.
}
Therefore, we have
\EQn{\label{esti:radialsobolev-1}
\abs{\eqref{esti:radialstrichartz-1}}\lsm & \rev{|x|}\rev{|k|}\rev{\jb{|k|\sin{\th_k}}}\normb{\wh f(\rho)\rho^2}_{L_{|\rho-|k||\lsm 1}^1}\\
\lsm &\rev{|x|}\rev{\jb{|k|\sin{\th_k}}}\normb{\wh f(\rho)\rho}_{L_{|\rho-|k||\lsm 1}^2}.
}
Similarly, we also have
\EQn{\label{esti:radialsobolev-2}
\abs{\eqref{esti:radialstrichartz-2}} + \abs{\eqref{esti:radialstrichartz-3}} \lsm \rev{|x|}\rev{\jb{|k|\sin{\th_k}}}\normb{\wh f(\rho)\rho}_{L_{|\rho-|k||\lsm 1}^2}.
}
Then, by \eqref{esti:radialsobolev-1} and \eqref{esti:radialsobolev-2}, 
\EQ{
|x|^2 \sum_{k\in \Z^3}|f_k|^2 \lsm & \sum_{k\in \Z^3} \rev{\jb{|k|\sin{\th_k}}^2}\normb{\wh f(\rho)\rho}_{L_{|\rho-|k||\lsm 1}^2}^2\\
\sim & \sum_{N\in \N} \sum_{\substack{k_1,k_2\in \Z^2\\|k_1|\loe N,|k_2|\loe N}} \rev{\jb{k_1^2+k_2^2}} \normb{\wh f(\rho)\rho}_{L_{|\rho-N|\lsm 1}^2}^2\\
\lsm &\sum_{N\in \N} \ln N \normb{\wh f(\rho)\rho}_{L_{|\rho-N|\lsm 1}^2}^2
 \lsm  \norm{f}_{H_x^{\epsilon}(\R^3)}^2.
}
This finishes the proof of \eqref{eq:radialsobolev}.
\end{proof}
We also need the following mismatch estimates concerning the commutator of $\chi_j$ and $\bx_k$. The same result was already proved in \cite{DLM19Adv} for 4D, and their argument can be easily extended to general dimensions. Therefore, we omit the details of proof.
\begin{lem}[Mismatch estimates]\label{lem:mismatch}
Let $2\loe r\loe \I$, $j,l\goe 0$ and $k,m\in \Z^3$. Suppose that $l> j+5$ and $|k-m|> 100$. Then for any integer $M>0$, we have
\EQn{\label{eq:mismatch-1}
\norm{\chi_j \bx_k \chi_l}_{L_x^2(\R^3)\ra L_x^r(\R^3)}\loe C_M 2^{-Ml},
}
and
\EQn{\label{eq:mismatch-2}
\norm{\bx_k \chi_l \bx_m}_{L_x^2(\R^3)\ra L_x^r(\R^3)}\loe C_M 2^{-Ml}|k-m|^{-M}.
}
\end{lem}
Now, we give the proof of Proposition \ref{prop:nabla2infty} using Lemmas \ref{lem:radialsobolev} and \ref{lem:mismatch}.
\begin{proof}[Proof of Proposition \ref{prop:nabla2infty}]
Take some $\epsilon>0$ that will be defined later. We first consider the $r<\I$ case. Then, for any $p\goe r$, by Lemma \ref{lem:large-deviation}, 
\begin{subequations}
\EQnn{
\norm{\jb{\nabla}^s e^{it\De}f^\om}_{L_\om^p L_t^2 L_x^r}\lsm &\sum_{N\in2^\N} N^s\norm{ e^{it\De}P_Nf^\om}_{ L_t^2 L_x^r L_\om^p} \nonumber\\
\lsm & \sqrt{p}\sum_{N\in2^\N} N^s\norm{ e^{it\De}\bx_kP_Nf}_{ L_t^2 L_x^r l_k^2} \nonumber\\
\lsm &\sqrt{p} \sum_{N\in2^\N}\sum_{j\goe 0} N^s\norm{\chi_j\bx_k\chi_{\loe j+5} e^{it\De}P_Nf}_{ L_t^2 L_x^r l_{k:|k|\sim N}^2} \label{esti:nabla2infty-1}\\
&+ \sqrt{p}\sum_{N\in2^\N}\sum_{j\goe 0} N^s\norm{\chi_j\bx_k\chi_{> j+5} e^{it\De}P_Nf}_{ L_t^2 L_x^r l_{k:|k|\sim N}^2}.\label{esti:nabla2infty-2}
}
\end{subequations}

We first bound the term \eqref{esti:nabla2infty-1}. When $j=0$, by Minkowski's, Bernstein's inequalities, Lemmas \ref{lem:radialsobolev} and \ref{lem:local-smoothing}, 
\EQn{\label{esti:nabla2infty-1-1}
\sum_{N\in2^\N}N^s\norm{\chi_0\bx_k\chi_{\loe 5} e^{it\De}P_Nf}_{ L_t^2 L_x^r l_{k:|k|\sim N}^2} \lsm & \sum_{N\in2^\N} N^s\normb{\chi_0\bx_k\chi_{\loe 5} e^{it\De}P_Nf}_{ L_t^2  l_{k:|k|\sim N}^2L_x^r }\\
\lsm & \sum_{N\in2^\N} N^s\normb{\bx_k\chi_{\loe 5} e^{it\De}P_Nf}_{ L_t^2  l_{k:|k|\sim N}^2L_x^2 }\\
\lsm & \sum_{N\in2^\N} N^s\normb{\chi_{\loe 5} e^{it\De}P_Nf}_{ L_{t,x}^2}\\
\lsm & \sum_{N\in2^\N} N^{s-\half 1} \norm{P_Nf}_{ L_{x}^2}\lsm \norm{f}_{H_x^{s+\epsilon-\half 1}}.
}
For $j\goe1$, also by Lemmas \ref{lem:radialsobolev} and \ref{lem:local-smoothing}, 
\EQn{\label{esti:nabla2infty-1-2}
&\sum_{N\in2^\N}\sum_{j\goe 1} N^s\norm{\chi_j\bx_k\chi_{\loe j+5} e^{it\De}P_Nf}_{ L_t^2 L_x^r l_{k:|k|\sim N}^2}\\
\lsm & \sum_{N\in2^\N}\sum_{j\goe 1} N^s2^{-(1-\frac2r)j}\norm{\chi_j|x|^{1-\frac2r}\bx_kP_{\sim N}\chi_{\loe j+5} e^{it\De}P_Nf}_{ L_t^2 L_x^r l_k^2}\\
\lsm & \sum_{N\in2^\N}\sum_{j\goe 1} N^s2^{-(1-\frac2r)j}\norm{\abs{\nabla}^\epsilon P_{\sim N}\chi_{\loe j+5} e^{it\De}P_Nf}_{ L_{t,x}^2}\\
\lsm & \sum_{N\in2^\N}\sum_{j\goe 1} N^{s+\epsilon}2^{-(1-\frac2r)j}\norm{\chi_{\loe j+5} e^{it\De}P_Nf}_{ L_{t,x}^2}\\
\lsm & \sum_{N\in2^\N}\sum_{j\goe 1} N^{s-\frac12+\epsilon}2^{-(\frac 12-\frac2r)j}\norm{P_Nf}_{ L_{x}^2}\lsm \norm{f}_{H_x^{s+2\epsilon-\half 1}}.
}
Then, combining \eqref{esti:nabla2infty-1-1} and \eqref{esti:nabla2infty-1-2}, we have
\EQn{\label{eq:nabla2infty-1}
\eqref{esti:nabla2infty-1} \lsm \sqrt{p}\norm{f}_{H_x^{s+2\epsilon-\half 1}}.
}

Next, we consider \eqref{esti:nabla2infty-2}. We decompose that
\begin{subequations}
\EQnn{
\eqref{esti:nabla2infty-2} \lsm&\sqrt{p}\sum_{N\in2^\N}\sum_{j\goe 0} N^s\norm{\chi_j\bx_k\chi_{> j+5} e^{it\De}P_Nf}_{ L_t^2 L_x^r l_k^2}\nonumber\\
\lsm & \sqrt{p}\sum_{N\in2^\N}\sum_{j\goe 0} \sum_{l>j+5} N^s\norm{\chi_j\bx_k\chi_{l}\wt\chi_l e^{it\De}P_Nf}_{ L_t^2 L_x^r l_k^2}\nonumber\\
\lsm & \sqrt{p}\sum_{N\in2^\N}\sum_{j\goe 0} \sum_{l>j+5} N^s\normb{\sum_{m\in\Z^3:|m-k|\loe 100}\chi_j\bx_k\chi_{l}\bx_m \wt\chi_l e^{it\De}P_Nf}_{ L_t^2 L_x^r l_{k}^2} \label{esti:nabla2infty-2-1}\\
& + \sqrt{p}\sum_{N\in2^\N}\sum_{j\goe 0} \sum_{l>j+5} N^s\normb{\sum_{m\in\Z^3:|m-k|> 100}\chi_j\bx_k\chi_{l}\bx_m \wt\chi_l e^{it\De}P_Nf}_{ L_t^2 L_x^r l_{k}^2}.\label{esti:nabla2infty-2-2}
}
\end{subequations}
Now, we take some $M\gg1$. For \eqref{esti:nabla2infty-2-1}, by Minkowski's inequality, Lemmas \ref{lem:mismatch}, and \ref{lem:local-smoothing},
\EQn{
\eqref{esti:nabla2infty-2-1} \lsm &\sqrt{p}\sum_{N\in2^\N}\sum_{j\goe 0} \sum_{l>j+5} N^s\normb{\sum_{m\in\Z^3:|m-k|\loe 100}\chi_j\bx_k\chi_{l}\bx_m \wt\chi_l e^{it\De}P_Nf}_{ L_t^2 L_x^r l_{k}^2}\\
\lsm &\sqrt{p}\sum_{N\in2^\N}\sum_{j\goe 0} \sum_{l>j+5} N^s\normb{\sum_{m\in\Z^3:|m-k|\loe 100}\norm{\chi_j\bx_k\chi_{l}\bx_m \wt\chi_l e^{it\De}P_Nf}_{ L_t^2 L_x^r}}_{l_{k}^2}\\
\lsm &\sqrt{p}\sum_{N\in2^\N}\sum_{j\goe 0} \sum_{l>j+5}  N^s 2^{-Ml}\normb{\sum_{m\in\Z^3:|m-k|\loe 100}\norm{\bx_m \wt\chi_l e^{it\De}P_Nf}_{ L_{t,x}^2}}_{l_{k}^2}\\
\lsm &\sqrt{p}\sum_{N\in2^\N}\sum_{j\goe 0} \sum_{l>j+5} N^s2^{-Ml}\norm{\bx_k \wt\chi_l e^{it\De}P_Nf}_{ l_k^2 L_{t,x}^2}\\
\lsm &\sqrt{p} \sum_{N\in2^\N}\sum_{j\goe 0} \sum_{l>j+5} N^s2^{-Ml}\norm{\wt\chi_l e^{it\De}P_Nf}_{ L_{t,x}^2}\\
\lsm &\sqrt{p} \sum_{N\in2^\N}\sum_{j\goe 0} \sum_{l>j+5} N^{s-\half 1}2^{-\brko{M-\half 1}l}\norm{P_Nf}_{ L_{x}^2}\lsm\sqrt{p} \norm{f}_{H_x^{s-\half 1+\epsilon}}.
}
Similarly by Lemmas \ref{lem:mismatch},  \ref{lem:local-smoothing}, and Young's inequality in $k$,
\EQn{
\eqref{esti:nabla2infty-2-2}=&\sqrt{p} \sum_{N\in2^\N}\sum_{j\goe 0} \sum_{l>j+5} N^s\normb{\sum_{m\in\Z^3:|m-k|> 100}\chi_j\bx_k\chi_{l}\bx_m \wt\chi_l e^{it\De}P_Nf}_{ L_t^2 L_x^r l_{k}^2}\\
\lsm &\sqrt{p} \sum_{N\in2^\N}\sum_{j\goe 0} \sum_{l>j+5} N^s\normb{\sum_{m\in\Z^3:|m-k|> 100}\normb{\bx_k\chi_{l}\bx_m \wt \bx_m \wt\chi_l e^{it\De}P_Nf}_{ L_t^2 L_x^r}}_{l_{k}^2}\\
\lsm &\sqrt{p} \sum_{N\in2^\N}\sum_{j\goe 0} \sum_{l>j+5} N^s2^{-Ml}\normb{\sum_{m\in\Z^3:|m-k|> 100}|k-m|^{-M}\normb{\wt \bx_m\wt\chi_l e^{it\De}P_Nf}_{L_{t,x}^2}}_{l_{k}^2}\\
\lsm &\sqrt{p}\sum_{N\in2^\N}\sum_{j\goe 0} \sum_{l>j+5} N^s2^{-Ml}\normb{\normb{\wt \bx_m\wt\chi_l e^{it\De}P_Nf}_{L_{t,x}^2}}_{l_{m}^2}\\
\lsm &\sqrt{p} \sum_{N\in2^\N}\sum_{j\goe 0} \sum_{l>j+5} N^s2^{-Ml}\norm{\wt\chi_l e^{it\De}P_Nf}_{ L_{t,x}^2}\\
\lsm &\sqrt{p} \sum_{N\in2^\N}\sum_{j\goe 0} \sum_{l>j+5} N^{s-\half 1}2^{-\brko{M-\half 1}l}\norm{P_Nf}_{ L_{x}^2}\lsm \sqrt{p}\norm{f}_{H_x^{s-\half 1+\epsilon}}.
}
Therefore, we have
\EQn{\label{eq:nabla2infty-2}
\eqref{esti:nabla2infty-2}\lsm \sqrt{p}\norm{f}_{H_x^{s-\half 1+\epsilon}}.
}

By \eqref{eq:nabla2infty-1} and \eqref{eq:nabla2infty-2}, 
\EQ{
\norm{\jb{\nabla}^s e^{it\De}f^\om}_{L_\om^p L_t^2 L_x^r}\lsm& \sum_{N\in2^\N} N^{s}\norm{ e^{it\De}P_Nf^\om}_{ L_t^2 L_x^{r} L_\om^p}\\
\lsm&\eqref{esti:nabla2infty-1} + \eqref{esti:nabla2infty-2} \lsm \sqrt{p}\norm{f}_{H_x^{s-\half 1+2\epsilon}}.
}
Let $\epsilon\loe \rev2(s_0-s+\frac12)$, then we have \eqref{eq:nabla2infty-pro} holds for $r<\I$.

When $r=\I$, using the similar argument above with $r=\frac3\ep$, 
\EQ{
\norm{\jb{\nabla}^s e^{it\De}f^\om}_{L_\om^p L_t^2 L_x^\I}\lsm &\sum_{N\in2^\N} N^{s+\epsilon}\norm{ e^{it\De}P_Nf^\om}_{ L_t^2 L_x^{\frac3\epsilon} L_\om^p} \lsm \sqrt{p}\norm{f}_{H_x^{s-\half 1+3\epsilon}}.
}
Let $\epsilon\loe \rev3(s_0-s+\frac12)$, then we finish the proof of the $r=\I$ case.
\end{proof}

\vskip 1.5cm

\section{Local well-posedness}\label{sec:lwp}

\vskip .5cm

\subsection{Reduction to the deterministic problem}
Suppose that $u=v+w$ with $u_0=v_0+w_0$, and $v,w$ satisfy
\EQn{
	\label{eq:nls-w}
	\left\{ \aligned
	&i\pd_t w + \De w = |w+v|^2 (w+v), \\
	& v=e^{it\De}v_0, \\
	& w(0,x) = w_0(x).
	\endaligned
	\right.
}
This decomposition is referred as Bourgain's trick \cite{Bou94CMP} or Da Prato-Debussche's trick \cite{dPD02JFA}. For $l\in [0,1]$, define the working space
\EQn{\label{defn:X1/2-norm}
	\norm{w}_{X^l(I)}=\brkb{\sum_{N\in2^\N} N^{2l}\norm{w_N}_{U_\De^2(I;L_x^2)}^2}^{\frac12}.
}
Then, we have the local results for $H_x^{\frac16}$-data and $H_x^{\frac13+}$-data, separately:
\begin{prop}\label{prop:local-H1/2}
Let $\frac16\loe s<\frac12$, $v\in Y^s\cap Z^s(\R)$, and $w_0\in H_x^{\frac12}$. Then, there exists some $T>0$ depending on $s$, $w_0$, and $v_0$ such that there uniquely exists a solution $w$ of \eqref{eq:nls-w} on $[0,T]$ with
\EQ{
	w\in C([0,T];H_x^{\frac12})\cap X^{\frac12}([0,T]).
}
\end{prop}
%Next, we turn to the improved local result for $H_x^{\frac13+}$-data. %Define
%\EQ{
%\norm{w}_{X^{1}(I)}=\brkb{\sum_{N\in2^\N} N^2\norm{w_N}_{U_\De^2(I;L_x^2)}^2}^{\frac12}.
%}
%In the following, we write $U_\De^2=U_\De^2(I;L_x^2)$ and $V_\De^2=V_\De^2(I;L_x^2)$ for short. 
%Then our local result for $H_x^{\frac13+}$-data is
\begin{prop}\label{prop:local-H1}
	Let $\frac13< s\loe\frac12$, $v\in Y^s\cap Z^s(\R)$, and $w_0\in H_x^{1}$. Then, there exists some $T>0$ depending on $s$, $\norm{w_0}_{H_x^1}$, and $v_0$ such that there uniquely exists a solution $w$ of \eqref{eq:nls-w} on $[0,T]$ with
	\EQ{
		w\in C([0,T];H_x^{1})\cap X^1([0,T]).
	}
\end{prop}
In fact, the $s=\frac12$ case is trivial. However, we need it for the global result in Proposition \ref{prop:global-derterministic}.

Now, we give the proof of Theorem \ref{thm:local}, assuming that Propositions \ref{prop:local-H1/2} and \ref{prop:local-H1} hold.
\begin{proof}[Proof of Theorem \ref{thm:local}]
Let 
\EQ{
u(t)=e^{it\De}f^\om + w(t),
}
then $w$ satisfies the equation \eqref{eq:nls-w} with
\EQ{
v_0=f^\om\text{, }w_0=0\text{, and }v=e^{it\De}f^\om.
}

We first prove Theorem \ref{thm:local} $(1)$, using the result in Proposition \ref{prop:local-H1/2}. By Corollary \ref{cor:Y-norm-Z-norm}, we have for almost every $\om\in\Om$,
\EQ{
\norm{v}_{Y^s(\R)} + \norm{v}_{Z^s(\R)}<\I.
}
%Then for such $\omega$, there exists $T=T(\om,\norm{v}_{Y^s(\R)},\norm{v}_{Z^s(\R)})\in(0,1)$ such that
%\EQ{
%\norm{v}_{Y^s([0,T])}\loe \de\text{, and }T^{\frac{1}{6}} \norm{v}_{Z^s([0,T])}^3\lsm \de^3,
%}
%where $\de$ is defined in the statement of Proposition \ref{prop:local-H1/2}. 
Since $w_0=0$, we can apply Proposition \ref{prop:local-H1/2} to obtain the existence and uniqueness of $w\in C([0,T]; H_x^{\frac12})$ for almost every $\om\in\Om$.

The proof of Theorem \ref{thm:local} $(2)$ by  Proposition \ref{prop:local-H1} is similar as above, so we omit the details.
\end{proof}
\subsection{Proof of Proposition \ref{prop:local-H1/2}}
%In this section, we prove Proposition \ref{prop:local-H1/2}. 
%Recall that
%\EQ{
%	\norm{w}_{X^{\frac12}(I)}=\brkb{\sum_{N\in2^\N} N\norm{w_N}_{U_\De^2(I;L_x^2)}^2}^{\frac12}.
%}
We make the choices of some parameters and define the auxiliary working space:
\enu{
\item 
Let $C_0>0$ be the constant such that
\EQ{
	\norm{e^{it\De}w_0}_{X^{\frac12}([0,+\I))}\loe C_0\norm{w_0}_{H_x^{\frac12}}.
}
\item 
Let
\EQ{
	R:=\max\fbrk{C_0\norm{w_0}_{H_x^{\frac12}},1}.
}
\item
Let $\de$ and $\ep$ be some constants such that $0<\de,\ep\ll1$.
\item 
Define the following space:
\EQ{
	\norm{w}_{\wt X^{\frac12}(I)} = & \normb{\jb{\nabla}^{\frac13}w_N}_{l_N^2 L_t^2 L_x^{9}(2^\N\times I\times \R^3)} + \normb{\jb{\nabla}^{\frac16}w_N}_{l_N^2 L_t^4 L_x^{\frac92}(2^\N\times I\times \R^3)}\\
	& + \normb{\jb{\nabla}^{\frac16}w_N}_{l_N^2 L_t^3 L_x^{6}(2^\N\times I\times \R^3)} 
	+ \norm{w}_{L_t^4 L_x^6(I\times \R^3)} + \norm{w}_{L_t^3 L_x^6(I\times \R^3)}\\
	&+ \normb{\jb{\nabla}^{\frac12}w}_{L_t^2 L_x^6(I\times \R^3)} + \norm{w}_{L_t^{\frac{3}{1-\ep}}L_x^{\frac{6}{1-3\ep}}(I\times \R^3)}\\
	& + \norm{w}_{L_t^{\frac{4(4-3\ep)}{5+3\ep}}L_x^{\frac{2(4-3\ep)}{1-3\ep}}(I\times \R^3)}.
}
\item 
Let $T>0$ satisfy the smallness conditions 
\EQ{
	\norm{e^{it\De}w_0}_{\wt X^{\frac12}([0,T])} + \norm{v}_{Y^s([0,T])}\loe\de,
}
and
\EQ{
\de T^{\ep^2} \norm{v}_{Z^s(\R)}^{2} + T^{\rev{100}}R^{2} \norm{v}_{Z^s(\R)}\lsm \de^3.
}
}
We remark that
\EQ{
	X^{\frac12}([0,T])\hra \wt X^{\frac12}([0,T]),
} 
and $T$ depends on $s$, $\de$, $\ep$, $v_0$, and $w_0$. Let the working space be defined by
\EQ{
B_{R,\de,T}:=\fbrk{w\in C([0,T];H_x^{\frac12}): \norm{w}_{ X^{\frac12}([0,T])}\loe 2R\text{, } \norm{w}_{\wt X^{\frac12}([0,T])}\loe 2\de}.
}
Define
\EQ{
\Phi_{w_0,v}(w)=e^{it\De}w_0-i\int_0^t e^{i(t-s)\De}(|u|^2u)\ds.
}
Now, we are going to prove that $\Phi_{w_0,v}$ is a contraction mapping on $B_{R,\de,T}$, which is reduced to prove the following nonlinear estimate
\EQ{
\normb{\int_0^t e^{i(t-s)\De}(|u|^2u)\ds}_{ X^{\frac12}([0,T])}\loe \de.
}
In fact, we can use similar argument to prove
\EQ{
\normb{\Phi_{w_0,v}(w_1)-\Phi_{w_0,v}(w_2)}_{X^{\frac12}([0,T])} \loe \frac12 \norm{w_1-w_2}_{X^{\frac12}([0,T])},
}
and then finish the proof of contraction mapping.

Therefore, we reduce the proof of Proposition \ref{prop:local-H1/2} to the following lemma:
\begin{lem}
Let $\frac16\loe s<\frac12$, and $\de$, $\ep$, $C_0$, $R$, $T$ be defined as above. Assume that the following estimates hold:
\EQ{
	\norm{v}_{Y^s([0,T])}\loe \de\text{, }\norm{w}_{X^{\frac12}([0,T])}\lsm R\text{, and } \norm{w}_{\wt X^{\frac12}([0,T])}\lsm\de.
}
Then
\EQn{\label{eq:local-h1/2-main}
\normb{\int_0^t e^{i(t-s)\De}(|u|^2u)\ds}_{X^{\frac12}([0,T])}\loe \de.
}
\end{lem}
\begin{proof}
In the following, we shall slightly abuse notation and write $u$ for both itself and its complex conjugate, and  all the space-time norms are taken over $[0,T]\times\R^3$ without writing its integral region. First, to prove \eqref{eq:local-h1/2-main}, by Lemma \ref{lem:upvpduality}, H\"older's inequality, and embedding $V_\De^2\hra L_t^\I L_x^2$, we are reduced to consider
\EQn{\label{esti:local-h1/2-main-0.5}
	\normb{\int_0^t e^{i(t-s)\De}(|u|^2u)\ds}_{X^{\frac12}} 
	\lsm & \brkb{\sum_{N\in2^\N} N\sup_{\norm{g}_{V_\De^2}=1}\absb{\int_0^T\jb{P_N(|u|^2u),g}\dd t}^2}^{\frac12}\\
	\lsm & \brkb{\sum_{N\in2^\N} N\sup_{\norm{g}_{V_\De^2}=1}\norm{P_N(|u|^2u)}_{L_t^1 L_x^2}^2 \norm{g}_{L_t^\I L_x^2}^2}^{\frac12}\\
	\lsm & \normb{N^{\frac12}P_N(|u|^2u)}_{l_N^2L_t^1 L_x^2}.
}

Noting that
\EQ{
\absb{P_N(|u|^2u)} \lsm \absb{\sum_{N_1:N_1\gsm N} P_N\big(u_{N_1} u_{\loe N_1}^2\big)},
} 
by Cauchy-Schwarz's inequality in $N_1$ and Lemma \ref{lem:littlewood-paley}, 
\EQn{\label{esti:local-h1/2-main-1}
\normb{N^{\frac12}P_N(|u|^2u)}_{l_N^2L_t^1 L_x^2} 
\lsm &  \normb{\normb{\sum_{N_1:N_1\gsm N}\frac{N^{\frac12}}{N_1^{\frac12}}N_1^{\frac12}P_N(u_{N_1} u_{\loe N_1}^2)}_{l_N^2}}_{L_t^1 L_x^2} \\
\lsm & \normb{\normb{N_1^{\frac12}P_N(u_{N_1} u_{\loe N_1}^2)}_{l_N^2}}_{L_t^1 l_{N_1}^2L_x^2} \\
\lsm & \normb{N_1^{\frac12}u_{N_1} u_{\loe N_1}^2}_{L_t^1 L_x^2l_{N_1}^2} 
\lsm  I + II,
}
where we denote
\EQ{
I:= \normb{N_1^{\frac12}w_{N_1} u_{\loe N_1}^2}_{L_t^1 L_x^2 l_{N_1}^2}\text{, and } II:=\normb{N_1^{\frac12}v_{N_1} u_{\loe N_1}^2}_{L_t^1 L_x^2 l_{N_1}^2}.
}

First, we deal with the term $I$, where the $\frac12$-order derivative acts on $w$. This is the simpler case, since $w$ allows estimates with the derivative of order $\frac12$. By frequency support property, 
\EQnnsub{
I \lsm & \normb{N_1^{\frac12}w_{N_1} u_{\sim N_1} u_{\loe N_1}}_{L_t^1 L_x^2 l_{N_1}^2} 
\label{eq:local-h1/2-wterm-1high}\\
&+ \normb{N_1^{\frac12}w_{N_1} u_{\ll N_1}^2}_{L_t^1 L_x^2 l_{N_1}^2}.
\label{eq:local-h1/2-wterm-2low}
}
Now, we estimate \eqref{eq:local-h1/2-wterm-1high}. By H\"older's inequality and Lemma \ref{lem:schurtest}, it holds that
\EQn{\label{esti:local-h1/2-wterm-1high}
\eqref{eq:local-h1/2-wterm-1high} \lsm & \sum_{N_1}  N_1^{\frac13}\norm{w_{N_1}}_{L_t^2 L_x^{9}} N_1^{\frac16} \norm{u_{\sim N_1}}_{L_t^4 L_x^{\frac92}} \norm{u_{\loe N_1}}_{L_t^4 L_x^6}\\
\lsm & \normb{\jb{\nabla}^{\frac13}w_N}_{l_N^2 L_t^2 L_x^{9}} \normb{\jb{\nabla}^{\frac16} u_{\sim N}}_{l_{N}^2 L_t^4 L_x^{\frac92}} \norm{u}_{L_t^4 L_x^6}\\
\lsm & \norm{w}_{\wt X^{\frac12}}(\norm{w}_{\wt X^{\frac12}} + \norm{v}_{Y^s})^2.
}
For \eqref{eq:local-h1/2-wterm-2low}, by H\"older's inequality, Lemmas \ref{lem:linfty-littewoodpaley} and \ref{lem:littlewood-paley}, 
\EQn{\label{esti:local-h1/2-wterm-2low}
\eqref{eq:local-h1/2-wterm-2low} \lsm & \normb{N_1^{\frac12}w_{N_1}}_{L_t^2 L_x^6 l_{N_1}^2} \normb{\sup_{N_1} |u_{\ll N_1}|}_{L_t^4 L_x^6 }^2 \\
\lsm & \normb{\jb{\nabla}^{\frac12}w}_{L_t^2 L_x^6} \norm{u}_{L_t^4 L_x^{6}}^2
\lsm  \norm{w}_{\wt X^{\frac12}}(\norm{w}_{\wt X^{\frac12}} + \norm{v}_{Y^s})^2.
}
Then, by \eqref{esti:local-h1/2-wterm-1high} and \eqref{esti:local-h1/2-wterm-2low}, we have
\EQn{\label{esti:local-h1/2-main-2}
I \lsm  \eqref{eq:local-h1/2-wterm-1high} + \eqref{eq:local-h1/2-wterm-2low}\lsm \norm{w}_{\wt X^{\frac12}}(\norm{w}_{\wt X^{\frac12}} + \norm{v}_{Y^s})^2\lsm \de^3.
}

Next, we consider the term $II$, where the $\frac12$-order derivative acts on $v$. However, the function $v$ can only have $\frac16$-order derivative. Therefore, we need to transfer the additional fractional order derivative to other functions. We make a frequency decomposition:
\EQnnsub{
II\lsm & \normb{N_1^{\frac12}v_{N_1}u_{\sim N_1}^2}_{L_t^1 L_x^2 l_{N_1}^2}  \label{eq:local-h1/2-vterm-2high}\\
& + \normb{N_1^{\frac12}v_{N_1}v_{\ll N_1} u_{\loe N_1}}_{L_t^1 L_x^2 l_{N_1}^2} \label{eq:local-h1/2-vterm-1low-v}\\
& + \normb{N_1^{\frac12}v_{N_1}w_{\ll N_1} u_{\loe N_1}}_{L_t^1 L_x^2 l_{N_1}^2}. \label{eq:local-h1/2-vterm-1low-w}
}
By Lemma \ref{lem:littlewood-paley}, H\"older's inequality, and embedding $l_N^2\hra l_N^3$,
\EQn{\label{esti:local-h1/2-vterm-2high}
\eqref{eq:local-h1/2-vterm-2high}
\lsm & \sum_{N_1\in2^\N} N_1^{\frac16} \norm{v_{N_1}}_{L_t^3 L_x^6} \brkb{N_1^{\frac16}\norm{u_{\sim N_1}}_{L_t^3 L_x^6}}^2\\
\lsm & \normb{\jb{\nabla}^{\frac16}v_{N_1}}_{l_{N_1}^2 L_t^3 L_x^6} \normb{\jb{\nabla}^{\frac16}u_{\sim N_1}}_{l_{N_1}^2 L_t^3 L_x^6}^2\\
\lsm & \norm{v}_{Y^s}(\norm{w}_{\wt X^{\frac12}} + \norm{v}_{Y^s})^2\lsm \de^3.
}

Next, we consider \eqref{eq:local-h1/2-vterm-1low-v} and \eqref{eq:local-h1/2-vterm-1low-w}. The proof is more difficult, where we use the bilinear Strichartz estimate to transfer derivative. However, this approach will create the term $\norm{v}_{Z^s}$, which cannot get smallness by letting the interval small. Therefore, we also need some $T$ to control the $Z^s$-norm. 

Now, we consider the term \eqref{eq:local-h1/2-vterm-1low-v}. By H\"older's inequality,
\EQn{\label{esti:local-h1/2-vterm-1low-v-1}
	\eqref{eq:local-h1/2-vterm-1low-v} 
	\lsm & \sum_{N_2\ll N_1} N_1^{\frac12} \norm{v_{N_1}v_{N_2} u_{\loe N_1}}_{L_t^1 L_x^2}\\
	\lsm & \sum_{N_2\ll N_1} N_1^{\frac12} \norm{v_{N_1} v_{N_2}}_{L_t^{\frac1{1-\ep}}L_x^2}^{\frac23+\ep} \norm{v_{N_1}}_{L_{t,x}^\I}^{\frac13-\ep} \norm{v_{N_2}}_{L_{t,x}^\I}^{\frac13-\ep} \norm{u_{\loe N_1}}_{L_t^{\frac{3}{1-\ep +3\ep^2}}L_x^{\frac{6}{1-3\ep}}}.
}
By Lemma \ref{lem:bilinearstrichartz}, for $N_2\ll N_1$, 
\EQn{\label{esti:local-h1/2-vterm-1low-v-2-bi}
	\norm{v_{N_1} v_{N_2}}_{L_t^{\frac{1}{1-\ep}}L_x^2}\lsm  N_2^{2\ep}N_1^{-\frac12}\norm{P_{N_1}v_0}_{L_x^2}\norm{P_{N_2}v_0}_{L_x^2}
	\lsm  N_1^{-\frac23}\norm{v}_{Z^s}^2.
}
Note that
\EQn{\label{esti:local-h1/2-vterm-1low-v-2-strz}
\norm{u}_{L_t^{\frac{3}{1-\ep }}L_x^{\frac{6}{1-3\ep}}} \lsm \norm{w}_{\wt X^{\frac12}} + \norm{v}_{Y^s}\lsm\de.
}
By \eqref{esti:local-h1/2-vterm-1low-v-1}, \eqref{esti:local-h1/2-vterm-1low-v-2-bi}, \eqref{esti:local-h1/2-vterm-1low-v-2-strz}, and H\"older's inequality,
\EQn{\label{esti:local-h1/2-vterm-1low-v-2}
\eqref{eq:local-h1/2-vterm-1low-v} 
\lsm & \sum_{N_2\ll N_1} N_1^{\frac12} \norm{v_{N_1} v_{N_2}}_{L_t^{\frac1{1-\ep}}L_x^2}^{\frac23+\ep} \norm{v_{N_1}}_{L_{t,x}^\I}^{\frac13-\ep} \norm{v_{N_2}}_{L_{t,x}^\I}^{\frac13-\ep} \norm{u_{\loe N_1}}_{L_t^{\frac{3}{1-\ep +3\ep^2}}L_x^{\frac{6}{1-3\ep}}}\\
\lsm & T^{\ep^2}\sum_{N_2\ll N_1} N_1^{\frac{1}{18}-\frac23\ep} \norm{v}_{Z^s}^{\frac43+2\ep} N_1^{-\frac1{18}+\frac12\ep-\ep^2} \normb{\jb{\nabla}^{\frac16-\ep}v_{N_1}}_{L_{t,x}^\I}^{\frac13-\ep} \norm{v_{N_2}}_{L_{t,x}^\I}^{\frac13-\ep} \norm{u}_{L_t^{\frac{3}{1-\ep }}L_x^{\frac{6}{1-3\ep}}} \\
\lsm & \de T^{\ep^2}\norm{v}_{Z^s}^{\frac53+\ep}\sum_{N_2\ll N_1} N_1^{-\frac16\ep-\ep^2}  \norm{v_{N_2}}_{L_{t,x}^\I}^{\frac13-\ep}
\lsm  \de T^{\ep^2} \norm{v}_{Z^s}^{2}.
}

Finally, we consider the term \eqref{eq:local-h1/2-vterm-1low-w}. By H\"older's inequality,
\EQn{\label{esti:local-h1/2-vterm-1low-w-1}
	\eqref{eq:local-h1/2-vterm-1low-w} 
	\lsm & \sum_{N_2\ll N_1} N_1^{\frac12} \norm{v_{N_1}w_{N_2} u_{\loe N_1}}_{L_t^1 L_x^2}\\
	\lsm & \sum_{N_2\ll N_1} N_1^{\frac12} \norm{v_{N_1} w_{N_2}}_{L_t^{\frac43}L_x^2}^{\frac23+\ep} \norm{v_{N_1}}_{L_{t,x}^\I}^{\frac13-\ep} \norm{w_{N_2}}_{L_t^{q_1} L_x^{r_1}}^{\frac13-\ep} \norm{u_{\loe N_1}}_{L_t^{q_1} L_x^{r_1}},
}
where $(q_1,r_1)$ is defined by
\EQ{
\frac{2-3\ep}{4}=(\frac{4}{3}-\ep)\rev{q_1}\text{, and } \frac{1-3\ep}{6}=(\frac{4}{3}-\ep)\rev{r_1}.
}
Noting that $q_1=\frac{4(4-3\ep)}{3(2-3\ep)}=\frac83+$ and $r_1=\frac{2(4-3\ep)}{1-3\ep}=8+$, we have $\frac32-\frac{2}{q_1}-\frac3{r_1}=\frac38+$, and there exists $q_2=\frac{4(4-3\ep)}{5+3\ep}$ such that
\EQ{
\frac{2}{q_2}+\frac3{r_1}=1.
}
Then, we have
\EQn{\label{esti:local-h1/2-vterm-1low-w-2-strz}
\norm{u_{\loe N_1}}_{L_t^{q_2}L_x^{r_1}}\lsm\norm{w}_{L_t^{q_2}L_x^{r_1}} + \norm{v}_{L_t^{q_2}L_x^{r_1}}\lsm \de.
}
By Lemma \ref{lem:bilinearstrichartz}, for $N_2\ll N_1$,
\EQn{\label{esti:local-h1/2-vterm-1low-w-2-bi}
\norm{v_{N_1} w_{N_2}}_{L_t^{\frac43}L_x^2}\lsm & N_2^{\frac12}N_1^{-\frac12}\norm{P_{N_1}v_0}_{L_x^2}\norm{w_{N_2}}_{U_\De^2}\\
\lsm & N_1^{-\frac23}\norm{v}_{Z^s} \norm{w}_{X^{\frac12}}\lsm N_1^{-\frac23} R\norm{v}_{Z^s}.
}
We remark that for \eqref{esti:local-h1/2-vterm-1low-w-2-bi}, if we do not invoke the $U^p$-$V^p$ method, by Lemma \ref{lem:bilinearstrichartz}, it reduces to deal with the term
\EQ{
	N_1^{\frac12}\brkb{\norm{P_{N_1}w_0}_{L_x^2} + \norm{P_{N_1}(|u|^2u)}_{L_t^1 L_x^2}}.
}
Thus, the argument would be more complex, especially when $N_1^{1/2}$ acts on $|v|^2v$.

By \eqref{esti:local-h1/2-vterm-1low-w-1}, \eqref{esti:local-h1/2-vterm-1low-w-2-strz}, \eqref{esti:local-h1/2-vterm-1low-w-2-bi}, and H\"older's inequality in $t$,
\EQn{\label{esti:local-h1/2-vterm-1low-w-2}
	\eqref{eq:local-h1/2-vterm-1low-w} 
	\lsm & \sum_{N_2\ll N_1} N_1^{\frac12} \norm{v_{N_1} w_{N_2}}_{L_t^{\frac43}L_x^2}^{\frac23+\ep} \norm{v_{N_1}}_{L_{t,x}^\I}^{\frac13-\ep} \norm{w_{N_2}}_{L_t^{q_1} L_x^{r_1}}^{\frac13-\ep} \norm{u_{\loe N_1}}_{L_t^{q_1} L_x^{r_1}}\\
	\lsm & T^{(\rev3-\ep)(\rev{q_1}-\rev{q_2})} \sum_{N_2\ll N_1} N_1^{\frac{1}{18}-\frac23\ep} R^{\frac23+\ep} \norm{v}_{Z^s}^{\frac23+\ep} \norm{v_{N_1}}_{L_{t,x}^\I}^{\frac13-\ep} \norm{w}_{L_t^{q_2} L_x^{r_1}}^{\frac13-\ep} \norm{u}_{L_t^{q_2} L_x^{r_1}}\\
	\lsm & T^{\rev{100}}\de^{\frac43-\ep} R^{\frac23+\ep} \norm{v}_{Z^s}^{\frac23+\ep} \sum_{N_2\ll N_1} N_1^{-\frac16\ep-\ep^2}\normb{\jb{\nabla}^{\frac16-\ep}v}_{L_{t,x}^\I}^{\frac13-\ep}
	\lsm  T^{\rev{100}}R^{2} \norm{v}_{Z^s}.
}

Therefore, by \eqref{esti:local-h1/2-vterm-2high}, \eqref{esti:local-h1/2-vterm-1low-v-2}, and \eqref{esti:local-h1/2-vterm-1low-w-2}, 
\EQn{\label{esti:local-h1/2-main-3}
II\lsm \eqref{eq:local-h1/2-vterm-2high} + \eqref{eq:local-h1/2-vterm-1low-v} + \eqref{eq:local-h1/2-vterm-1low-w}\lsm \de^3 +  \de T^{\ep^2} \norm{v}_{Z^s}^{2} + T^{\rev{100}}R^{2} \norm{v}_{Z^s}.
}
Then, combining \eqref{esti:local-h1/2-main-1}, \eqref{esti:local-h1/2-main-2}, and \eqref{esti:local-h1/2-main-3}, we have
\EQ{
\normb{\int_0^t e^{i(t-s)\De}(|u|^2u)\ds}_{L_t^\I H_x^{\frac12}\cap X^{\frac12}}\loe & I + II \\
\loe & C\brk{\de^3 + \de T^{\ep^2} \norm{v}_{Z^s}^{2} + T^{\rev{100}}R^{2} \norm{v}_{Z^s} } \\
\loe & C\de^3\loe \de.
}
This completes the proof of the lemma.
\end{proof}
\subsection{Proof of Proposition \ref{prop:local-H1}}
%In this section, we prove Proposition \ref{prop:local-H1}. Recall that
%\EQ{
%\norm{w}_{X^{1}(I)}=\brkb{\sum_{N\in2^\N} N^2\norm{w_N}_{U_\De^2(I;L_x^2)}^2}^{\frac12}.
%}
In the following, we write $U_\De^2=U_\De^2(I;L_x^2)$ and $V_\De^2=V_\De^2(I;L_x^2)$ for short. 
We make the choices of some parameters:
\enu{
\item 
Let $C_0>0$ be the constant such that
\EQ{
	\norm{e^{it\De}w_0}_{X^{1}(\R)}\loe C_0\norm{w_0}_{H_x^{1}}.
}
\item 
Let $0<\ep<\frac{1}{100}(s-\frac13)$ and
\EQ{
	R:=\max\fbrk{C_0\norm{w_0}_{H_x^{1}},1}.
} 
\item 
Let $T>0$ satisfy the smallness conditions:
\EQn{\label{assume:local-h1-T}
	\norm{v}_{Y^s([0,T])}\loe R\text{, and }
	CT^{\frac{\ep}{100}}\brkb{\norm{v}_{Z^s(\R)}^3 + R^3}\loe\frac12R.
}
}
Note that $T$ depends on $s$, $\norm{w_0}_{H_x^1}$, $v$, and $\norm{v}_{Z^s(\R)}$. Let the working space be defined by
\EQ{
	B_{R,T}:=\fbrk{w\in C([0,T];H_x^{1}): \norm{w}_{X^1([0,T])}\loe 2R}.
}
Define
\EQ{
	\Phi_{w_0,v}(w)=e^{it\De}w_0-i\int_0^t e^{i(t-s)\De}(|u|^2u)\ds.
}
Similar as in the former section, in order to get that $\Phi_{w_0,v}$ is a contraction mapping on $B_{R,T}$, it suffices to prove:
\begin{lem}\label{lem:local-h1}
Let $\frac13<s\loe\frac12$, $0<\ep<\frac{1}{100}(s-\frac13)$, and $C_0,R$ be defined as above. Assume that $0<T<1$ satisfies the smallness condition \eqref{assume:local-h1-T}, and let 
\EQ{
	\norm{w}_{X^1([0,T])}\lsm R.
}
Then 
\EQn{\label{eq:local-h1-main}
	\normb{\int_0^t e^{i(t-s)\De}(|u|^2u)\ds}_{X^{1}([0,T])}\loe R.
}
\end{lem}
\begin{proof}
Again, we do not distinguish $u$ and $\wb u$, and all the space-time norms are restricted on $[0,T]\times \R^3$. By Lemma \ref{lem:upvpduality} and frequency decomposition, we have
\EQ{
	&\normb{\int_0^t e^{i(t-s)\De}(|u|^2u)\ds}_{X^1} \\
	\lsm & \brkb{\sum_{N\in2^\N} N^2\sup_{\norm{g}_{V_\De^2}=1}\absb{\int_0^T\jb{P_N(|u|^2u),g}\dd t}^2}^{\frac12}\\
	\lsm & \brkb{\sum_{N\in2^\N} N^2\sup_{\norm{g}_{V_\De^2}=1}\absb{\int_0^T\int P_N(\sum_{N_1}u_{N_1}u_{\loe N_1}^2)g\dx\dd t}^2}^{\frac12}
	\lsm I+II
}
where
\EQ{
I:=\brkb{\sum_{N\in2^\N} N^2\sup_{\norm{g}_{V_\De^2}=1}\absb{\int_0^T\int P_N(\sum_{N_1}w_{N_1}u_{\loe N_1}^2)g\dx\dd t}^2}^{\frac12},
}
and
\EQ{
II:=\brkb{\sum_{N\in2^\N} N^2\sup_{\norm{g}_{V_\De^2}=1}\absb{\int_0^T\int P_N(\sum_{N_1}v_{N_1}u_{\loe N_1}^2)g\dx\dd t}^2}^{\frac12}.
}

We first consider the term $I$, where the first order derivative acts on $w$. By H\"older's inequality and embedding $V_\De^2\hra L_t^\I L_x^2$, we have
\EQn{\label{esti:local-h1-I-1}
I\lsm & \brkb{\sum_{N\in2^\N} N^2\sup_{\norm{g}_{V_\De^2}=1}\normb{P_N(\sum_{N_1:N\lsm N_1}w_{N_1}u_{\loe N_1}^2)}_{L_t^1 L_x^2}^2\norm{g}_{L_t^\I L_x^2}^2}^{\frac12}\\
\lsm & \brkb{\sum_{N\in2^\N} N^2\sup_{\norm{g}_{V_\De^2}=1}\normb{P_N(\sum_{N_1:N\lsm N_1}w_{N_1}u_{\loe N_1}^2)}_{L_t^1 L_x^2}^2\norm{g}_{V_\De^2}^2}^{\frac12} \\
\lsm & \brkb{\sum_{N\in2^\N} N^2\normb{P_N(\sum_{N_1:N\lsm N_1}w_{N_1}u_{\loe N_1}^2)}_{L_t^1 L_x^2}^2}^{\frac12}.
}
Then, by \eqref{esti:local-h1-I-1}, Minkowski's inequality, H\"older's inequality in $N_1$, and Lemma \ref{lem:littlewood-paley},
\EQn{\label{esti:local-h1-I-2}
I\lsm & \brkb{\sum_{N\in2^\N} N^2\normb{P_N(\sum_{N_1:N\lsm N_1}w_{N_1}u_{\loe N_1}^2)}_{L_t^1 L_x^2}^2}^{\frac12}\\
\lsm &  \normb{\sum_{N_1:N\lsm N_1}\frac{N}{N_1}N_1P_N(w_{N_1}u_{\loe N_1}^2)}_{L_t^1 L_x^2l_N^2} \\
\lsm & \normb{N_1P_N(w_{N_1}u_{\loe N_1}^2)}_{L_t^1 L_x^2l_N^2l_{N_1}^2} 
\lsm  \normb{N_1w_{N_1}u_{\loe N_1}^2}_{L_t^1 l_{N_1}^2L_x^2}.
}
By \eqref{esti:local-h1-I-2}, H\"older's inequality, Lemmas \ref{lem:linfty-littewoodpaley} and \ref{lem:strichartz},
\EQn{\label{esti:local-h1-I-3}
I \lsm & \normb{N_1w_{N_1}u_{\loe N_1}^2}_{L_t^1 l_{N_1}^2L_x^2} \\
\lsm & \normb{\norm{N_1w_{N_1}}_{L_x^6}\norm{u_{\loe N_1}}_{L_x^6}^2}_{L_t^1 l_{N_1}^2}\\
\lsm & \norm{\jb{\nabla} w}_{L_t^2 L_x^6}\norm{u}_{L_t^4 L_x^6}^2\\
\lsm & T^{\frac{1}{2}}\norm{\jb{\nabla} w}_{L_t^2 L_x^6}\norm{u}_{L_t^\I L_x^6}^2\\
\lsm & T^{\frac{1}{2}} \norm{w}_{X^1}\brkb{\norm{v}_{Z^s} + \norm{w}_{X^1}}^{2}.
}
Therefore, by \eqref{esti:local-h1-I-3} and the choice of $T$, we have
\EQn{\label{esti:local-h1-I}
I\loe CT^{\frac{1}{2}} \norm{w}_{X^1}\brkb{\norm{v}_{Z^s} + \norm{w}_{X^1}}^{2}\loe C T^{\frac{\ep}{100}}(\norm{v}_{Z^s}^3 +R^3)\loe\frac12R.
}

We next consider the term $II$, where the first order derivative acts fully on $v$. However, $v$ can only have estimates with the derivative of order $\frac13$, thus there is a gap of $\frac23$-order derivative. Note that the bilinear Strichartz estimate can only lower down $\frac12$-order derivative. Therefore, this is the main case where we need to exploit the duality structure. To this end, we make a frequency decomposition:
\EQnnsub{
II\lsm & \brkb{\sum_{N\in2^\N} N^2\sup_{\norm{g}_{V_\De^2}=1}\absb{\int_0^T\int P_N(\sum_{N_1}v_{N_1}u_{\sim N_1}^2)g\dx\dd t}^2}^{\frac12} \label{eq:local-h1-nablav-2high}\\
&+ \brkb{\sum_{N\in2^\N} N^2\sup_{\norm{g}_{V_\De^2}=1}\absb{\int_0^T\int P_N(\sum_{N_1}v_{N_1}u_{\ll N_1}u_{\sim N_1})g\dx\dd t}^2}^{\frac12} \label{eq:local-h1-nablav-1high}\\
&+ \brkb{\sum_{N\in2^\N} N^2\sup_{\norm{g}_{V_\De^2}=1}\absb{\int_0^T\int P_N(\sum_{N_1}v_{N_1}u_{\ll N_1}^2)g\dx\dd t}^2}^{\frac12}. \label{eq:local-h1-nablav-2low}
}

We first estimate the \eqref{eq:local-h1-nablav-2high}. Using the same method as in \eqref{esti:local-h1-I-1} and \eqref{esti:local-h1-I-2}, 
\EQn{\label{esti:local-h1-nablav-2high-1}
\eqref{eq:local-h1-nablav-2high} \lsm & \brkb{\sum_{N\in2^\N} N^2\sup_{\norm{g}_{V_\De^2}=1}\absb{\int_0^T\int P_N(\sum_{N_1}v_{N_1}u_{\sim N_1}^2)g\dx\dd t}^2}^{\frac12} \\
\lsm & \sum_{N_1\in2^\N}\normb{N_1v_{N_1}u_{\sim N_1}^2}_{L_t^1 L_x^2}. 
}
Then, by \eqref{esti:local-h1-nablav-2high-1}, Lemma \ref{lem:strichartz}, H\"older's inequality in $N_1$, and $l_{N_1}^2\hra l_{N_1}^3$, we have
\EQn{\label{esti:local-h1-nablav-2high-2}
\eqref{eq:local-h1-nablav-2high}
\lsm & \sum_{N_1\in2^\N}\normb{N_1v_{N_1}u_{\sim N_1}^2}_{L_t^1 L_x^2} \\
\lsm & T^{\frac12}\sum_{N_1\in2^\N} \normb{\jb{\nabla}^{\frac13}v_{N_1}}_{L_{t,x}^6} \normb{\jb{\nabla}^{\frac13}u_{\sim N_1}}_{L_{t,x}^6} \normb{\jb{\nabla}^{\frac13}u_{\sim N_1}}_{L_{t,x}^6} \\
\lsm & T^{\frac12} \normb{\jb{\nabla}^{\frac13}v_N}_{l_N^2L_{t,x}^6} \normb{\jb{\nabla}^{\frac13}u_{N_1}}_{l_{N_1}^2L_{t,x}^6}^2 \\
\lsm & T^{\frac12} \norm{v}_{Y^s}\brkb{\norm{v}_{Y^s} + \norm{w}_{X^1}}^{2}\lsm T^{\frac12} R^3.
}

Next, we consider \eqref{eq:local-h1-nablav-1high}. Similar as above, by H\"older's inequality and Lemma \ref{lem:strichartz}, 
\EQn{\label{esti:local-h1-nablav-1high-1}
\eqref{eq:local-h1-nablav-1high}\lsm & \brkb{\sum_{N\in2^\N} N^2\sup_{\norm{g}_{V_\De^2}=1}\absb{\int_0^T\int P_N(\sum_{N_1}v_{N_1}u_{\ll N_1}u_{\sim N_1})g\dx\dd t}^2}^{\frac12} \\
\lsm & \brkb{\sum_{N\in2^\N} N^2\normb{ P_N(\sum_{N_1:N\lsm N_1}v_{N_1}u_{\ll N_1}u_{\sim N_1})}_{L_t^{1} L_x^{2}}^2 }^{\frac12} \\
\lsm & \sum_{N_1} N_1 \norm{v_{N_1}u_{\ll N_1}u_{\sim N_1}}_{L_t^{1} L_x^{2}}\\
\lsm & \sum_{N_1} N_1 \norm{v_{N_1}u_{\ll N_1}v_{\sim N_1}}_{L_t^{1} L_x^{2}}+ \sum_{N_1} N_1 \norm{v_{N_1}u_{\ll N_1}w_{\sim N_1}}_{L_t^{1} L_x^{2}}.
}
For the first term in the right hand side of \eqref{esti:local-h1-nablav-1high-1}, by H\"older's inequality,
\EQn{\label{esti:local-h1-nablav-1high-2-1}
\sum_{N_1} N_1 \norm{v_{N_1}u_{\ll N_1}v_{\sim N_1}}_{L_t^{1} L_x^{2}} \lsm & \sum_{N_1\gg N_2} N_1 \norm{v_{N_1}u_{N_2}}_{L_t^{1} L_x^{2}} \norm{v_{\sim N_1}}_{L_{t,x}^{\I}}.
}
By the bilinear Strichartz estimate in Lemma \ref{lem:bilinearstrichartz}, for $N_2\ll N_1$, 
\EQn{\label{esti:local-h1-nablav-1high-2-2}
\norm{v_{N_1}u_{N_2}}_{L_t^{\frac{12}{11}} L_x^{2}} \lsm & N_2^{\frac16} N_1^{-\frac12} \norm{P_{N_1}v_0}_{L_x^2} 
\brkb{\norm{P_{N_2}v_0}_{L_x^2} + \norm{P_{N_2}w}_{U_\De^2}} \\
\lsm & N_2^{-\frac16} N_1^{-\frac56} \normb{\jb{\nabla}^{\frac13}P_{N_1}v_0}_{L_x^2} 
\brkb{\normb{\jb{\nabla}^{\frac13}v_0}_{L_x^2} + N_2^{\frac13} \norm{P_{N_2}w}_{U_\De^2}} \\
\lsm & N_2^{-\frac16} N_1^{-\frac56} \norm{v}_{Z^s} 
\brkb{\norm{v}_{Z^s} +  \norm{w}_{X^1}},
}
which implies
\EQn{\label{esti:local-h1-nablav-1high-2-3}
& \sum_{N_1\gg N_2} N_1 \norm{v_{N_1}u_{N_2}}_{L_t^{1} L_x^{2}} \norm{v_{\sim N_1}}_{L_{t,x}^{\I}} \\
\lsm & T^{\frac1{12}} \sum_{N_1\gg N_2} N_1 \norm{v_{N_1}u_{N_2}}_{L_t^{\frac{12}{11}} L_x^{2}} \norm{v_{\sim N_1}}_{L_{t,x}^{\I}} \\
\lsm & T^{\frac1{12}} \sum_{N_1\gg N_2}   N_1^{\frac16} N_2^{-\frac16} \norm{v}_{Z^s} 
\brkb{\norm{v}_{Z^s} +  \norm{w}_{X^1}}N_1^{-\frac13+\ep} \normb{\jb{\nabla}^{\frac13-\ep}v_{\sim N_1}}_{L_{t,x}^{\I}} \\
\lsm & T^{\frac1{12}} \sum_{N_1\gg N_2}   N_1^{-\frac16+\ep} N_2^{-\frac16} \norm{v}_{Z^s}^2 
\brkb{\norm{v}_{Z^s} +  \norm{w}_{X^1}} \\
\lsm & T^{\frac1{12}}  \norm{v}_{Z^s}^2 
\brkb{\norm{v}_{Z^s} +  \norm{w}_{X^1}}.
}
Then,
\EQn{\label{esti:local-h1-nablav-1high-2-4}
\sum_{N_1} N_1 \norm{v_{N_1}u_{\ll N_1}v_{\sim N_1}}_{L_t^{1} L_x^{2}} \lsm T^{\frac1{12}}  
\brkb{\norm{v}_{Z^s}^3 + R^3}.
}
For the second term in the right hand side of \eqref{esti:local-h1-nablav-1high-1}, by H\"older's inequality,
\EQn{\label{esti:local-h1-nablav-1high-2-5}
\sum_{N_1} N_1 \norm{v_{N_1}u_{\ll N_1}w_{\sim N_1}}_{L_t^{1} L_x^{2}} \lsm & \sum_{N_1} N_1 \norm{v_{N_1}}_{L_{t,x}^\I} \norm{u_{\ll N_1}}_{L_t^1 L_x^\I} \norm{w_{\sim N_1}}_{L_t^\I L_x^2} \\
\lsm & T^{\frac12} \sum_{N_1} N_1^{-\frac13+\ep} \normb{\jb{\nabla}^{\frac13-\ep}v_{N_1}}_{L_{t,x}^\I} \norm{u}_{L_t^2 L_x^\I} N_1\norm{w_{\sim N_1}}_{U_\De^2}\\
\lsm & T^{\frac12} \sum_{N_1} N_1^{-\frac13+\ep} \norm{v}_{Z^s} \brkb{\norm{v}_{L_t^2 L_x^\I} + \norm{w}_{L_t^2L_x^\I}} \norm{w}_{X^1} \\
\lsm & T^{\frac12} \norm{v}_{Z^s} \brkb{\norm{v}_{Y^s} + \norm{w}_{X^1}} \norm{w}_{X^1} \\
\lsm & T^{\frac12}  \brkb{\norm{v}_{Z^s}^3 + R^3}.
}
Therefore, by \eqref{esti:local-h1-nablav-1high-2-4} and \eqref{esti:local-h1-nablav-1high-2-5},
\EQn{\label{esti:local-h1-nablav-1high-2}
\eqref{eq:local-h1-nablav-1high} \lsm & T^{\frac1{12}}  
\brkb{\norm{v}_{Z^s}^3 + R^3} + T^{\frac12}  \brkb{\norm{v}_{Z^s}^3 + R^3} \\
\lsm & T^{\frac1{12}}  
\brkb{\norm{v}_{Z^s}^3 + R^3}.
}

Finally, we consider the main term \eqref{eq:local-h1-nablav-2low}. By frequency support property, 
\EQn{\label{esti:local-h1-nablav-2low-1}
	\eqref{eq:local-h1-nablav-2low}= & \brkb{\sum_{N\in2^\N} N^2\sup_{\norm{g}_{V_\De^2}=1}\absb{\int_0^T\int P_N(\sum_{N_1}v_{N_1}u_{\ll N_1}^2)g\dx\dd t}^2}^{\frac12}\\
	\lsm & \sum_{N\in2^\N} N\sup_{\norm{g}_{V_\De^2}=1}\absb{\int_0^T\int \sum_{N_1}v_{N_1}u_{\ll N_1}^2g_{N}\dx\dd t}\\
	\lsm & \sum_{N\in2^\N} N\sup_{\norm{g}_{V_\De^2}=1}\absb{\int_0^T\int v_{N}u_{\ll N}^2 g_{N}\dx\dd t}\\
	\lsm & \sum_{N_1\loe N_2\ll N} N\sup_{\norm{g}_{V_\De^2}=1}\absb{\int_0^T\int v_{N}g_{N}u_{N_1}u_{N_2}\dx\dd t}.
}
To estimate this, we need to use the bilinear Strichartz estimate for both $g_N u_{N_1}$ and $\nabla v_N u_{N_2}$. Now, we give the estimate for $g_N u_{N_1}$, where we also need to pass $g_N$ into $V_\De^2$ by interpolation. By Lemma \ref{lem:bilinearstrichartz}, for $N_1\ll N$,  
\EQn{\label{esti:local-h1-nablav-2low-bilinear-1}
\norm{g_N u_{N_1}}_{L_t^{\frac{3}{2+\ep}} L_x^2}\lsm\frac{N_1^{\frac23-\frac23\ep}}{N^{\frac12}}\norm{g_N}_{U_\De^2} \brkb{\norm{P_{N_1}v_0}_{L_x^2} + \norm{w_{N_1}}_{U_\De^2}},
}
and by H\"older's inequality and Lemma \ref{lem:strichartz},
\EQn{\label{esti:local-h1-nablav-2low-bilinear-2}
\norm{g_N u_{N_1}}_{L_t^{\frac{3}{2+\ep}} L_x^2}\lsm & \norm{g_N}_{L_t^{\frac83}L_x^4} \norm{u_{N_1}}_{L_t^{\frac{24}{7+8\ep}} L_x^4}\\
\lsm & N_1^{\frac16-\frac23\ep} \norm{g_N}_{U_\De^{\frac83}} \brkb{\norm{P_{N_1}v_0}_{L_x^2} + \norm{w_{N_1}}_{U_\De^2}}.
}
Then, noting that $\ep<\rev{100}(s-\frac13)$, by \eqref{esti:local-h1-nablav-2low-bilinear-1}, \eqref{esti:local-h1-nablav-2low-bilinear-2}, and Lemma \ref{lem:upvp-interpolation}, for $N_1\ll N$,
\EQn{\label{esti:local-h1-nablav-2low-bilinear-3}
\norm{g_N u_{N_1}}_{L_t^{\frac{3}{2+\ep}} L_x^2}\lsm &\frac{N^{\ep}}{N_1^{\ep}}\frac{N_1^{\frac23-\frac23\ep}}{N^{\frac12}}\norm{g_N}_{V_\De^2} \brkb{\norm{P_{N_1}v_0}_{L_x^2} + \norm{w_{N_1}}_{U_\De^2}}\\
\lsm &\frac{N_1^{\frac13-\frac53\ep}}{N^{\frac12-\ep}} N_1^{-100\ep} \brkb{N_1^{\frac13+100\ep}\norm{P_{N_1}v_0}_{L_x^2} + N_1^{\frac13+100\ep} \norm{w_{N_1}}_{U_\De^2}}\\
\lsm & \frac{N_1^{\frac13-100\ep}}{N^{\frac12-\ep}} \brkb{\norm{v}_{Z^s} +  \norm{w}_{X^1}}.
}
Next, we give the estimate for $\nabla v_N u_{N_2}$. Using Lemma \ref{lem:bilinearstrichartz} again, for $N_2\ll N$, we also have
\EQn{\label{esti:local-h1-nablav-2low-bilinear-4}
\norm{\nabla v_N u_{N_2}}_{L_t^{\frac{1}{1-10\ep}} L_x^2}\lsm & N_2^{20\ep}N^{\frac12}\norm{P_N v_0}_{L_x^2}\brkb{\norm{P_{N_2} v_0}_{L_x^2} + \norm{w_{N_2}}_{U_\De^2}}\\
\lsm & N_2^{-\frac13}N^{\frac16-100\ep}\norm{ v}_{Z^s}\brkb{N_2^{\frac13+20\ep}\norm{P_{N_2} v_0}_{L_x^2} + N_2^{\frac13+20\ep} \norm{w_{N_2}}_{U_\De^2}}\\
\lsm & N_2^{-\frac13}N^{\frac16-100\ep}\norm{ v}_{Z^s}\brkb{\norm{v}_{Z^s} + \norm{w}_{X^1}}.
}
Then, we are ready to bound \eqref{eq:local-h1-nablav-2low}. By \eqref{esti:local-h1-nablav-2low-bilinear-3}, \eqref{esti:local-h1-nablav-2low-bilinear-4}, and H\"older's inequality,
\EQn{\label{esti:local-h1-nablav-2low-2}
&N\int_0^T\int g_{N}v_N u_{N_1} u_{N_2}\dx\dd t \\
\lsm &  \norm{g_N u_{N_1}}_{L_t^{\frac{3}{2+\ep}} L_x^2} \norm{\nabla v_N u_{N_2}}_{L_t^{\frac{1}{1-10\ep}} L_x^2}^{\frac13} \norm{\nabla v_N}_{L_t^{\frac{2}{9\ep}} L_x^\I}^{\frac23} \norm{u_{N_2}}_{L_t^\I L_x^2}^{\frac23} \\
\lsm & T^{\frac{\ep}{100}} \frac{N_1^{\frac13-100\ep}}{N^{\frac12-\ep}}  \brkb{N_2^{-\frac13}N^{\frac16-100\ep}}^{\frac13} N^{\frac23\brko{\frac23+\ep}} \normb{\jb{\nabla}^{\frac13-\ep} v_N}_{L_t^{\frac{1}{\ep}} L_x^\I}^{\frac23} \\
& \cdot N_2^{-\frac{2}{9}} \normb{\jb{\nabla}^{\frac13}u_{N_2}}_{L_t^\I L_x^2}^{\frac23} \brkb{\norm{v}_{Z^s} +  \norm{w}_{X^1}} \norm{ v}_{Z^s}^{\frac13}\brkb{\norm{v}_{Z^s} + \norm{w}_{X^1}}^{\frac13}\\
\lsm & T^{\frac{\ep}{100}} N_1^{\frac13-100\ep} N_2^{-\frac13} N^{-20\ep} \norm{v}_{Y^s}^{\frac23} \norm{ v}_{Z^s}^{\frac13} \brkb{\norm{v}_{Z^s}+\norm{w}_{X^1}}^{2} .
}
Note that 
\EQ{
\sum_{N_1\loe N_2\ll N} N_1^{\frac13-100\ep} N_2^{-\frac13} N^{-20\ep}\lsm 1,
}
then by \eqref{esti:local-h1-nablav-2low-1} and \eqref{esti:local-h1-nablav-2low-2},
\EQn{\label{esti:local-h1-nablav-2low}
\eqref{eq:local-h1-nablav-2low}\lsm & \sum_{N_1\loe N_2\ll N} N\sup_{\norm{g}_{V_\De^2}=1}\absb{\int_0^T\int v_{N}g_{N}u_{N_1}u_{N_2}\dx\ds}
 \\
\lsm & T^{\frac{\ep}{100}}\sum_{N_1\loe N_2\ll N} N_1^{\frac13-100\ep} N_2^{-\frac13} N^{-20\ep} \norm{v}_{Y^s}^{\frac23} \norm{ v}_{Z^s}^{\frac13} \brkb{\norm{v}_{Z^s}+\norm{w}_{X^1}}^{2}\\
\lsm & T^{\frac{\ep}{100}} R^{\frac23} \norm{ v}_{Z^s}^{\frac13}\brkb{\norm{v}_{Z^s} + R}^{2}\lsm T^{\frac{\ep}{100}} \brkb{\norm{v}_{Z^s}^{3} + R^{3}}.
}

Therefore, by \eqref{esti:local-h1-nablav-2high-2}, \eqref{esti:local-h1-nablav-1high-2}, and \eqref{esti:local-h1-nablav-2low}, noting also that $T<1$,
\EQn{\label{esti:local-h1-II}
	II\loe &  \eqref{eq:local-h1-nablav-2high} +  \eqref{eq:local-h1-nablav-1high} +  \eqref{eq:local-h1-nablav-2low} \\
	\loe &  C\brkb{T^{\frac12} + T^{\frac{1}{12}} + T^{\frac{\ep}{100}}} \brkb{\norm{v}_{Z^s}^{3} + R^{3}} \\
	\loe & C T^{\frac{\ep}{100}} \brkb{\norm{v}_{Z^s}^{3} + R^{3}}
	\loe  \frac12 R.
}
Then, by \eqref{esti:local-h1-I} and \eqref{esti:local-h1-II}, we have
\EQ{
	\normb{\int_0^t e^{i(t-s)\De}(|u|^2u)\ds}_{ X^{1}}\loe I + II \loe R.
}
This proves \eqref{eq:local-h1-main}.
\end{proof}

\vskip 1.5cm

\section{Global well-posedness and scattering}\label{sec:gwp}

\vskip .5cm

\subsection{Reduction to the deterministic problem}\label{sec:reduction-global}
Noting that $s>\frac{17}{40}$, fix a suitable small $\ep>0$ such that
\EQn{\label{defn:ep-gwp}
0<\ep<\min\fbrkb{\frac13(3s-1),\frac16(s-\frac{17}{40})}.
}
Let $\wt Y^s(I)$ be defined by its norm
\EQn{\label{defn:ys-tilde-norm}
\norm{v}_{\wt Y^s(I)}:=\norm{v}_{Y^s(I)} + \normb{\jb{\nabla}^{s+\frac12-\ep}v_N}_{l_N^2L_t^2 L_x^\I(2^\N\times I\times \R^3)}.
}
Recall that
\EQ{
	\norm{v}_{Z^s(I)}=\norm{\jb{\nabla}^{s-\ep}P_Nv}_{l_N^2L_t^\I L_x^\I(2^\N\times I\times\R^3)} + \norm{\jb{\nabla}^{s}P_Nv}_{l_N^2L_t^\I L_x^2(2^\N\times I\times\R^3)}.
}
\begin{prop}\label{prop:global-derterministic}
Let $\frac{17}{40}< s\loe\frac12$, $A>0$, and $\ep$ be sufficiently small satisfying \eqref{defn:ep-gwp}. Then, there exists $N_0=N_0(A)\gg1$ such that the following properties hold.  Let $u_0\in H_x^{s}$,   $v_0$ satisfy that $\supp\wh v_0\subset \fbrk{\xi\in\R^3:|\xi|\goe \frac12 N_0}$,  and $w_0=u_0-v_0$. Moreover, let $v=e^{it\Delta}v_0$ and $w=u-v$. Suppose that  $v\in \wt Y^s\cap Z^s(\R)$, $w_0\in H^1$ such that
\EQ{
	\norm{u_0}_{H_x^s}+\norm{v}_{\wt Y^s\cap Z^s(\R)}\loe A\text{, and }E(w_0) \loe A N_0^{2(1-s)}.
}
Then, there exists a solution $u$ of \eqref{eq:nls-3D} on $\R$ with $w\in C(\R;H_x^{1})$. Furthermore, there exists $u_{\pm}\in H_x^1$ such that
\EQ{
\lim_{t\ra\pm\I}\norm{u(t)-v(t)-e^{it\De}u_{\pm}}_{H_x^1}=0.
}
\end{prop}
We will give the proof of Proposition \ref{prop:global-derterministic} in Sections \ref{sec:space-time}, \ref{sec:energy-bound}, and \ref{sec:global-scattering}. Now we prove Theorem \ref{thm:global} assuming that Proposition \ref{prop:global-derterministic} holds.
\begin{proof}[Proof of Theorem \ref{thm:global}]
Let $N_0=N_0(M,\norm{f}_{H_x^s})\in 2^\N$ that will be defined later, and make a high-low frequency decomposition for the initial data
\EQ{
	u(t)=e^{it\De}P_{\goe N_0}f^\om + w(t),
}
then $w$ satisfies the equation \eqref{eq:nls-w} with
\EQ{
	u_0=f^\om\text{, }v_0=P_{\goe N_0}f^\om\text{, }w_0=P_{\loe N_0 }f^\om\text{, and }v=e^{it\De}P_{\goe N_0}f^\om.
}

Since $f$ is radial, by Corollary \ref{cor:Y-norm-Z-norm}, boundedness of the operator $P_{\goe N_0}$, Proposition \ref{prop:nabla2infty}, and Lemma \ref{lem:probability-estimate}, we have
\EQn{\label{esti:omM-1}
\PP\brkb{\fbrkb{\om\in\Om:\norm{u_0}_{H_x^s}+ \norm{v}_{\wt Y^s\cap Z^s(\R)}>\la}}\lsm e^{-C\la^2\norm{f}_{H_x^s}^{-2}}.
}

For any $p\goe 2$, we have
\EQ{
\norm{w_0}_{L_\om^p \dot H_x^1}\lsm & \normb{ \sum_{k\in\Z^3}g_k(\om) \nabla P_{\loe N_0}f_k}_{L_x^2 L_\om^p }\\
\lsm & \sqrt{p}\norm{\nabla P_{\loe N_0}f_k}_{L_x^2 l_{k\in\Z^3}^2 } \\
\lsm & \sqrt{p}\norm{\nabla P_{\loe N_0}f}_{L_x^2}\lsm \sqrt{p} N_0^{1-s} \norm{f}_{H_x^s}.
}
For any $p\goe 2$, we also have
\EQ{
	\norm{w_0}_{L_\om^p L_x^4}\lsm \sqrt{p} \norm{\bx_kP_{\loe N_0}f}_{L_x^4 l_k^2}\lsm \sqrt{p} \norm{f}_{L_x^2}.
}
Then, by Lemma \ref{lem:probability-estimate},
\EQn{\label{esti:omM-2}
\PP\brkb{\fbrkb{\om\in\Om: \frac{1}{N_0^{1-s}}\norm{w_0}_{ \dot H_x^1} + \norm{w_0}_{L_x^4}\goe\la}} \lsm e^{-C\la^2\norm{f}_{H_x^s}^{-2}}.
}

For any $M\goe 1$, let $\wt \Om_M$ be defined by
\EQn{\label{defn:omega-M}
\wt \Om_M=\Big\{\om\in\Om: &\norm{u_0}_{H_x^s} + \norm{v}_{\wt Y^s\cap Z^s(\R)}< M \norm{f}_{H_x^s},\\
&\frac{1}{N_0^{1-s}}\norm{w_0}_{ \dot H_x^1} + \norm{w_0}_{L_x^4}<M\norm{f}_{H_x^s}\Big\}.
}
Therefore, by \eqref{esti:omM-1} and \eqref{esti:omM-2}, we have
\EQn{\label{Om-Mc}
\PP(\wt \Om_M^c) \lsm e^{-CM^2}.
}
For any $\om \in\wt \Om_M$, we have $\norm{v}_{\wt Y^s\cap Z^s(\R)}< M \norm{f}_{H_x^s}$, and
\EQ{
E(w_0)\loe CM^2N_0^{2(1-s)}\cdot\max\fbrk{M^2\norm{f}_{H_x^s}^4,1}.
}

Therefore, for any $M>1$ and any $\om\in\wt \Om_M$, let 
\EQ{
A=A(M,\norm{f}_{H_x^s}):=\max\fbrkb{M \norm{f}_{H_x^s},CM^2\cdot\max\fbrko{M^2\norm{f}_{H_x^s}^4,1}},
}
then we have $v=e^{it\De}P_{\goe N_0}f^\om$,
\EQ{
\norm{u_0}_{H_x^s} +\norm{v}_{\wt Y^s\cap Z^s(\R)}\loe A\text{, and }E(w_0)\loe AN_0^{2(1-s)}.
} 
Therefore, we can apply Proposition \ref{prop:global-derterministic}. Let $N_0$ depend on $A$ as in the statement of Proposition \ref{prop:global-derterministic}, and we obtain a global solution $w$ that scatters. Then, for any $\om\in\wt \Om=\cup_{M>1}\wt \Om_M$, we can also derive that \eqref{eq:nls-w} admits a global solution $w$ that scatters. By \eqref{Om-Mc},  we have that  $\PP(\wt \Om)=1$. Then for almost every $\om\in\Om$, we obtain the global well-posedness and scattering for \eqref{eq:nls-w}. This finishes the proof of Theorem \ref{thm:global}.

\end{proof}

\subsection{Global space-time estimates}\label{sec:space-time}

\begin{lem}[Interaction Morawetz]\label{lem:inter-mora}
Let $w\in C([0,T];H_x^1)$ be the solution of perturbation equation \eqref{eq:nls-w}. Then, we have
\EQn{\label{eq:inter-mora}
\norm{w}_{L_{t,x}^4}^4 \lsm \norm{w}_{L_t^\I L_x^2}^2 \norm{w}_{L_t^\I \dot H_x^{\frac12}}^2 + \norm{\nabla w}_{L_t^\I L_x^2}^2 \norm{w}_{L_t^\I L_x^2}^4 \norm{v}_{L_t^2 L_x^\I}^2 + \norm{v}_{L_{t,x}^4}^4,
}
where all the space-time norms are taken over $[0,T]\times \R^3$.
\end{lem}
\begin{proof}
Recall that $w$ satisfies
\EQ{
i\pd_t w + \De w=|w|^2w+e,
}
where $e=|u|^2u-|w|^2w$.
Denote that 
$$
m(t,x)=\frac12|w(t,x)|^2;\quad 
p(t,x)=\frac12\im \brko{\wb w(t,x)\nabla w(t,x)}.
$$
Then, we have
\EQn{\label{eq:local-mass-flow}
\pd_tm=-2\nabla\cdot p+ \im\brk{e\bar w},
}
and
\EQn{\label{eq:local-momentum-flow}
\pd_tp = & -\re \nabla\cdot \brko{\nabla \wb w \nabla w}  -\frac14 \nabla \brkb{|w|^4} + \half 1 \nabla \De m
+ \re \brk{\wb e\nabla w}-\frac12\re\nabla\brk{\wb w e}.
}
Moreover, we note that 
$$
\partial_j\Big(\frac{x_k}{|x|}\Big)=\frac{\delta_{jk}}{|x|}-\frac{x_jx_k}{|x|^3};\quad
\nabla\cdot \frac{x}{|x|}=\frac2{|x|};\quad \De\nabla\cdot \frac{x}{|x|}=-8\pi\delta(x).
$$
Let 
\EQ{
M(t):= \int\!\!\int_{\R^{3+3}} \frac{x-y}{|x-y|}\cdot p(t,x)\> m(t,y)\dx\dy,
}
%By \eqref{eq:local-momentum-flow}, we have Morawetz identity
%\EQn{\label{eq:morawetz}
%&\pd_t\brkb{\half 1\int \frac{x-y}{|x-y|}\cdot\im\brk{\wb w(x)\nabla w(x)}\dx}\\
%= & \int\mbrkb{\abs{\nabla w}^2-\brkb{\frac{x-y}{|x-y|}\cdot\nabla w}^2}\dx + 2\pi |w(t,y)|^2\\
%& + \frac{1}{2}\int \frac{|w(x)|^{4}}{|x-y|}\dx
% +\re\int \frac{x-y}{|x-y|} e(x)\nabla \wb w(x)\dx\\
%& +\re\int \rev{|x-y|}\wb w(x)e(x)\dx.
%}
then by \eqref{eq:local-mass-flow} and \eqref{eq:local-momentum-flow}, we have the interaction Morawetz identity
\begin{subequations}
\EQnn{
 \pd_t M(t)
= &\iint_{\R^{3+3}} \frac{x-y}{|x-y|}\cdot \partial_t p(t,x)\> m(t,y)\dx\dy\nonumber\\
& \quad + \iint_{\R^{3+3}} \frac{x-y}{|x-y|}\cdot p(t,x)\> \partial_t m(t,y)\dx\dy\nonumber\\
= &\iint_{\R^{3+3}} \frac{x-y}{|x-y|}\cdot \Big( -\re \nabla\cdot \brko{\nabla \wb w \nabla w}  -\frac14 \nabla \brkb{|w|^4}\Big)(t,x)\> m(t,y)\dx\dy\label{esti:inter-mora-angular-1}\\
& \quad -2\iint_{\R^{3+3}} \frac{x-y}{|x-y|}\cdot p(t,x)\> \nabla\cdot p(t,y)\dx\dy\label{esti:inter-mora-angular-2}\\
&\quad +\frac12\iint_{\R^{3+3}} \frac{x-y}{|x-y|}\cdot \nabla \De m(t,x)\> m(t,y)\dx\dy\label{esti:inter-mora-angular-3}\\
&\quad +\iint_{\R^{3+3}} \frac{x-y}{|x-y|}\cdot  p(t,x)\>  \im\brk{e\bar w}(t,y)\dx\dy\label{esti:inter-mora-remainder-1}\\
& \quad + \iint_{\R^{3+3}} \frac{x-y}{|x-y|}\cdot \re \brk{\wb e\nabla w}(t,x)\> m(t,y)\dx\dy\label{esti:inter-mora-remainder-2}\\
& \quad + \iint_{\R^{3+3}} \frac1{|x-y|}\cdot \re \brk{\wb e w}(t,x)\> m(t,y)\dx\dy.\label{esti:inter-mora-remainder-3}
}
\end{subequations}
Note that by the classical argument in \cite{CKSTT04CPAM}, we have
\EQ{
\eqref{esti:inter-mora-angular-1} + \eqref{esti:inter-mora-angular-2} \goe 0,
}
and
\EQ{
\eqref{esti:inter-mora-angular-3}\gsm \norm{w(t)}_{L_x^4}^4.
}
Moreover,
\EQ{
\sup_{t\in[0,T]} M(t)\lsm \norm{w}_{L_t^\I L_x^2}^2 \norm{w}_{L_t^\I \dot H_x^{\frac12}}^2.
}
Then, integrating over $[0,T]$, it holds that
\EQ{
C\norm{w}_{L_{t,x}^4}^4\loe M(T)-M(0)+ \int_0^T |\eqref{esti:inter-mora-remainder-1}| + |\eqref{esti:inter-mora-remainder-2}| +|\eqref{esti:inter-mora-remainder-3}| \dd t,
}
thus
\EQ{
\norm{w}_{L_{t,x}^4}^4\lsm \norm{w}_{L_t^\I L_x^2}^2 \norm{w}_{L_t^\I \dot H_x^{\frac12}}^2+ \int_0^T |\eqref{esti:inter-mora-remainder-1}| + |\eqref{esti:inter-mora-remainder-2}| +|\eqref{esti:inter-mora-remainder-3}|\dd t.
}
By H\"older's inequality and Lemma \ref{lem:hardy}, we have
\EQ{
&\int_0^T|\eqref{esti:inter-mora-remainder-1}| + |\eqref{esti:inter-mora-remainder-2}| +|\eqref{esti:inter-mora-remainder-3}|\dd t\\
\lsm & \norm{\nabla w}_{L_t^\I L_x^2} \norm{w}_{L_t^\I L_x^2}^2 \norm{e}_{L_t^1 L_x^2}\\
\lsm & \norm{\nabla w}_{L_t^\I L_x^2} \norm{w}_{L_t^\I L_x^2}^2 \norm{v}_{L_t^2 L_x^\I}\brkb{\norm{v}_{L_{t,x}^4}^2 + \norm{w}_{L_{t,x}^4}^2}.
}

Therefore, we have
\EQ{
\norm{w}_{L_{t,x}^4}^4\lsm \norm{w}_{L_t^\I L_x^2}^2 \norm{w}_{L_t^\I \dot H_x^{\frac12}}^2 + \norm{\nabla w}_{L_t^\I L_x^2} \norm{w}_{L_t^\I L_x^2}^2 \norm{v}_{L_t^2 L_x^\I}\brkb{\norm{v}_{L_{t,x}^4}^2 + \norm{w}_{L_{t,x}^4}^2}.
}
Then, by Young's inequality, we have
\EQ{
\norm{w}_{L_{t,x}^4}^4\lsm \norm{w}_{L_t^\I L_x^2}^2 \norm{w}_{L_t^\I \dot H_x^{\frac12}}^2 + \norm{\nabla w}_{L_t^\I L_x^2}^2 \norm{w}_{L_t^\I L_x^2}^4 \norm{v}_{L_t^2 L_x^\I}^2 + \norm{v}_{L_{t,x}^4}^4.
}
This completes the proof of this lemma.
\end{proof}
However, the use of $L_{t,x}^4$-norm is not enough for our argument. 
First, we can update our space-time estimates to $H^\frac12$-level.
\begin{lem}[$H^\frac12$-regular promotion]\label{lem:transfer}
Let $0< \ep\ll1$ satisfy \eqref{defn:ep-gwp}, and $w\in C([0,T];H_x^1)$ be the solution of the perturbation equation \eqref{eq:nls-w}. Then, for any $0\loe l\loe \frac12$ and $L_x^2$-admissible $(q,r)$, 
\EQ{
\normb{N^{l} w_N}_{l_N^2L_t^qL_x^r} \lsm & \norm{w_0}_{H_x^{l}} + \brkb{\norm{w}_{L_t^\I  H_x^{l+\frac12}}+ \normb{\jb{\nabla}^{l}v_N}_{l_N^2L_t^2 L_x^{\I}}}\brkb{\norm{w}_{L_{t,x}^4}^2 +\norm{v}_{Y^s}^2},
}
where all the space-time norms are taken over $[0,T]\times \R^3$.
\end{lem}
\begin{proof}
By Minkowski's inequality, Lemmas \ref{lem:strichartz} and \ref{lem:littlewood-paley}, we have
\EQn{\label{esti:transfer-1}
\normb{N^{l} P_Nw}_{l_N^2L_t^qL_x^r} \lsm & \norm{w_0}_{\dot H_x^{l}} + \normb{P_N\jb{\nabla}^{l}(|u|^2u)}_{l_N^2L_t^2 L_x^{\frac65}+ l_N^2L_t^1 L_x^2}\\
\lsm & \norm{w_0}_{\dot H_x^{l}} + \normb{P_N\jb{\nabla}^{l}(|w|^2w)}_{L_t^2 L_x^{\frac65}l_N^2} \\
&+ \normb{P_N\jb{\nabla}^{l}(|u|^2u-|w|^2w)}_{l_N^2L_t^2 L_x^{\frac65}+ l_N^2L_t^1 L_x^2} \\
\lsm & \norm{w_0}_{\dot H_x^{l}} + \normb{\jb{\nabla}^{l}(|w|^2w)}_{L_t^2 L_x^{\frac65}} \\
&+ \normb{P_N\jb{\nabla}^{l}(|u|^2u-|w|^2w)}_{l_N^2L_t^2 L_x^{\frac65}+ l_N^2L_t^1 L_x^2} \\
\lsm & \norm{w_0}_{\dot H_x^{l}} + \normb{\jb{\nabla}^lw}_{L_t^\I L_x^3} \norm{w}_{L_{t,x}^4}^2 +I + II,
}
where
\EQ{
I:= \normb{P_N\jb{\nabla}^{l}\big(\sum_{N_1:N\lsm N_1}w_{N_1}v_{\loe N_1}O(u_{\loe N_1}+v_{\loe N_1})\big)}_{L_t^2 L_x^{\frac65}l_N^2}
}
and
\EQ{ II:=\normb{P_N\jb{\nabla}^{l}(\sum_{N_1:N\lsm N_1}v_{N_1}u_{\loe N_1} u_{\loe N_1}}_{L_t^1 L_x^2l_N^2}.
}
Here, $O(u_{\loe N_1}+v_{\loe N_1})$ denotes a linear combination of $u_{\loe N_1}$ and $v_{\loe N_1}$.
By H\"older's inequality, we have
\EQn{\label{esti:transfer-2}
I
\lsm & \sum_{N_1}\normb{\jb{\nabla}^{l}\big(w_{N_1}v_{\loe N_1}O(u_{\loe N_1}+v_{\loe N_1})\big)}_{L_t^2 L_x^{\frac65}} \\
\lsm & \sum_{N_1}\normb{N_1^{l}w_{N_1}v_{\loe N_1}O(u_{\loe N_1}+v_{\loe N_1})}_{L_t^2 L_x^{\frac65}} \\
\lsm & \sum_{N_1}\normb{N_1^{l}w_{N_1}}_{L_t^\I L_x^2} \norm{v_{\loe N_1}}_{L_t^4 L_x^{12}} \norm{O(u_{\loe N_1}+v_{\loe N_1})}_{L_{t,x}^4}\\
\lsm & \normb{\jb{\nabla}^{l+\frac12} w}_{L_t^\I L_x^2} \norm{v}_{Y^s} \brkb{ \norm{w}_{L_{t,x}^4} + \norm{v}_{Y^s}},
}
and by Lemma \ref{lem:linfty-littewoodpaley},
\EQn{\label{esti:transfer-3}
II\lsm & \normb{N_1^lv_{N_1}u_{\loe N_1}u_{\loe N_1}}_{L_t^1 L_x^2l_{N_1}^2} \\
\lsm & \normb{\normo{N_1^{l}v_{N_1}}_{l_{N_1}^2}\sup_{N_1\in2^\N}|u_{\loe N_1}|^2}_{L_t^1 L_x^{2}} \\
\lsm & \normb{N_1^{l}v_{N_1}}_{L_t^2 L_x^\I l_{N_1}^2}\normb{\sup_{N_1\in2^\N}|u_{\loe N_1}|}_{L_{t,x}^4}^2 \\
\lsm & \normb{N_1^{l}v_{N_1}}_{l_{N_1}^2L_t^2 L_x^\I} \brkb{\norm{w}_{L_{t,x}^4} +\norm{v}_{Y^s}}^2.
}
Therefore, by \eqref{esti:transfer-1}, \eqref{esti:transfer-2}, and \eqref{esti:transfer-3}, we have the desired estimates.
\end{proof}

Based on the space-time estimates in $H^\frac12$-level, and keeping in mind that the equation is $H^\frac12$-critical, we can further obtain the estimates in $H^l$-level with $l>\frac12$. We need larger class of space-time norms, and it is more convenient to invoke the $U^p$-$V^p$ method: recall that
\EQ{
	\norm{w}_{X^{l}(I)}=\brkb{\sum_{N\in2^\N} N^{2l}\norm{w_N}_{U_\De^2(I;L_x^2)}^2}^{\frac12}.
}
\begin{lem}[$H^1$-regular  promotion]\label{lem:transfer-2}
Let $0< \ep\ll1$ satisfy \eqref{defn:ep-gwp}, and $w\in C([0,T];H_x^1)$ be the solution of perturbation equation \eqref{eq:nls-w}. For any $\frac13<s\loe\frac12$, we have
\EQ{
\norm{w}_{X^1}
\lsm & \norm{w_0}_{H_x^{1}} + \norm{w}_{L_t^\I H_x^1}\brkb{\norm{v}_{Y^s\cap Z^s} + \normb{N^{\frac12}w_{N}}_{l_{N}^2L_t^2 L_x^6}} \brkb{\norm{v}_{Y^s} + \norm{w}_{L_t^6 L_x^{\frac92}}}\\
& + \norm{w}_{L_t^\I H_x^1}\brkb{\norm{v}_{Y^s} + \normb{N^{\frac12-3\ep}w_{N}}_{l_{N}^2L_t^2 L_x^6}} \brkb{\norm{v}_{Y^s} + \norm{\jb{\nabla}^{3\ep}w}_{L_t^{\frac{6}{1-6\ep}} L_x^{\frac{9}{2+6\ep}}}}\\
& + \norm{v}_{\wt Y^s} \brkb{\norm{v}_{Y^s} + \norm{N^sw_{ N}}_{l_{N}^2L_{t}^2 L_x^6}} \brkb{\norm{v}_{Z^s} + \norm{w}_{L_t^\I H_x^{\frac12}}} \\
& + \norm{v}_{\wt Y^s} \brkb{\norm{v}_{ Z^s} +\norm{w}_{X^{s}}}\brkb{\norm{v}_{Z^s} +\norm{w}_{L_t^\I H_x^s}},
}
where all the space-time norms are taken over $[0,T]\times \R^3$.
\end{lem}
\begin{proof}
Similar to the proof of Lemma \ref{lem:transfer}, we have
\EQ{
\norm{w}_{X^1} \lsm & \norm{w_0}_{H_x^{1}} + \brkb{\sum_{N\in2^\N} N^{2}\sup_{\norm{g}_{V_\De^2}=1}\absb{\int_0^T\int P_N(|u|^2u)g\dx\dd t}^2}^{\frac12}\\
\lsm & \norm{w_0}_{H_x^1} + \brkb{\sum_{N\in2^\N} N^{2}\sup_{\norm{g}_{V_\De^2}=1}\absb{\int_0^T\int P_N(\sum_{N_1}u_{N_1}u_{\loe N_1}^2)g\dx\dd t}^2}^{\frac12}\\
\lsm & \norm{w_0}_{H_x^{l}} +I + II,
}
where
\EQ{
	I:=\brkb{\sum_{N\in2^\N} N^{2}\sup_{\norm{g}_{V_\De^2}=1}\absb{\int_0^T\int P_N(\sum_{N_1}w_{N_1}u_{\loe N_1}^2)g\dx\dd t}^2}^{\frac12},
}
and
\EQ{
	II:=\brkb{\sum_{N\in2^\N} N^{2}\sup_{\norm{g}_{V_\De^2}=1}\absb{\int_0^T\int P_N(\sum_{N_1}v_{N_1}u_{\loe N_1}^2)g\dx\dd t}^2}^{\frac12}.
}

We first consider the term $I$, where the first order derivative acts on $w$. By H\"older's inequality, and Lemmas  \ref{lem:strichartz} and \ref{lem:littlewood-paley}, 
\EQnnsub{
	I\lsm &  \brkb{\sum_{N\in2^\N} N^{2}\normb{P_N(\sum_{N_1:N\lsm N_1}w_{N_1}u_{\loe N_1}^2)}_{L_t^{\frac32} L_x^{\frac{18}{13}}+ L_t^{\frac{3}{2-3\ep}} L_x^{\frac{18}{13+12\ep}}}^2}^{\frac12} \nonumber\\
	\lsm & \normb{\jb{\nabla}(\sum_{N_1}w_{N_1}u_{\loe N_1}^2)}_{L_t^{\frac32} L_x^{\frac{18}{13}} + L_t^{\frac{3}{2-3\ep}} L_x^{\frac{18}{13+12\ep}}} \nonumber\\
	\lsm & \normb{\jb{\nabla}(\sum_{N_1}w_{N_1}u_{\sim N_1} u_{\loe N_1})}_{L_t^{\frac32} L_x^{\frac{18}{13}}} \label{eq:transfer-hl-I-3}\\
	& + \normb{\jb{\nabla}(\sum_{N_1}w_{N_1}u_{\ll N_1}^2)}_{L_t^{\frac{3}{2-3\ep}} L_x^{\frac{18}{13+12\ep}}}.\label{eq:transfer-hl-I-4}
}
Now, the main task is to update the summation of $w_{N_1}$ to $l_{N_1}^2$. To this end, for \eqref{eq:transfer-hl-I-3}, we can simply use H\"older's inequality in $N_1$ for $w_{N_1}$ and $u_{\sim N_1}$. Precisely, by \eqref{esti:local-h1-I-2}, H\"older's inequality, and Lemma \ref{lem:littlewood-paley},
\EQn{\label{esti:transfer-hl-I-3}
\eqref{eq:transfer-hl-I-3} \lsm & \normb{\sum_{N_1}\normb{N_1w_{N_1}u_{\sim N_1} u_{\loe N_1}}_{L_{x}^{\frac{18}{13}}}}_{L_{t}^{\frac32}}\\
\lsm & \normb{\sum_{N_1}\normb{N_1w_{N_1}}_{L_x^2}\norm{u_{\sim N_1}}_{L_x^{\I}} \norm{u_{\loe N_1}}_{L_x^{\frac92}}}_{L_{t}^{\frac32}} \\
\lsm & \normb{\normb{N_1w_{N_1}}_{l_{N_1}^2L_x^2}\norm{u_{\sim N_1}}_{l_{N_1}^2L_x^{\I}} \norm{u}_{L_x^{\frac92}}}_{L_{t}^{\frac32}} \\
\lsm & \normb{\jb{\nabla}w}_{L_t^\I L_x^2} \norm{u_{N_1}}_{l_{N_1}^2L_t^{2} L_x^{\I}} \norm{u}_{L_t^{6} L_x^{\frac92}}\\
\lsm & \norm{w}_{L_t^\I H_x^1} \brkb{\norm{v}_{Y^s} + \normb{N^{\frac12}w_{N}}_{l_{N}^2L_t^2 L_x^6} }\brkb{\norm{v}_{Y^s} + \norm{w}_{L_t^6 L_x^{\frac92}}}.
}
For the second term \eqref{eq:transfer-hl-I-4}, we need to invoke the vector-valued Hardy-Littlewood maximal function to cover the critical summation problem.  Using Lemma \ref{lem:HL-boundedness} and H\"older's inequality,
\EQn{\label{esti:transfer-hl-I-4}
\eqref{eq:transfer-hl-I-4}
\lsm & \normb{NP_N\brkb{\sum_{N_1:N_1\sim N}w_{N_1}u_{\ll N_1}^2}}_{L_t^{\frac{3}{2-3\ep}} L_x^{\frac{18}{13+12\ep}}l_N^2} \\
\lsm & \normb{N\M\brkb{w_{N}u_{\ll N}^2}}_{L_t^{\frac{3}{2-3\ep}} L_x^{\frac{18}{13+12\ep}}l_N^2} \\
\lsm & \normb{Nw_{N}u_{\ll N}^2}_{L_t^{\frac{3}{2-3\ep}} L_x^{\frac{18}{13+12\ep}}l_N^2} \\
\lsm & \normb{\norm{Nw_{N}}_{l_N^2}\sup_{N}\absb{u_{\ll N}}^2}_{L_t^{\frac{3}{2-3\ep}} L_x^{\frac{18}{13+12\ep}}} \\
\lsm & \norm{Nw_{N}}_{L_t^\I L_x^2l_N^2} \normb{\sup_{N}\absb{u_{\ll N}}^2}_{L_t^{\frac{3}{2-3\ep}} L_x^{\frac{9}{2+6\ep}}}.
}
By Lemmas \ref{lem:schurtest}, \ref{lem:littlewood-paley}, and H\"older's inequality,
\EQn{\label{esti:transfer-hl-I-5}
\normb{\sup_{N}\absb{u_{\ll N}}^2}_{L_t^{\frac{3}{2-3\ep}} L_x^{\frac{9}{2+6\ep}}} \lsm & \normb{\sup_{N}\sum_{N_1,N_2:N_1\loe N_2\ll N}\absb{u_{N_1}u_{N_2}}}_{L_t^{\frac{3}{2-3\ep}} L_x^{\frac{9}{2+6\ep}}} \\
\lsm & \normb{\sum_{N_1\loe N_2}\absb{u_{N_1}u_{N_2}}}_{L_t^{\frac{3}{2-3\ep}} L_x^{\frac{9}{2+6\ep}}} \\
\lsm & \normb{\sum_{N_1\loe N_2} \brkb{\frac{N_1}{N_2}}^{3\ep} N_1^{-3\ep}u_{N_1}N_2^{3\ep}u_{N_2}}_{L_t^{\frac{3}{2-3\ep}} L_x^{\frac{9}{2+6\ep}}}\\
\lsm & \normb{ \norm{N_1^{-3\ep}u_{N_1}}_{l_{N_1}^2}\norm{N_2^{3\ep}u_{N_2}}_{l_{N_2}^2}}_{L_t^{\frac{3}{2-3\ep}} L_x^{\frac{9}{2+6\ep}}}\\
\lsm & \norm{N_1^{-3\ep}u_{N_1}}_{L_t^2 L_x^\I l_{N_1}^2} \norm{N_2^{3\ep}u_{N_2}}_{L_t^{\frac{6}{1-6\ep}} L_x^{\frac{9}{2+6\ep}}l_{N_2}^2} \\
\lsm & \brkb{\norm{v_{N_1}}_{l_{N_1}^2L_t^2 L_x^\I} + \norm{N_1^{-3\ep}w_{N_1}}_{l_{N_1}^2L_t^2 L_x^\I}} \\
& \cdot\brkb{\norm{v}_{Y^s} + \norm{\jb{\nabla}^{3\ep}w}_{L_t^{\frac{6}{1-6\ep}} L_x^{\frac{9}{2+6\ep}}}} \\
\lsm & \brkb{\norm{v}_{Y^s} + \normb{N_1^{\frac12-3\ep}w_{N_1}}_{l_{N_1}^2L_t^2 L_x^6}}\\
&\cdot \brkb{\norm{v}_{Y^s} + \norm{\jb{\nabla}^{3\ep}w}_{L_t^{\frac{6}{1-6\ep}} L_x^{\frac{9}{2+6\ep}}}}.
}
Combining \eqref{esti:transfer-hl-I-4} and \eqref{esti:transfer-hl-I-5}, we have
\EQn{\label{esti:transfer-hl-I-6}
\eqref{eq:transfer-hl-I-4}\lsm & \norm{Nw_{N}}_{L_t^\I L_x^2l_N^2} \normb{\sup_{N}\absb{u_{\ll N}}^2}_{L_t^{\frac{3}{2-3\ep}} L_x^{\frac{9}{2+6\ep}}} \\
\lsm & \norm{w}_{L_t^\I H_x^1} \brkb{\norm{v}_{Y^s} + \normb{N_1^{\frac12-3\ep}w_{N_1}}_{l_{N_1}^2L_t^2 L_x^6}} \brkb{\norm{v}_{Y^s} + \norm{\jb{\nabla}^{3\ep}w}_{L_t^{\frac{6}{1-6\ep}} L_x^{\frac{9}{2+6\ep}}}}.
}
Therefore, by \eqref{esti:transfer-hl-I-3} and \eqref{esti:transfer-hl-I-6},
\EQn{\label{esti:transfer-hl-I}
I\lsm& \eqref{eq:transfer-hl-I-3} + \eqref{eq:transfer-hl-I-4}\\ \lsm& \norm{w}_{L_t^\I H_x^1}\brkb{\norm{v}_{Y^s} + \normb{N^{\frac12}w_{N}}_{l_{N}^2L_t^2 L_x^6}} \brkb{\norm{v}_{Y^s} + \norm{w}_{L_t^6 L_x^{\frac92}}}\\
& + \norm{w}_{L_t^\I H_x^1}\brkb{\norm{v}_{Y^s} + \normb{N^{\frac12-3\ep}w_{N}}_{l_{N}^2L_t^2 L_x^6}} \brkb{\norm{v}_{Y^s} + \norm{\jb{\nabla}^{3\ep}w}_{L_t^{\frac{6}{1-6\ep}} L_x^{\frac{9}{2+6\ep}}}}.
}

Next, we estimate the term $II$,
%We first consider the case when $0\loe l<s+\frac12$, by  the same argument as in the local results (see \eqref{esti:local-h1/2-main-1} for example) and H\"older's inequality,
%\EQn{\label{esti:transfer-hl-II-case1}
%II \lsm & \normb{ N_1^1v_{N_1}u_{\loe N_1}^2}_{L_t^1 L_x^2l_{N_1}^2} \\
%\lsm & \normb{\jb{\nabla}v_{N_1}}_{L_t^2 L_x^\I l_{N_1}^2} \normb{\sup_{N_1}|u_{\loe N_1}|^2}_{L_{t,x}^2} \\
%\lsm & \normb{\jb{\nabla}v_{N_1}}_{l_{N_1}^2L_t^2 L_x^\I} \brkb{\norm{v}_{L_t^2 L_x^6} \norm{v}_{L_t^\I L_x^{3}} + \norm{w}_{L_t^2 L_x^6} \norm{w}_{L_t^\I L_x^{3}}}\\
%\lsm & \norm{v}_{\wt Y^s} \brkb{\norm{v}_{Y^s}\norm{v}_{Z^s} + \normb{w}_{L_t^2 L_x^6}\norm{w}_{L_t^\I H_x^{\frac12}}}\\
%\lsm & \norm{v}_{\wt Y^s} \brkb{\norm{v}_{Y^s}\norm{v}_{Z^s} + \norm{w_{N_1}}_{l_{N_1}^2L_t^2L_x^6}\norm{w}_{L_t^\I H_x^{\frac12}}} \\
%\lsm & \norm{v}_{\wt Y^s} \brkb{\norm{v}_{Y^s} +\norm{N^sw_{N}}_{l_{N}^2L_t^2L_x^6}}\brkb{\norm{v}_{Z^s} +\norm{w}_{L_t^\I H_x^{\frac12}}}.
%}
%Then, we consider the case when $s+\frac12\loe l\loe1$, 
where we need to use the bilinear Strichartz estimate and the duality structure. By frequency support property, we obtain
\EQnnsub{
II\lsm & \normb{ N_1v_{N_1}u_{\sim N_1}u_{\loe N_1}}_{L_t^1 L_x^2 l_{N_1}^2} \label{eq:transfer-hl-II-1}\\
& + \brkb{\sum_{N\in2^\N} N^{2}\sup_{\norm{g}_{V_\De^2}=1}\absb{\int_0^T\int P_N(\sum_{N_1}v_{N_1}u_{\ll N_1}^2)g\dx\dd t}^2}^{\frac12}. \label{eq:transfer-hl-II-2}
}
Now, we estimate \eqref{eq:transfer-hl-II-1}. Then, by $l_{N_1}^1\hra l_{N_1}^2$ and H\"older's inequality, 
\EQn{\label{esti:transfer-hl-II-1}
\eqref{eq:transfer-hl-II-1}\lsm & \sum_{N_1} \normb{\jb{\nabla}^{1-s}v_{N_1}}_{L_t^2 L_x^\I} \norm{\jb{\nabla}^su_{\sim N_1}}_{L_{t}^2 L_x^6} \norm{u_{\loe N_1}}_{L_t^\I L_x^3}\\
\lsm & \normb{\jb{\nabla}^{1-s}v_{N_1}}_{l_{N_1}^2L_t^2 L_x^\I} \norm{\jb{\nabla}^su_{\sim N_1}}_{l_{N_1}^2L_{t}^2 L_x^6} \brkb{\norm{v}_{L_t^\I L_x^3}+\norm{w}_{L_t^\I H_x^{\frac12}}}\\
\lsm & \normb{\jb{\nabla}^{s+\frac12-\ep}v_{N_1}}_{l_{N_1}^2L_t^2 L_x^\I} \brkb{\norm{\jb{\nabla}^sv_{\sim N_1}}_{l_{N_1}^2L_{t}^2 L_x^6} + \norm{\jb{\nabla}^sw_{\sim N_1}}_{l_{N_1}^2L_{t}^2 L_x^6}} \\
&\cdot\brkb{\norm{v}_{L_t^\I L_x^3}+\norm{w}_{L_t^\I H_x^{\frac12}}} \\
\lsm & \norm{v}_{\wt Y^s} \brkb{\norm{v}_{Y^s} + \norm{N_1^sw_{ N_1}}_{l_{N_1}^2L_{t}^2 L_x^6}} \brkb{\norm{v}_{Z^s} + \norm{w}_{L_t^\I H_x^{\frac12}}}.
}
Next, we consider \eqref{eq:transfer-hl-II-2}. To this end, we establish a bilinear Strichartz estimate before the proof for \eqref{eq:transfer-hl-II-2}. By Lemma \ref{lem:bilinearstrichartz}, for $N_1\ll N$,
\EQn{\label{esti:transfer-hl-II-2-bi-1}
\norm{u_{N_1}g_N}_{L_{t,x}^2} \lsm \frac{N_1}{N^{\frac12}} \brkb{\norm{P_{N_1}v_0}_{L_x^2} + \norm{w_{N_1}}_{U_\De^2}} \norm{g}_{U_\De^2},
}
and by Bernstein's, H\"older's inequalities, Lemma \ref{lem:strichartz}, and embedding $U_\De^4\hra L_t^\I L_x^2$,
\EQn{\label{esti:transfer-hl-II-2-bi-2}
\norm{u_{N_1}g_N}_{L_{t,x}^2} \lsm & \norm{u_{N_1}}_{L_t^2 L_x^\I} \norm{g_N}_{L_t^\I L_x^2} \\
\lsm & N_1^{\frac12}\norm{u_{N_1}}_{L_t^2 L_x^6} \norm{g}_{U_\De^4}\\
\lsm & N_1^{\frac12} \brkb{\norm{P_{N_1}v_0}_{L_x^2} + \norm{w_{N_1}}_{U_\De^2}} \norm{g}_{U_\De^4}.
}
By \eqref{esti:transfer-hl-II-2-bi-1}, \eqref{esti:transfer-hl-II-2-bi-2}, and Lemma \ref{lem:upvp-interpolation}, we have the bilinear Strichartz estimate
\EQn{\label{esti:transfer-hl-II-2-bi}
\norm{u_{N_1}g_N}_{L_{t,x}^2} \lsm & \frac{N^\ep}{N_1^\ep}\frac{N_1}{N^{\frac12}} \brkb{\norm{P_{N_1}v_0}_{L_x^2} + \norm{w_{N_1}}_{U_\De^2}} \norm{g}_{V_\De^2} \\
\lsm & \frac{N_1^{1-s+\ep}}{N^{\frac12-s+\ep}} \brkb{\norm{P_{N_1}v_0}_{L_x^2} + \norm{w_{N_1}}_{U_\De^2}} \norm{g}_{V_\De^2}.
}
Noting that  $N_1\loe N_2$, by the choice of $\ep$ in \eqref{defn:ep-gwp}, we have $N_1^{1-s+\ep}<N_1^{2s-2\ep}<N_1^{s-\ep}N_2^{s-\ep}$. Then, by \eqref{esti:transfer-hl-II-2-bi},
\EQn{\label{esti:transfer-hl-II-2}
\eqref{eq:transfer-hl-II-2} \lsm & \normb{\sum_{N_1,N_2:N_1\loe N_2\ll N} \sup_{\norm{g}_{V_\De^2}=1}\absb{\int_0^T\int Nv_{N}u_{N_1}u_{N_2}g_N\dx\dd t}}_{l_N^2} \\
\lsm & \normb{\sum_{N_1,N_2:N_1\loe N_2\ll N} \sup_{\norm{g}_{V_\De^2}=1} \norm{\jb{\nabla}v_N}_{L_t^2 L_x^\I} \norm{u_{N_2}}_{L_t^\I L_x^2} \norm{u_{N_1}g_N}_{L_{t,x}^2}}_{l_N^2} \\
\lsm & \sum_{N_1\loe N_2}  \normb{\jb{\nabla}^{s+\frac12-\ep}v_N}_{l_N^2L_t^2 L_x^\I} \norm{u_{N_2}}_{L_t^\I L_x^2} N_1^{1-s+\ep} \brkb{\norm{P_{N_1}v_0}_{L_x^2} + \norm{w_{N_1}}_{U_\De^2}} \\
\lsm & \norm{v}_{\wt Y^s} \sum_{N_1\loe N_2} N_1^{-\ep}N_2^{-\ep} N_2^s\norm{u_{N_2}}_{L_t^\I L_x^2} N_1^{s} \brkb{\norm{P_{N_1}v_0}_{L_x^2} + \norm{w_{N_1}}_{U_\De^2}} \\
\lsm & \norm{v}_{\wt Y^s} \brkb{\norm{v}_{Z^s} +\norm{w}_{X^{s}}}\brkb{\norm{v}_{Z^s} +\norm{w}_{L_t^\I H_x^s}}.
}
By \eqref{esti:transfer-hl-II-1} and \eqref{esti:transfer-hl-II-2},
\EQn{\label{esti:transfer-hl-II-case2}
II \lsm & \norm{v}_{\wt Y^s} \brkb{\norm{v}_{Y^s} + \norm{N_1^sw_{ N_1}}_{l_{N_1}^2L_{t}^2 L_x^6}} \brkb{\norm{v}_{Z^s} + \norm{w}_{L_t^\I H_x^{\frac12}}}\\ 
& + \norm{v}_{\wt Y^s} \brkb{\norm{v}_{ Z^s} +\norm{w}_{X^{s}}}\brkb{\norm{v}_{Z^s} +\norm{w}_{L_t^\I H_x^s}}.
}
Then \eqref{esti:transfer-hl-I} and \eqref{esti:transfer-hl-II-case2} give the desired estimates.
\end{proof}

\subsection{Energy bound}\label{sec:energy-bound}
\begin{prop}\label{prop:energy-bound}
Let the assumptions in Proposition \ref{prop:global-derterministic} hold. Take some $T>0$ such that $w\in C([0,T];H_x^1)$. Then, there exists $N_0=N_0(A)\gg1$ with the following properties. Assume that $\wh v_0$ is supported on $\fbrk{\xi\in\R^3:|\xi|\goe \frac12 N_0}$,
\EQ{
	\norm{u_0}_{H_x^s}+\norm{v}_{\wt Y^s\cap Z^s(\R)}\loe A\text{, and }E(w_0) \loe A N_0^{2(1-s)}.
}
Then, we have
\EQn{\label{eq:energy-bound}
\sup_{t\in[0,T]}E(w(t))\loe 2A N_0^{2(1-s)}.
}
\end{prop}
\begin{proof}
Let $I=[0,T]$ and $N_0=N_0(A)$ that will be defined later. From now on, all the space-time norms are taken over $I\times\R^3$. We implement a bootstrap procedure on $I$: assume an a priori bound 
\EQn{\label{eq:bound-w-hypothesis}
\sup_{t\in I} E(w(t))\loe 2AN_0^{2(1-s)},
} 
then it suffices to prove that
\EQn{
\sup_{t\in I} E(w(t))\loe \frac32AN_0^{2(1-s)}.
}

To start with, we collect useful estimates on $I$. Now, we use the notation $C=C(A)$ for short, and  the implicit constants in ``$\lesssim$'' depend on $A$. Moreover, we take all the space-time norms over $I\times \R^3$. By interpolation 
\EQ{
\norm{v}_{L_x^3} \lsm \norm{v}_{L_x^2}^{\frac23} \norm{v}_{L_x^\I}^{\frac13} \text{, and }\norm{v}_{L_x^4} \lsm \norm{v}_{L_x^2}^{\frac12} \norm{v}_{L_x^\I}^{\frac12}
}
we have
\EQn{\label{eq:bound-v-ltinfty}
\norm{v}_{L_t^\I L_x^3\cap L_t^\I L_x^4} + \norm{\jb{\nabla}^sv}_{L_t^\I L_x^2} + \norm{v}_{L_{t,x}^4}\lsm\norm{v}_{Y^s\cap Z^s}\lsm1. 
}
By the frequency support of $v$, we have for any $0\loe l\loe s+\frac12-\ep$,
\EQn{\label{eq:bound-v-l2linfty}
	\normb{\jb{\nabla}^lv_N}_{l_N^2L_t^2 L_x^\I}\lsm N_0^{l-s-\frac12+\ep}\norm{v}_{\wt Y^s}\lsm N_0^{l-s-\frac12+\ep}\lsm 1.
} 
By the conservation of mass, we have $\norm{u(t)}_{L_x^2} = \norm{u_0}_{L_x^2}$. Then, combining \eqref{eq:bound-v-ltinfty}, we have for all $t\in[0,T]$,
\EQn{\label{eq:bound-w-l2}
	\norm{w(t)}_{L_x^2}\lsm \norm{u(t)}_{L_x^2} + \norm{v(t)}_{L_x^2}\lsm \norm{u_0}_{L_x^2} +1\lsm 1.
}
By bootstrap hypothesis \eqref{eq:bound-w-hypothesis},
\EQn{\label{eq:bound-w-h1}
\norm{w}_{L_t^\I L_x^4} \lsm N_0^{\frac12(1-s)}\text{, and }\norm{w}_{L_t^\I \dot H_x^1}\lsm N_0^{1-s}.
}
Then, by interpolation, \eqref{eq:bound-w-l2}, and \eqref{eq:bound-w-h1}, we have for any $0\loe l \loe 1$, 
\EQn{\label{eq:bound-w-ltinfty}
\norm{w}_{L_t^\I \dot H_x^l}\lsm N_0^{l(1-s)}.
}
Next, we derive various space-time bounds combining Lemmas \ref{lem:inter-mora}, \ref{lem:transfer}, and \ref{lem:transfer-2}, under the above setting.
\begin{lem}
Suppose that the assumptions in Proposition \ref{prop:energy-bound} and the estimate \eqref{eq:bound-w-hypothesis} hold for some $N_0\gg1$.  
\begin{itemize}
\item[(1)] Then the following interaction Morawetz estimate holds:
\EQn{\label{eq:bound-w-l44}
\norm{w}_{L_{t,x}^4} \lsm & N_0^{\frac{1-s}4}.
}
\item[(2)] Let $0< \ep\ll1$ be defined in Proposition \ref{prop:global-derterministic}, then
	\EQn{\label{eq:bound-w-x0}
		\norm{w}_{X^0}\lsm  N_0^{(\frac56 + \frac16\ep)(1-s)}.
	}
\item[(3)] If $0< l\loe \frac12$ and $(q,r)$ is $L_x^2$-admissible, then
\EQn{\label{eq:bound-w-xl-1/2low}
\norm{N^lw_N}_{l_N^2L_t^qL_x^r}\lsm  N_0^{(l+1)(1-s)}.
}
\item[(4)] Moreover,
\EQn{\label{eq:bound-w-xl-1/2high}
	\norm{w}_{X^1}\lsm  N_0^{3(1-s)}.
}
\end{itemize}
\end{lem}
\begin{remark}
Roughly speaking, the interaction Morawetz estimate in Lemma \ref{lem:inter-mora} yields
\EQ{
	\norm{w}_{L_{t,x}^4}^4 \lsm N_0^{1-s} + N_0^{2(1-s)}\norm{v}_{L_t^2 L_x^\I}^2.
}
Since $v$ is high-frequency truncated, we are able to cover the additional increment for the remainder in the view of \eqref{eq:bound-v-l2linfty}. This is the main reason why we implement the high-low frequency decomposition for the initial data.
\end{remark}
\begin{proof}
Note that $s>0$, by the perturbed Morawetz estimate in Lemma \ref{lem:inter-mora} and \eqref{eq:bound-w-ltinfty},
\EQ{ 
\norm{w}_{L_{t,x}^4}^4 \lsm & \norm{w}_{L_t^\I L_x^2}^2 \norm{w}_{L_t^\I \dot H_x^{\frac12}}^2 + \norm{\nabla w}_{L_t^\I L_x^2}^2 \norm{w}_{L_t^\I L_x^2}^4 \norm{v}_{L_t^2 L_x^\I}^2 + \norm{v}_{L_{t,x}^4}^4\\
\lsm & N_0^{1-s} + N_0^{2(1-s)}N_0^{-1-2s+\ep} +1\lsm N_0^{1-s}.
}

Next, we prove \eqref{eq:bound-w-x0}. Noting that $(\frac{2}{1-\ep},\frac{6}{1+2\ep})$ is $L_x^2$-admissible, by embedding $V_\De^2\hra L_t^{\frac{2}{1-\ep}}L_x^{\frac{6}{1+2\ep}}$, Lemma \ref{lem:littlewood-paley}, H\"older's inequality, \eqref{eq:bound-w-l44}, \eqref{eq:bound-w-l2}, and \eqref{eq:bound-w-h1},
	\EQ{
		\norm{w}_{X^0} \lsm & \norm{w_0}_{L_x^2} + \norm{|u|^2u}_{L_t^{\frac{2}{1+\ep}}L_x^{\frac{6}{5-2\ep}}} \\
		\lsm & 1 + \norm{u}_{L_{t,x}^4}^{2(1+\ep)}\norm{u}_{L_t^\I L_x^{\frac{6(1-2\ep)}{2-5\ep}}}^{1-2\ep} \\
		\lsm & 1 + (\norm{w}_{L_{t,x}^4} + \norm{v}_{Y^s})^{2(1+\ep)}\Big(\norm{w}_{L_t^\I L_x^2}^{\frac{1-4\ep}{3(1-2\ep)}}\norm{w}_{L_t^\I L_x^4}^{\frac{2(1-\ep)}{3(1-2\ep)}} + \norm{v}_{Z^s}\Big)^{1-2\ep} \\
		\lsm & 1 + \big(N_0^{\frac14(1-s)}+1\big)^{2(1+\ep)}\Big(N_0^{\frac{2(1-\ep)}{3(1-2\ep)}\cdot\frac12(1-s)} +1\Big)^{1-2\ep} \\
		\lsm & N_0^{(\frac56 + \frac16\ep)(1-s)}.
	}

Then, we prove \eqref{eq:bound-w-xl-1/2low}. By Lemma \ref{lem:transfer} and \eqref{eq:bound-w-l44}, for any $0< l\loe \frac12$ and $L_x^2$-admissible $(q,r)$,
\EQ{
\norm{N^{l}w_N}_{l_N^2L_t^qL_x^r}\lsm & \norm{w_0}_{\dot H_x^{l}} + \brkb{\norm{w}_{L_t^\I H_x^{l+\frac12}}+ \normb{\jb{\nabla}^{l}v_N}_{l_N^2L_t^2 L_x^{\I}}}\brkb{\norm{v}_{Y^s}^2 + \norm{w}_{L_{t,x}^4}^2}\\
\lsm & N_0^{l(1-s)} + \brkb{N_0^{(l+\frac12)(1-s)}+1}\brkb{1 + N_0^{\frac{1}{2}(1-s)}}\lsm N_0^{(l+1)(1-s)}.
}

Now, we prove the $X^1$-estimate. To this end, we first derive the interpolation
\EQn{\label{esti:bound-w-xl-1/2high-1}
\norm{w}_{L_t^6 L_x^{\frac92}} \lsm \norm{w}_{L_{t,x}^4}^{\frac23} \norm{w}_{L_t^\I \dot H_x^1}^{\frac13}\lsm N_0^{\frac12(1-s)},
}
and 
\EQn{\label{esti:bound-w-xl-1/2high-2}
\norm{\jb{\nabla}^{3\ep}w}_{L_t^{\frac{6}{1-6\ep}} L_x^{\frac{9}{2+6\ep}}} \lsm \norm{w}_{L_{t,x}^4}^{\frac23-4\ep} \norm{w}_{L_t^\I  H_x^1}^{\frac13+4\ep}\lsm N_0^{(\frac12+3\ep)(1-s)}.
}
By \eqref{eq:bound-w-x0},
\EQn{\label{esti:bound-w-xl-1/2high-3}
\norm{w}_{X^s}\lsm \norm{w}_{X^0}^{1-s}\norm{w}_{X^1}^s \lsm N_0^{(1-s)^2}\norm{w}_{X^1}^s.
}
Then, by Lemma \ref{lem:transfer-2},  \eqref{eq:bound-w-xl-1/2low}, \eqref{esti:bound-w-xl-1/2high-1}, \eqref{esti:bound-w-xl-1/2high-2}, and \eqref{esti:bound-w-xl-1/2high-3},
\EQ{
	\norm{w}_{X^1}
	\lsm & \norm{w_0}_{H_x^{l}} + \norm{w}_{L_t^\I H_x^1}\brkb{\norm{v}_{Y^s\cap Z^s} + \normb{N^{\frac12}w_{N}}_{l_{N}^2L_t^2 L_x^6}} \brkb{\norm{v}_{Y^s} + \norm{w}_{L_t^6 L_x^{\frac92}}}\\
	& + \norm{w}_{L_t^\I H_x^1}\brkb{\norm{v}_{Y^s} + \normb{N^{\frac12-3\ep}w_{N}}_{l_{N}^2L_t^2 L_x^6}} \brkb{\norm{v}_{Y^s} + \norm{\jb{\nabla}^{3\ep}w}_{L_t^{\frac{6}{1-6\ep}} L_x^{\frac{9}{2+6\ep}}}}\\
	& + \norm{v}_{\wt Y^s} \brkb{\norm{v}_{Y^s} + \norm{N^sw_{ N}}_{l_{N}^2L_{t}^2 L_x^6}} \brkb{\norm{v}_{Z^s} + \norm{w}_{L_t^\I H_x^{\frac12}}}\\
	& + \norm{v}_{\wt Y^s} \brkb{\norm{v}_{ Z^s} +\norm{w}_{X^{s}}}\brkb{\norm{v}_{Z^s} +\norm{w}_{L_t^\I H_x^s}} \\
	\lsm & N_0^{1-s} + N_0^{1-s}(1+ N_0^{\frac32(1-s)})(1+N_0^{\frac12(1-s)})\\
	& + N_0^{1-s}(1+ N_0^{(\frac32-3\ep)(1-s)})(1+N_0^{(\frac12+3\ep)(1-s)})  \\
	& + (1+ N_0^{\frac32(1-s)})(1+N_0^{\frac12(1-s)}) + (1 + N_0^{(1-s)^2}\norm{w}_{X^1}^s)(1 + N_0^{s(1-s)}) \\
	\lsm &  N_0^{3(1-s)} + N_0^{1-s}\norm{w}_{X^1}^s.
}
Noting that $3(1-s)\goe 1$, by Young's inequality,
$$
\norm{w}_{X^1}\lesssim N_0^{3(1-s)} + N_0\lesssim N_0^{3(1-s)}.
$$
Hence,  \eqref{eq:bound-w-xl-1/2high} holds. This finishes the proof of the lemma.
\end{proof}

Now, we are prepared to give the proof of Proposition \ref{prop:energy-bound}. By \eqref{eq:nls-w} and integration-by-parts, we have
\EQ{
\frac{\rm d}{\rm d t}E(w(t))
= & \im\int\De \wb w\brkb{|u|^2u-|w|^2w}\dx +\im\int |u|^2u\brkb{|u|^2u-|w|^2w}\dx \\
= & -\im\int\nabla \wb w\cdot\nabla\brkb{|u|^2u-|w|^2w}\dx+\im\int |u|^2u\brkb{|u|^2u-|w|^2w}\dx.
}
Again, we do not distinguish between $u$ and $\wb u$. Then, we have
\EQnn{
\sup_{t\in I} E(t) \lsm& E(w_0) \nonumber\\
&+ \absb{\int_I \int \nabla w\cdot \nabla v\brko{v+w}^2\dx\dd t} \label{eq:energy-bound-1}\\
&+ \absb{\int_I \int \nabla  w\cdot \nabla w v\brko{v+w}\dx\dd t}\label{eq:energy-bound-2}\\
&+\absb{\int_I\int |u|^2u\brkb{|u|^2u-|w|^2w}\dx \dd t}. \label{eq:energy-bound-3}
}

\textbf{Estimate on \eqref{eq:energy-bound-1}.} This is the main case, where we need the restriction $s>\frac{17}{40}$. We first make a frequency decomposition:
\EQnnsub{
\eqref{eq:energy-bound-1}\lsm & \sum_{N\in2^\N}\absb{\int_I \int \nabla w\cdot \nabla v_N v_{\gsm N}(v+w)\dx\dd t}\label{eq:energy-bound-1-1v-high}\\
& + \sum_{N\in2^\N}\absb{\int_I \int \nabla w_N\cdot \nabla v_N v_{\ll N}(v_{\ll N}+w_{\ll N})\dx\dd t}\label{eq:energy-bound-1-1v-low}\\
& +\sum_{N\in2^\N} \absb{\int_I \int \nabla w_N\cdot \nabla v_Nw_{\gsm N}(w+v)\dx\dd t}\label{eq:energy-bound-1-2w-high}\\
& + \sum_{N\in2^\N} \absb{\int_I \int \nabla w_N\cdot \nabla v_Nw_{\ll N}^2\dx\dd t}.\label{eq:energy-bound-1-2w-low}
} 

For \eqref{eq:energy-bound-1-1v-high}, we can directly transfer the derivative from $v_N$ to $v_{\gsm N}$. By H\"older's inequality, Lemma \ref{lem:littlewood-paley}, and \eqref{eq:bound-v-l2linfty}, 
\EQn{\label{esti:energy-bound-1-1v-high}
\eqref{eq:energy-bound-1-1v-high} \lsm & \sum_{N\lsm N_1}\absb{\int_I \int \nabla w\cdot \nabla v_N v_{N_1}(v+w)\dx\dd t} \\ 
\lsm & \sum_{N\lsm N_1} \norm{w}_{L_t^\I \dot H_x^1} \norm{\nabla v_N}_{L_t^2 L_x^\I}\norm{v_{N_1}}_{L_t^2 L_x^\I}\brkb{\norm{v}_{L_t^\I L_x^2} + \norm{w}_{L_t^\I L_x^2}} \\
\lsm & N_0^{1-s} \sum_{N\lsm N_1}\frac{N^{\frac12}}{N_1^{\frac12}}  \normb{\jb{\nabla}^{\frac12} v_N}_{L_t^2 L_x^\I}\normb{\jb{\nabla}^{\frac12}v_{N_1}}_{L_t^2 L_x^\I} \\
\lsm & N_0^{1-s} \normb{\jb{\nabla}^{\frac12} v_N}_{l_N^2L_t^2 L_x^\I}\normb{\jb{\nabla}^{\frac12}v_{N_1}}_{l_{N_1}^2L_t^2 L_x^\I}\lsm N_0^{1-s}.
}

Next, we bound \eqref{eq:energy-bound-1-1v-low}, where we use the bilinear Strichartz estimate for $\nabla w_N v_{\ll N}$ to lower down the derivative of $\nabla v_N$. 
From Lemma \ref{lem:bilinearstrichartz}, \eqref{eq:bound-w-xl-1/2high}, and \eqref{eq:bound-v-ltinfty}, for $N_1\ll N$, we have that 
\EQn{\label{esti:energy-bound-1-1v-low-bilinear}
\norm{\nabla w_N v_{N_1}}_{L_{t,x}^2}\lsm & \frac{N_1}{N^{\frac12}} N\norm{w_N}_{U_\De^2}\norm{P_{N_1}v_0}_{L_x^2}\\
\lsm & N_1^{1-s}N^{-\frac12}\norm{w}_{X^1}\norm{P_{N_1}v_0}_{H_x^s}\\
\lsm & N_1^{1-s}N^{-\frac12} N_0^{3(1-s)}.
}
Note that $\frac{17}{40}<s$ gives
\EQn{\label{esti:energy-bound-1-1v-low-li}
\norm{\nabla v_N}_{L_t^2 L_x^\I}\lsm N^{\frac{3}{40}}\normb{\jb{\nabla}^{s+\frac12-\ep} v_N}_{L_t^2 L_x^\I}\text{, and } \norm{v_{N_1}}_{L_t^2 L_x^\I}\lsm N_1^{-\frac{37}{40}}\normb{\jb{\nabla}^{s+\frac12-\ep} v_{N_1}}_{L_t^2 L_x^\I}.
} 
Then, combining H\"older's inequaltiy, \eqref{esti:energy-bound-1-1v-low-bilinear}, \eqref{esti:energy-bound-1-1v-low-li},  \eqref{eq:bound-v-ltinfty}, \eqref{eq:bound-v-l2linfty}, and \eqref{eq:bound-w-ltinfty}, it holds that
\EQn{\label{esti:energy-bound-1-1v-low}
\eqref{eq:energy-bound-1-1v-low} \lsm & \sum_{N_1\ll N}\absb{\int_I \int \nabla w_N\cdot \nabla v_N v_{N_1}(v_{\ll N}+w_{\ll N})\dx\dd t} \\
\lsm & \sum_{N_1\ll N} \norm{\nabla w_N}_{L_t^\I L_x^2}^{\frac{3}{4}} \norm{\nabla v_{N}}_{L_t^2 L_x^\I} \norm{v_{N_1}}_{L_t^2 L_x^\I}^{\frac{3}{4}} \norm{\nabla w_N v_{N_1}}_{L_{t,x}^2}^{\frac{1}{4}}\\
&\cdot\brkb{\norm{v_{\ll N}}_{L_t^\I L_x^2} + \norm{w_{\ll N}}_{L_t^\I L_x^2}} \\
\lsm & N_0^{\frac{3}{4}(1-s)} \sum_{N_1\ll N}  \norm{\nabla v_{N}}_{L_t^2 L_x^\I} \norm{v_{N_1}}_{L_t^2 L_x^\I}^{\frac{3}{4}} \brkb{N_1^{1-s}N^{-\frac12}N_0^{3(1-s)}}^{\frac{1}{4}} \\
\lsm & N_0^{\frac{3}{2}(1-s)} \sum_{N_1\ll N} N^{\frac{3}{40}} N_1^{-\frac{111}{160}} N_1^{\frac{1}{4}(1-s)}N^{-\frac18}\\
\lsm & N_0^{\frac{3}{2}(1-s)} \sum_{N_1\ll N}  N_1^{-\frac{11}{20}} N^{-\frac{1}{20}} \lsm   N_0^{-\frac{1}{2}(1-s)} N_0^{2(1-s)}.
}

Next, we deal with the term \eqref{eq:energy-bound-1-2w-high}, where we can directly transfer the derivative from $v_N$ to $w_{\gsm N}$. 
Note that by $s>\frac{17}{40}$, 
\EQn{\label{esti:energy-bound-1-2w-high-1}
\sum_{N,N_1:N_0\lsm N\lsm N_1} \frac{N^{\frac12-s+\ep}}{N_1^{\frac{1}{10}}}\lsm N_0^{-\frac{1}{100}}.
}
By interpolation, \eqref{eq:bound-w-l44}, and \eqref{eq:bound-w-xl-1/2low}, we also have
\EQn{\label{esti:energy-bound-1-2w-high-2}
\normb{N_1^{\frac{1}{10}}w_{N_1}}_{L_{t,x}^4}\lsm \norm{w}_{L_{t,x}^4}^{\frac{3}{5}} \normb{\jb{\nabla}^{\frac{1}{4}}w}_{L_{t,x}^4}^{\frac{2}{5}}\lsm N_0^{\frac{3}{4}(1-s)}.
}
Therefore, by H\"older's inequality, \eqref{eq:bound-v-l2linfty}, \eqref{esti:energy-bound-1-2w-high-1},  \eqref{esti:energy-bound-1-2w-high-2}, and \eqref{eq:bound-w-l44}, 
\EQn{\label{esti:energy-bound-1-2w-high}
\eqref{eq:energy-bound-1-2w-high}\lsm & \sum_{N,N_1\in2^\N:N\lsm N_1} \absb{\int_I \int \nabla w_N\cdot \nabla v_Nw_{N_1}(w+v)\dx\dd t}\\
\lsm & \sum_{N,N_1\in2^\N:N\lsm N_1} \frac{N^{\frac12-s+\ep}}{N_1^{\frac{1}{10}}} \norm{\nabla w_N}_{L_t^\I L_x^2} \normb{\jb{\nabla}^{s+\frac12-\ep}v_N}_{L_t^2 L_x^\I}\\
&\cdot \normb{N_1^{\frac{1}{10}}w_{N_1}}_{L_{t,x}^4}\brkb{ \norm{w}_{L_{t,x}^4} + \norm{v}_{L_{t,x}^4}}\\
\lsm & N_0^{-\frac{1}{100}}N_0^{1-s} N_0^{\frac{3}{4}(1-s)} N_0^{\frac{1}{4}(1-s)} \lsm  N_0^{-\frac{1}{100}}N_0^{2(1-s)}.
}

Now, we deal with the term \eqref{eq:energy-bound-1-2w-low}, which is the main part of the whole argument. We postponed here to illustrate the key idea. Roughly speaking, by H\"older's inequality, \eqref{eq:bound-w-ltinfty}, and \eqref{eq:bound-w-l44}, 
\EQn{\label{eq:energy-bound-1-2w-low-observation}
\absb{\int_I \int \nabla w_N\cdot \nabla v_Nw_{\ll N}^2\dx\dd t}
\lsm& \norm{\nabla w}_{L_t^\I L_x^2}\norm{\nabla v}_{L_t^2 L_x^\I}\norm{w}_{L_{t,x}^4}^2\\
\lsm& N_0^{\frac32(1-s)}\norm{\nabla v}_{L_t^2 L_x^\I}.
}
Although we lack the $\norm{\nabla v}_{L_t^2 L_x^\I}$-estimate, the bilinear Strichartz estimate for $\nabla w_N w_{\ll N}$ can be introduced to lower down the derivative of $\nabla v_N$. In the view of \eqref{eq:bound-w-xl-1/2low} and \eqref{eq:bound-w-xl-1/2high}, this procedure will cause the increase of $N_0$. This is allowed, since there is still $N_0^{\frac12(1-s)}$-gap towards the energy increment $N_0^{2(1-s)}$ in \eqref{eq:energy-bound-1-2w-low-observation}. 

Now, we give the concrete argument for the estimate of \eqref{eq:energy-bound-1-2w-low}.  By Lemma \ref{lem:bilinearstrichartz}, \eqref{eq:bound-w-xl-1/2low},  \eqref{eq:bound-w-xl-1/2high}, and \eqref{eq:bound-w-x0}, noting that $N_1\ll N$,
\EQn{\label{esti:energy-bound-1-2w-low-bilinear}
\norm{\nabla w_N w_{N_1}}_{L_{t,x}^2}\lsm& \frac{N_1}{N^{\frac12}} N\norm{w_N}_{U_\De^2}\norm{w_{N_1}}_{U_\De^2}\\
\lsm& N_1  N^{-\frac12} \norm{w_N}_{X^1} \norm{w_{N_1}}_{X^0}\\
\lsm & N_1  N^{-\frac12} N_0^{3(1-s)} N_0^{(\frac56 + \frac16\ep)(1-s)} \\
\lsm & N_1  N^{-\frac12} N_0^{(\frac{23}6 + \frac16\ep)(1-s)}
}
Therefore, by H\"older's inequality,  \eqref{esti:energy-bound-1-2w-low-bilinear} and \eqref{eq:bound-v-l2linfty},
\EQn{
\eqref{eq:energy-bound-1-2w-low}\lsm & \sum_{N_1\loe N_2\ll N} \absb{\int_I \int \nabla w_N\cdot \nabla v_Nw_{N_1}w_{N_2}\dx\dd t}\\
\lsm & \sum_{N_1\loe N_2\ll N} \norm{\nabla w_N}_{L_t^\I L_x^2}^{\frac{17}{20}+10\ep} \norm{\nabla v_N}_{L_t^2 L_x^\I} \norm{\nabla w_N w_{N_1}}_{L_{t,x}^2}^{\frac{3}{20} -10\ep}\\
&\cdot \norm{w_{N_1}}_{L_{t,x}^4}^{\frac{7}{10}+20\ep}\norm{w_{N_2}}_{L_{t,x}^4} \norm{w_{N_1}}_{L_t^\I L_x^2}^{\frac{3}{20}-10\ep}\\
\lsm & \sum_{N_1\loe N_2\ll N} N_0^{(\frac{17}{20}+10\ep)(1-s)}\> N^{\frac12-s+\ep}\|v\|_{\widetilde Y^s}\>\big[N_1  N^{-\frac12}  N_0^{(\frac{23}{6}+\frac16\ep)(1-s)}\big]^{\frac{3}{20} -10\ep}\\
&\quad \cdot N_0^{\frac14(\frac{7}{10}+20\ep)(1-s)} N_0^{\frac14(1-s)} N_1^{-(\frac{3}{20} -10\ep)}N_0^{(\frac{3}{20} -10\ep)(1-s)}\\
\lsm & N_0^{(2-20\ep)(1-s)} \sum_{N_1\loe N_2\ll N}  N^{\frac{17}{40}-s+6\ep}.
}
By the choice of $\ep$ in \eqref{defn:ep-gwp},  $\frac{17}{40}-s+6\ep<0$, then we obtain that 
\EQn{\label{esti:energy-bound-1-2w-low}
\eqref{eq:energy-bound-1-2w-low}
\lsm &  N_0^{-20\ep(1-s)} N_0^{2(1-s)}.
}

Combining  \eqref{esti:energy-bound-1-1v-high}, \eqref{esti:energy-bound-1-1v-low}, \eqref{esti:energy-bound-1-2w-high}, and  \eqref{esti:energy-bound-1-2w-low}, we have
\EQn{\label{esti:energy-bound-1}
	\eqref{eq:energy-bound-1}\lsm & \eqref{eq:energy-bound-1-1v-high} +\eqref{eq:energy-bound-1-1v-low}+\eqref{eq:energy-bound-1-2w-high} + \eqref{eq:energy-bound-1-2w-low}\\
	\lsm& \brkb{N_0^{-(1-s)} +  N_0^{-\frac{1}{2}(1-s)} + N_0^{-\frac{1}{100}} + N_0^{-20\ep(1-s)}} N_0^{2(1-s)}\\
	\lsm &  N_0^{-20\ep(1-s)} N_0^{2(1-s)}.
} 

\textbf{Estimate on \eqref{eq:energy-bound-2}.}  The proof for \eqref{eq:energy-bound-2} is easier, since there is no derivative acting on $v$. However, the integration contains two $\nabla w$ terms, which already leads to the increment of $N_0^{2(1-s)}$. Therefore, we need to cover the additional $N_0$. By H\"older's inequality,
\EQn{\label{esti:energy-bound-2-1}
	\eqref{eq:energy-bound-2}
	\lsm & \norm{\nabla w}_{L_t^\I L_x^2}^2\norm{v}_{L_t^2 L_x^\I}\norm{u}_{L_t^2 L_x^\I} \\
	\lsm & \norm{\nabla w}_{L_t^\I L_x^2}^2\norm{v}_{L_t^2 L_x^\I}\brkb{\norm{w}_{L_t^2 L_x^\I} + \norm{v}_{L_t^2 L_x^\I}}.
}
Due to the failure of endpoint Strichartz estimate for $L_t^2 L_x^\I$, we use the interpolation, Bernstein's inequality, Lemma \ref{lem:strichartz}, \eqref{eq:bound-w-xl-1/2low}, and  \eqref{eq:bound-w-xl-1/2high}
\EQn{\label{esti:energy-bound-2-2}
\norm{w}_{L_t^2 L_x^\I} \lsm & \norm{N^\ep P_Nw}_{l_N^2L_t^2 L_x^\I} \\
\lsm & \norm{P_Nw}_{l_N^2L_t^2 L_x^\I}^{1-2\ep} \normb{N^{\frac12}P_Nw}_{l_N^2L_t^2 L_x^\I}^{2\ep} \\
\lsm & \normb{N^{\frac12}P_Nw}_{l_N^2L_t^2 L_x^6}^{1-2\ep} \normb{NP_Nw}_{l_N^2U_\De^2}^{2\ep} \\
\lsm & \normb{N^{\frac12}P_Nw}_{l_N^2L_t^2 L_x^6}^{1-2\ep} \norm{w}_{X^1}^{2\ep} \\
\lsm & N_0^{(\frac32+3\ep)(1-s)}.
}
Since $s>\frac{17}{40}$ implies $1-\frac52s+(4-3s)\ep<1-\frac52s+3\ep<-\frac{1}{100}$, by \eqref{esti:energy-bound-2-1}, \eqref{esti:energy-bound-2-2}, and \eqref{eq:bound-v-l2linfty},
\EQn{\label{esti:energy-bound-2}
\eqref{eq:energy-bound-2}
\lsm & \norm{\nabla w}_{L_t^\I L_x^2}^2\norm{v}_{L_t^2 L_x^\I}\brkb{\norm{w}_{L_t^2 L_x^\I} + \norm{v}_{L_t^2 L_x^\I}} \\ 
\lsm & N_0^{2(1-s)} N_0^{-s-\frac12+\ep}\brkb{N_0^{(\frac32+3\ep)(1-s)} + N_0^{-s-\frac12+\ep}} \\
\lsm & N_0^{1-\frac52s+(4-3s)\ep} N_0^{2(1-s)} \lsm N_0^{-\frac1{100}}N_0^{2(1-s)}.
}

\textbf{Estimate on \eqref{eq:energy-bound-3}.} This is a simple case, where no derivative appears. By H\"older's inequality, \eqref{eq:bound-w-l44}, and \eqref{eq:bound-w-ltinfty},
\EQ{\label{esti:energy-bound-3}
\eqref{eq:energy-bound-3}\lsm & \absb{\int_I\int |u|^2u\brkb{|u|^2u-|w|^2w}\dx \dd t } \\
\lsm & \norm{v}_{L_t^2 L_x^\I} \brkb{\norm{v}_{L_{t,x}^4}^2 + \norm{w}_{L_{t,x}^4}^2} \brkb{\norm{v}_{L_t^\I L_x^6}^3 + \norm{w}_{L_t^\I L_x^6}^3} \\
\lsm & N_0^{-s-\frac12+\ep} N_0^{\frac12(1-s)} N_0^{3(1-s)} \lsm N_0^{1-\frac52s+\ep} N_0^{2(1-s)}\lsm N_0^{-\frac{1}{100}} N_0^{2(1-s)}.
} 

Then, by choosing $N_0=N_0(A)$ suitably large, and combining \eqref{esti:energy-bound-1},  \eqref{esti:energy-bound-2}, and \eqref{esti:energy-bound-3}, we have
\EQn{
	\sup_{t\in I} E(t) \loe& E(w_0) + \eqref{eq:energy-bound-1} + \eqref{eq:energy-bound-2}+ \eqref{eq:energy-bound-3}\\
	\loe & AN_0^{2(1-s)} + C(A)\cdot\brkb{ N_0^{-20\ep(1-s)} + N_0^{-\frac{1}{100}}} N_0^{2(1-s)} \\
	\loe &  \frac{3}{2} AN_0^{2(1-s)}.
}
Then, by the standard bootstrap argument, we finish the proof of \eqref{eq:energy-bound}.
\end{proof}

\subsection{Proof of Proposition \ref{prop:global-derterministic}}\label{sec:global-scattering}
We first prove the global well-posedness. Since $v\in Y^s\cap Z^s(\R)$ and $w_0\in H_x^1$, by Proposition \ref{prop:local-H1}, there exists $T_1$ depending on $\norm{w_0}_{H_x^1}$ and $\norm{v}_{Y^s(\R)\cap Z^s(\R)}$, such that $w\in C([0,T_1];H_x^1)$ solves \eqref{eq:nls-w}. By Proposition \ref{prop:energy-bound}, we have
\EQ{
E(w(T_1))\loe \sup_{t\in [0,T_1]} E(w(t)) \loe 2AN_0^{2(1-s)}.
}
Then, we have $\norm{w(T_1)}_{H_x^1}^2\loe 2AN_0^{2(1-s)}$, and can apply Proposition \ref{prop:local-H1} again starting from $T_1$. Since the energy bound in \eqref{eq:energy-bound} does not rely on $T$, we can extend the solution on $\R$ by induction, and get
\EQn{\label{eq:energy-bound-global}
	\sup_{t\in\R}E(w(t))\loe 2A N_0^{2(1-s)}.
}

Next, we prove the scattering statement. Since the global well-posedness already holds, we do not care about the explicit expression of $A$ and $N_0$. We only consider the forward-in-time case, and it suffices to prove that
\EQn{\label{eq:scattering-space-time}
\lim_{T\ra+\I}\normb{\int_T^{+\I} e^{-it\De}(|u|^2u)\dd t}_{H_x^1} = 0.
}
First note that by \eqref{eq:energy-bound-global}, we have Lemmas \ref{lem:inter-mora}, \ref{lem:transfer}, and \ref{lem:transfer-2} hold on $[0,\I)$. Then, 
\EQ{
\norm{w}_{X^1([0,+\I))}\loe C(A, N_0).
}
In particular, 
\EQ{
\lim_{T\ra+\I} \norm{w}_{L_t^4L_x^6([T,+\I)\times \R^3)} = 0.
}
Moreover, we also have
\EQ{
\lim_{T\ra+\I} \norm{v}_{\wt Y^s([T,+\I))} = 0.
}
Now, all the space-time norms are taken over $[T,+\I)\times \R^3$.
We split
\EQnnsub{
\text{L.H.S. of }\eqref{eq:scattering-space-time} \lsm & \normb{ \int_T^{+\I} e^{-it\De}(\jb{\nabla} wu^2)\dd t}_{L_x^2} \label{eq:scattering-space-time-nablaw}\\
& + \normb{ \int_T^{+\I} e^{-it\De}(\jb{\nabla} v u^2)\dd t}_{L_x^2}.\label{eq:scattering-space-time-nablav}
}
The proof for the first term \eqref{eq:scattering-space-time-nablaw} is easy. By Lemma \ref{lem:strichartz} and H\"older's inequality, 
\EQ{
\eqref{eq:scattering-space-time-nablaw} \lsm & \norm{\jb{\nabla} wu^2}_{L_t^2 L_x^{\frac65}}\\
\lsm & \norm{\jb{\nabla} w}_{L_t^\I L_x^2}\brkb{\norm{w}_{L_t^4 L_x^6}^2 + \norm{v}_{L_t^4 L_x^6}^2}\\
\lsm & C(A,N_0)\brkb{\norm{w}_{L_t^4 L_x^6([T,+\I))}^2 + \norm{v}_{\wt Y^s([T,+\I))}^2}.
}
Then, we have
\EQn{\label{esti:scattering-space-time-nablaw}
\lim_{T\ra+\I} \eqref{eq:scattering-space-time-nablaw} = 0.
}

Next, we deal with the term \eqref{eq:scattering-space-time-nablav}. By frequency decomposition,
\EQnnsub{
\eqref{eq:scattering-space-time-nablav}\lsm & \sum_{N\in2^\N}\normb{ \int_T^{+\I} e^{-it\De}(\jb{\nabla} v_N u_{\gsm N}u)\dd t}_{L_x^2}\label{eq:scattering-space-time-nablav-high}\\
& + \sum_{N\in2^\N} \normb{ \int_T^{+\I} e^{-it\De}(\jb{\nabla} v_N u_{\ll N}^2)\dd t}_{L_x^2}.\label{eq:scattering-space-time-nablav-low}
}

By H\"older's inequality and Lemma \ref{lem:schurtest}, we have
\EQ{
\eqref{eq:scattering-space-time-nablav-high} \loe & C \sum_{N\lsm N_1}\norm{\jb{\nabla} v_N u_{N_1}u}_{L_t^1 L_x^2} \\
\loe & C \sum_{N\lsm N_1}\frac{N^{\frac12}}{N_1^{\frac12}}\normb{\jb{\nabla}^{\frac12} v_N}_{L_t^2 L_x^\I} \normb{\jb{\nabla}^{\frac12}v_{N_1}}_{L_t^2 L_x^\I} \norm{u}_{L_t^\I L_x^2} \\
& + C \sum_{N\lsm N_1}\frac{N^{\frac12}}{N_1^{\frac12}}\normb{\jb{\nabla}^{\frac12} v_N}_{L_t^2 L_x^\I} \normb{\jb{\nabla}^{\frac12}w_{N_1}}_{L_{t,x}^4}\norm{u}_{L_{t,x}^4} \\
\loe & C \normb{\jb{\nabla}^{\frac12} v_N}_{l_N^2 L_t^2 L_x^\I} \normb{\jb{\nabla}^{\frac12}v_{N_1}}_{l_{N_1}^2 L_t^2 L_x^\I}\norm{u}_{ L_t^\I L_x^2}\\
& + C \normb{\jb{\nabla}^{\frac12} v_N}_{l_N^2 L_t^2 L_x^\I} \normb{\jb{\nabla}^{\frac12}w_{N_1}}_{l_{N_1}^2 L_{t,x}^4}\norm{u}_{ L_{t,x}^4} \\
\loe & C \norm{v}_{\wt Y^s([T,+\I))}^2 + C \norm{v}_{\wt Y^s([T,+\I))} \normb{w}_{X^1} \brkb{\norm{v}_{Y^s} + \norm{w}_{X^1}} \\
\loe & C \norm{v}_{\wt Y^s([T,+\I))}^2 + C(A,N_0) \norm{v}_{\wt Y^s([T,+\I))}.
}
Then, 
\EQn{\label{esti:scattering-space-time-nablav-high}
\lim_{T\ra+\I} \eqref{eq:scattering-space-time-nablav-high} = 0.
}

Finally, we estimate  \eqref{eq:scattering-space-time-nablav-low}, where we need to exploit the duality structure as in the proof of Proposition \ref{prop:local-H1}. In this case, it is unnecessary to invoke the $U^p$-$V^p$ method as before for two reasons: first, we are considering the dual operator of $e^{is\De}$; second, we have estimate for $\normb{\jb{\nabla}^{\frac56+\ep}v}_{L_t^2 L_x^\I}$ under the radial assumption. We can simply use the duality representation of the $L_x^2$-norm:
\EQn{\label{esti:scattering-space-time-nablav-low-1}
\eqref{eq:scattering-space-time-nablav-low}\loe & C \sum_{N_1\loe N_2\ll N} \normb{ \int_T^{+\I} e^{-it\De}(\jb{\nabla} v_N u_{N_1} u_{N_2})\dd t}_{L_x^2}\\
\loe &C \sum_{N_1\loe N_2\ll N} \sup_{\norm{g}_{L_x^2}=1} \int_T^{+\I} \jb{g,e^{-it\De}(\jb{\nabla} v_N u_{N_1} u_{N_2})}\dd t\\
\loe & C\sum_{N_1\loe N_2\ll N} \sup_{\norm{g}_{L_x^2}=1} \int_T^{+\I} \int (e^{it\De}g_{\sim N})\jb{\nabla} v_N u_{N_1} u_{N_2} \dx\dd t.
}
By Lemma \ref{lem:bilinearstrichartz}, for $N_1\ll N$,
\EQn{\label{esti:scattering-space-time-nablav-low-2-bilinear}
\norm{(e^{it\De}g_{\sim N}) u_{N_1}}_{L_{t,x}^2}\lsm & \frac{N_1}{N^{\frac12}} \norm{g_{\sim N}}_{L_x^2} \brkb{\norm{P_{N_1}v_0}_{L_x^2} + \norm{w_{N_1}}_{U_\De^2}}\\
\lsm & \frac{N_1^{\frac12}}{N^{\frac13}} \brkb{N_1^{\frac13}\norm{P_{N_1}v_0}_{L_x^2} + N_1^{\frac13}\norm{w_{N_1}}_{U_\De^2}} 
\lsm  \frac{N_1^{\frac13}}{N^{\frac16}} C(A, N_0).
}
By \eqref{esti:scattering-space-time-nablav-low-2-bilinear} and H\"older's inequality, for $N_1\loe N_2\ll N$,
\EQn{\label{esti:scattering-space-time-nablav-low-2}
& \int_T^{+\I}\int (e^{it\De}g_{\sim N})\jb{\nabla} v_N u_{N_1} u_{N_2} \dx\dd t\\
\loe & C \norm{(e^{it\De}g_{\sim N}) u_{N_1}}_{L_{t,x}^2}\norm{\jb{\nabla} v_N}_{L_t^2 L_x^\I}\norm{u_{N_2}}_{L_t^\I L_x^2} \\
\loe & C(A, N_0) \frac{N_1^{\frac13}}{N^{\frac16}} \norm{\jb{\nabla} v_N}_{L_t^2 L_x^\I}\norm{u_{N_2}}_{L_t^\I L_x^2}\\
\loe & C(A, N_0) \frac{N_1^{\frac13}}{N_2^{\frac13}} N^{-\ep} \normb{\jb{\nabla}^{\frac56+\ep} v_N}_{L_t^2 L_x^\I}\normb{\jb{\nabla}^{\frac13}u_{N_2}}_{L_t^\I L_x^2} \\
\loe & C(A, N_0) \frac{N_1^{\frac13}}{N_2^{\frac13}} N^{-\ep} \norm{v}_{\wt Y^s([T,+\I))}.
}
Then, by \eqref{esti:scattering-space-time-nablav-low-1} and \eqref{esti:scattering-space-time-nablav-low-2},
\EQ{
\eqref{eq:scattering-space-time-nablav-low}\loe & C \sum_{N_1\loe N_2\ll N} \sup_{\norm{g}_{L_x^2}=1} \int_T^{+\I}\int (e^{it\De}g_{\sim N})\jb{\nabla} v_N u_{N_1} u_{N_2} \dx\dd t\\
\loe & C(A, N_0) \sum_{N_1\loe N_2\ll N} \frac{N_1^{\frac13}}{N_2^{\frac13}} N^{-\ep} \norm{v}_{\wt Y^s([T,+\I))} \lsm C(A, N_0) \norm{v}_{\wt Y^s([T,+\I))}.
}
Then,
\EQn{\label{esti:scattering-space-time-nablav}
\lim_{T\ra+\I} \eqref{eq:scattering-space-time-nablav-low} = 0.
}
It follows from
\eqref{esti:scattering-space-time-nablaw}, \eqref{esti:scattering-space-time-nablav-high}, and  \eqref{esti:scattering-space-time-nablav} that \eqref{eq:scattering-space-time} holds. This finishes the proof of Proposition \ref{prop:global-derterministic}.

\bigskip
\section*{Acknowledgment}
J. Shen and Y. Wu are partially supported by NSFC 12171356 and 11771325.
A. Soffer is partially supported by the Simons' Foundation (No. 395767). 

The authors would like to thank the editor and anonymous referees for valuable comments and suggestions.

\end{document}